\documentclass[12pt,a4paper]{article}

\usepackage[export]{adjustbox}

\usepackage{bbm}
\usepackage{amsmath}
\usepackage{amstext}
\usepackage{amsfonts}
\usepackage{subfigure}
\usepackage{multirow}
\usepackage{microtype}
\usepackage{graphicx,caption}
\usepackage{float}
\usepackage{color}
\usepackage{epstopdf}
\usepackage[english]{babel}
\usepackage{amsthm}
\usepackage{amsmath}
\usepackage{mathtext}
\usepackage{amstext}
\usepackage{amsfonts}
\usepackage{graphicx}
\usepackage{caption}
\usepackage{wrapfig}
\usepackage{amssymb}
\usepackage{thmtools}
\usepackage{mathtools}
\usepackage{breqn}
\usepackage{indentfirst}
\usepackage{longtable}
\usepackage{changepage}
\usepackage{multirow}
\usepackage{url}
\usepackage{subfig}
\usepackage{sidecap}
\usepackage{float}
\usepackage{color}
\usepackage{booktabs}

\usepackage{epsfig}
\usepackage{calrsfs}
\usepackage{relsize}
\usepackage{amssymb}

\usepackage[shortlabels]{enumitem}

\renewcommand{\restriction}{\mathord{\upharpoonright}}

\newcommand{\vertiii}[1]{{\left\vert\kern-0.25ex\left\vert\kern-0.25ex\left\vert #1\right\vert\kern-0.25ex\right\vert\kern-0.25ex\right\vert}}

\newtheorem{proposition}{Proposition}[section]
\newtheorem{lemma}{Lemma}[section]

 \newtheorem{remark}{Remark}[section]

\newtheorem {definition}{Definition}[section]
\newcommand\abs[1]{\left|#1\right|}
\let\div\undefined
\DeclareMathOperator{\div}{div}

\DeclareMathAlphabet{\pazocal}{OMS}{zplm}{m}{n}

\numberwithin{equation}{section}

\DeclareMathOperator{\arsinh}{arsinh}

\begin{document}
\title{Reliable numerical solution of a class of nonlinear elliptic problems generated by the Poisson-Boltzmann equation}
\author{J. Kraus\footnote{University of Duisburg-Essen, Germany}, S. Nakov\footnote{RICAM, Austrian Academy of Sciences},
S. Repin\footnote{University of Jyv\"askyl\"a, Finland; V. A.~Steklov Institute of Mathematics in St.~Petersburg, Russia}}
\date{}
\maketitle
\thispagestyle{empty}
\begin{abstract}
\noindent 
We consider a class of nonlinear elliptic problems associated with models in biophysics, which are described by the Poisson-Boltzmann equation (PBE).
We prove mathematical correctness of the problem, study a suitable class of approximations,
and deduce guaranteed and fully computable bounds of approximation errors. The latter goal is achieved by means of the approach suggested
in~\cite{Repin_2000} for convex variational problems. Moreover, we establish the error identity, which defines the error
measure natural for the considered class of problems and show that it yields computable majorants and minorants of
the global error as well as indicators of local errors that provide efficient adaptation of meshes. Theoretical results are
confirmed by a collection of numerical tests that includes problems on $2D$ and $3D$ Lipschitz domains.

 
\end{abstract}
\quad\quad\quad\textbf{Keywords}: Poisson-Boltzmann equation,
semilinear partial differential equations, existence and uniqueness of solutions, convergence of finite element approximations,  a priori error estimates, guaranteed and efficient a posteriori error bounds, error indicators, adaptive mesh refinement,  qualified and unqualified convergence.



\setcounter{page}{1}
\pagenumbering{arabic}
\section{Introduction}
\subsection {Classical statement of the problem}

Let $\Omega\subset \mathbb R^d,d=2,3$ be a bounded domain with Lipschitz boundary $\partial\Omega$.
Henceforth we assume that  $\Omega$ contains an interior subdomain $\Omega_1$ with Lipschitz boundary
$\Gamma$. In general, $\Omega_1$ may consist of several disconnected parts (in this case all of them are
assumed to have Lipschitz continuous boundaries). We consider the following class of nonlinear elliptic equations
motivated by the Poisson-Boltzmann equation (PBE), which is widely used for computation of electrostatic
interactions in a system of biomolecules in ionic
solution~\cite{Sharp_Honig_1990, Fogolari_Brigo_Molinari_2002, Fogolari_Zuccato_Esposito_Viglino_1999}: 
\begin{subequations}\label{PBE_special_form_regular_nonlinear_part}
\begin{eqnarray}
-\nabla\cdot\left(\epsilon \nabla u\right)+k^2 \sinh(u+w)&=&l \quad\text{in } \Omega_1\cup \Omega_2, \label{PBE_special_form_regular_nonlinear_part_1}\\
	\left[u\right]_\Gamma&=&0,\label{PBE_special_form_regular_nonlinear_part_2}\\
	 \left[\epsilon\frac{\partial u}{\partial n}\right]_\Gamma&=&0, \label{PBE_special_form_regular_nonlinear_part_3}\\
	 u&=&0,\quad \text{on } \partial \Omega, \label{PBE_special_form_regular_nonlinear_part_4}
	  \end{eqnarray}
	 \end{subequations}
where $\Omega_2:=\Omega\setminus\left(\Omega_1\cup\Gamma\right)$, the coefficients
$\epsilon,k\in L^\infty(\Omega)$, $\epsilon_{\max}\ge\epsilon\ge \epsilon_{\min}>0$, $w$
is measurable, and $l\in L^2(\Omega)$. Typically, in biophysical applications, $\Omega_1$
is occupied by one or more macromolecules and $\Omega_2$ is occupied by a solution of
water and moving ions. The coefficients $\epsilon$ and $k$ represent the dielectric constant
and the modified Debye-Huckel parameter and $u$ is the dimensionless electrostatic potential.
Concerning the given functions $k$ and $w$, we can identify three main cases:
\begin{enumerate}[(a)]
\item \label{a} $k_{\max}\ge k(x)\ge k_{\min}>0$ in $\Omega$ and $w\in L^\infty(\Omega)$
\item \label{b} $k(x)\equiv 0$ in $\Omega_1$, $k_{\max}\ge k(x)\ge k_{\min}>0$ in $\Omega_2$ and $w\in L^\infty(\Omega_2)$
\item \label{c} $k(x)\equiv 0$ in $\Omega_2$, $k_{\max}\ge k(x)\ge k_{\min}>0$ in $\Omega_1$ and $w\in L^\infty(\Omega_1)$
\end{enumerate}
Throughout the paper, the major attention is paid to the case \ref{b}, which arrises when solving the PBE and is the
most interesting from the practical point of view. The cases \ref{a} and \ref{c} can be studied analogously
(with some rather obvious modifications). The case with nonhomogeneous Dirichlet boundary condition $u=g$
on $\partial \Omega$
can also be treated in this framework provided that the boundary
condition is defined as the trace of a function $g$ such that
 $g\in H^1(\Omega)\cap L^\infty(\Omega)$ and $\nabla g\in L^s(\Omega)$ with $s>\max\{2,d\}$.
 
The reliable and efficient solution of the nonlinear Poisson Boltzmann equation (PBE) for complex
geometries of the interior domain $\Omega_1$ (with Lipschitz boundary) and piecewise constant
dielectrics has important applications in biophysics and biochemistry, e.g., in modeling the effects
of water and ion screening on the potentials in and around soluble proteins, nucleic acids,
membranes, and polar molecules and ions, see \cite{Sharp_Honig_1990} and the references therein.
Although the solution of the linearized PBE, as in the linear Deybe-Huckel theory, often yields accurate
approximations \cite{Knapp_Schoberl_2014} certain mathematical models are valid only when based
on the nonlinear PBE.

Over the recent years adaptive finite element methods have proved to be an adequate technique
in the numerical solution of elliptic problems with local features due to point sources, heterogeneous
coefficients or nonsmooth boundaries or interfaces, see e.g., \cite{CC_Praetorius_2014, Praetorius_Zee_2016}
and also successfully used to solve the nonlinear PBE~\cite{Chen2006b,Holst2012}.
Adaptivity heavily relies on reliable and efficient error indicators that are typically developed in the
framework of a posteriori error control. While the theory of a posteriori error estimates for linear
elliptic partial differential equations is already well established and understood,
it is far less developed for nonlinear problems. A posteriori error analysis
based on functional estimates has already been successfully applied to variational nonlinear problems including
obstacle problems in~\cite{Repin_noninear_var_problems_2012, Repin_Valdman_2017}.
The accuracy verification approach taken in this work is also based on arguments
that are commonly used in duality theory and convex analysis and can be found, e.g.
in~\cite{Ekeland_Temam, Repin}.
Another important issue in the efficient solution of the nonlinear PBE is related to the fast solution
of the systems of nonlinear and finally linear algebraic equations that arise from--for instance
adaptive finite element--discretization. Multigrid methods may provide optimal or nearly optimal algorithms
in terms of computational complexity to perform this task, see, e.g., \cite{Oberoi_Allewell_1993},
a topic which is beyond the scope of this paper.

The main questions studied in the paper are related to the well posedness of Problem~\eqref{PBE_special_form_regular_nonlinear_part} and a posteriori error estimation of its numerical solution. We use a suitable weak formulation (Definition \ref{weak_formulation_for_the_class_of_nonlinear_problems}), where the nonlinearity does not satisfy any polynomial growth condition and consequently it does not induce a bounded mapping from $H_0^1(\Omega)$ to its dual $H^{-1}(\Omega)$. For this (more general) weak formulation we can guarantee the existence of a solution and prove its uniqueness using a result of Brezis and Browder \cite{Brezis_Browder_1978_One_Property_of_Sobolev_Spaces}. Additionally, in Propositio~\ref{Proposition_1}, we show that the solution is bounded (here \cite{Brezis_Browder_1978_One_Property_of_Sobolev_Spaces} 
is used again together with special test functions suggested in Stampacchia \cite{Stampacchia_1965, Kinderlehrer_Stampacchia}). Boundedness of the solution is
important and later used in the derivation of functional a posteriori error estimates. By applying the general framework from \cite{Repin_2000} and \cite{Repin} we derive guaranteed and computable bounds of the difference between the exact
solution and any function from the respective energy class in terms of the energy and combined energy norms (equations \eqref{Explicit_form_of_Error_Estimate_with_Forcing_Functionals} and \eqref{Upper_Bound_CEN_Error}). Moreover, we obtain an error equality \eqref{Explicit_form_of_Functional_error_equality} with respect to a certain measure for the error which is the sum of the usual combined energy norm $\vertiii{\nabla(v-u)}^2+\vertiii{y^*-p^*}_*^2$ and a nonlinear measure. In the case of a linear elliptic equation of the form $-\div(\epsilon\nabla u)+u=l$, this nonlinear measure reduces to $\|v-u\|_{L^2(\Omega)}^2+\|\div(y^*-p^*)\|_{L^2(\Omega)}^2$, where $v$ and $y^*$ are approximations to the exact solution $u$ and the exact flux $p^*=\epsilon\nabla u$. One advantage of the presented error estimate is that it is valid for any conforming approximations of $u$ and $\epsilon\nabla u$ and that it does not rely on Galerkin orthogonality or properties specific to the used numerical method. Another advantage is that only the mathematical structure of the problem is exploited and therefore no mesh dependent constants are present in the estimate. Majorants of the error not only give guaranteed bounds of  global (energy) error norms but also generate efficient error indicators (cf. \eqref{PBE_special_form_regular_nonlinear_part_1}, Figures \ref{Example_run26_2D_Function_f_200_Iso_Lines0000} and \ref{run26_2D_Th_Adapted_With_Functional_Error_Indicator0011}). Also, we derive a simple, but efficient lower bound for the error in the combined energy norm. Using only the error majorant, we obtain an analog of Cea's lemma which formes a basis  for  the a priori convergence analysis of finite element 
approximations for this class of semilinear problems. Finally, we present three numerical examples that verify the accuracy of error majorants and minorants and confirm efficiency of the error indicator in mesh adaptive procedures.

The outline of the paper is as follows. In Section~2, we recall some facts from the duality theory and general
a posteriori error estimation method for convex variational problems. 
Next, we briefly discuss correctness of Problem \eqref{PBE_special_form_regular_nonlinear_part} and prove
an a priori $L^\infty(\Omega)$ estimate for the solution $u$. In Section~3, we apply the abstract framework from
Section~2 and derive explicit forms of all the respective terms.
A special attention is paid to the general error identity that defines
a combined error measure natural for the considered class of problems.  At the end of Section~3, we prove
convergence of the conforming finite element method based on $P_1$ Lagrange elements. In Section~4, we
consider numerical examples in $2D$ and $3D$ and compare the results with solutions obtained by adaptive
mesh refinements based on different indicators. The last section includes a summary of the results and comments
on possible generalizations of the method to a wider class of nonlinear problems.

\section{Abstract framework}\label{Section_Abstract_Framework}
First, we briefly recall some results from the duality theory (\cite{Repin,Ekeland_Temam}). Consider a class of variational problems having the following common form:
\begin{align}\label{general_variational_problem}
&\text{Find } u\in V \text{ such that}\nonumber\\
&(P)\quad J(u)=\inf\limits_{v\in V}{J(v)}, \text{ where } J(v)=G(\Lambda v)+F(v).
\end{align}
Here, $V$, $Y$ are reflexive Banach spaces with the norms $\|.\|_V$ and $\|.\|_Y$, respectively, $F:V\to\mathbb {\overline R}$, $G:Y\rightarrow \mathbb {\overline R}$ are convex and proper functionals, and $\Lambda: V\to Y$ is a bounded linear operator. By $0_V$ we denote the zero element in $V$. It is assumed that $J$ is coercive and lower semicontinuous. In this case, Problem $(P)$ has a solution $u$, which is unique if $J$ is strictly convex.
 
The  spaces topologically dual to $V$ and $Y$ 
are denoted by $V^*$ and $Y^*$, respectively.
They are endowed  with the norms $\|.\|_{V^*}$ and $\|.\|_{Y^*}$. 
Henceforth, $\langle v^*,v\rangle$ denotes the duality product of $v^*\in V^*$ and $v\in V$. Analogously, $(y^*,y)$ is the duality product of $y^*\in Y^*$ and $y\in Y$.
 $\Lambda^*: Y^*\to V^*$ is the operator adjoint to $\Lambda$. It is defined by the relation
 \begin{align*}
\langle \Lambda^*y^*,v\rangle= (y^*,\Lambda v),\,\forall v\in V,\,\forall y^*\in Y^*.
 \end{align*}
We recall that a convex functional $J: V\rightarrow \mathbb{\overline  R}$ is called {\it uniformly convex} in a
ball $B(0_V,\delta)$ (see, e.g. \cite{Repin}) if there exists
a nonnegative proper and lower semicontinuous functional
$\Upsilon_\delta$, $\Upsilon_\delta:V\rightarrow \mathbb {\overline R}$ with $\Upsilon_\delta(v)=0$ iff $ v=0_V$
such that for all $v_1,v_2\in B(0_V,\delta)$ the following inequality holds:
\begin{align}\label{Defining_inequality_for_uniform_convexity}
J\left(\frac{v_1+v_2}{2}\right)+\Upsilon_\delta(v_1-v_2)\leq \frac{1}{2}\left(J(v_1)+J(v_2)\right).
\end{align}
The functional $\Upsilon_\delta$ enforces the standard midpoint convexity inequality and therefore is called
a {\it forcing} functional. 
\begin{remark}\label{Remark_forcing_functional}
In what follows, we will use the term forcing functional under slightly weaker conditions than usual hereby dropping
the requirement that $\Upsilon(v)=0$ implies that $v=0$.
\end{remark}

The functional $J^*: V^*\rightarrow \mathbb {\overline R}$ defined by the relation
\begin{align*}
J^*(v^*):= \sup\limits_{v\in V}{\{\langle v^*,v \rangle -J(v)\}}
\end{align*}
is called  {\it dual} (or Fenchel conjugate) conjugate to $J$ (see, e.g. \cite{Ekeland_Temam} ). In accordance with the general duality theory of the calculus of variations, the primal Problem \eqref{general_variational_problem} has a dual counterpart:
\begin{align}\label{variational_problem_for_pStar}
\begin{aligned}
&\text{Find }p^*\in Y^*\text{ such that}\\
&(P^*) \quad I^*(p^*)=\sup\limits_{y^*\in Y^*}{I^*(y^*)}, \text{ where } I^*(y^*):=-G^*(y^*)-F^*(-\Lambda^*y^*),
\end{aligned}
\end{align}
where $G^*$ and $F^*$ are the functionals conjugate to  $G$ and $F$, respectively. The problems $(P)$ and $(P^*)$ 
are generated by the Lagrangian  $L:V\times Y^*\to \mathbb {\overline R}$ defined by the relation
\begin{align*}
L(v,y^*)= (y^*,\Lambda v) -G^*(y^*)+F(v).
\end{align*}
If we additionally assume that $G^*$ is coercive and that $F(0_V)$ is finite, then it is well known that problems $(P)$ and $(P^*)$ have unique solutions $u\in V$ and $p^*\in Y^*$ and that strong duality relations hold (see \cite{Repin}, or Proposition 2.3, Remark 2.3, and Proposition 1.2 from chapter VI in \cite{Ekeland_Temam}):
\begin{align}\label{strong_duality_holds}
J(u)=\inf\limits_{v\in V}{J(v)}=\inf\limits_{v\in V}\sup\limits_{y^*\in Y^*}{L(v,y^*)}=\sup\limits_{y^*\in Y^*}\inf\limits_{v\in V}{L(v,y^*)}=\sup\limits_{y^*\in Y^*}{I^*(y^*)}=I^*(p^*).
\end{align}
Furthermore, the pair $(u,p^*)$ is a saddle point for the Lagrangian $L$, i.e.,
\begin{align}\label{saddle_point_inequalities}
L(u,y^*)\leq L(u,p^*)\leq L(v,p^*),\,\forall v\in V,\,\forall y^*\in Y^*
\end{align} 
and $u$ and $p^*$ satisfy the relations 
\begin{align}\label{relations_between_solutions_of_primal_dual_problems}
\Lambda u&\in \partial G^*(p^*),\qquad p^*\in \partial G(\Lambda u).
\end{align}

Now let $\Upsilon_G,\Upsilon_{G^*},\Upsilon_F,\Upsilon_{F^*}$ be forcing functionals for $G,G^*,F,F^*$,
respectively (it is not required that all of them are nontrivial).

\noindent
Using the linearity of $\Lambda$, we find that
\begin{align*} 
&\mathrel{\phantom{=}}\Upsilon_G(\Lambda v-\Lambda u)\leq \frac{1}{2}G(\Lambda v)+\frac{1}{2}G(\Lambda u)-G\left(\frac{\Lambda v+\Lambda u}{2}\right)\\
&\leq \frac{1}{2}G(\Lambda v)+\frac{1}{2}G(\Lambda u)-G\left(\frac{\Lambda v+\Lambda u}{2}\right)+\frac{1}{2}F(v)+\frac{1}{2}F(u)-F\left(\frac{v+u}{2}\right)-\Upsilon_F(v-u)\\
&=\frac{1}{2}J(v)+\frac{1}{2}J(u)-J\left(\frac{v+u}{2}\right)-\Upsilon_F(v-u)\\
&\leq \frac{1}{2}J(v)-\frac{1}{2}J(u)-\Upsilon_F(v-u).
\end{align*}
Similarly,
\begin{align*}
\Upsilon_{G^*}(y^*-p^*)\leq \frac{1}{2}I^*(p^*)-\frac{1}{2}I^*(y^*)-\Upsilon_{F^*}(-\Lambda^*y^*+\Lambda^*p^*).
\end{align*}
Summing up the above two inequalities and noting that $J(u)=I^*(p^*)$ we obtain the principle error estimate (see \cite{Repin_2000, Repin})
\begin{align}\label{Func_a_posteriori_estimate_with_forcing_functionals}
&\mathrel{\phantom{=}}\underline{\Upsilon_G(\Lambda v-\Lambda u)+\Upsilon_{G^*}(y^*-p^*)+\Upsilon_F(v-u)+\Upsilon_{F^*}(-\Lambda^*y^*+\Lambda^*p^*)}\nonumber\\
&\leq \frac{1}{2}(J(v)-I^*(y^*))=\frac{1}{2}\left[G(\Lambda v)+F(v)+G^*(y^*)+F^*(-\Lambda^* y^*)\right]\nonumber\\
&=\underline{\frac{1}{2}\left[D_G(\Lambda v,y^*)+D_F(v,-\Lambda^*y^*)\right]}=:\frac{1}{2}M_\oplus^2(v,y^*),
\end{align}
where 
\begin{align*}
D_G(\Lambda v,y^*):=G(\Lambda v)+G^*(y^*)-\langle y^*,\Lambda v\rangle
\end{align*}
 and 
 \begin{align*}
 D_F(v,-\Lambda^*y^*):=F(v)+F^*(-\Lambda^*y^*)+\langle \Lambda^*y^*,v\rangle
 \end{align*}
 are the {\it compound} functionals for $G$ and $F$, respectively \cite{Repin}. 
 A compound functional is nonnegative by the definition. Moreover, the equality 
\begin{align}\label{Duality_Gap_Equals_MvyStar}
J(v)-I^*(y^*)=D_G(\Lambda v,y^*)+D_F(v,-\Lambda^*y^*)=M_\oplus^2(v,y^*),
\end{align}
shows that  $D_G$ and $D_F$ can vanish simultaneously  
if and only  $v=u$ and $y^*=p^*$. The relation \eqref{Func_a_posteriori_estimate_with_forcing_functionals} 
exposes the general form of  the a posteriori error estimate
of the functional type expressed in terms of forcing functionals. Moreover, setting $v:=u$ and $y^*:=p^*$ in \eqref{Duality_Gap_Equals_MvyStar}, we obtain analogous identities for the primal and dual parts of the error:
\begin{subequations}\label{primal_dual_parts_of_error}
\begin{align}
&J(u)-I^*(y^*)=M_\oplus^2(u,y^*)=D_G(\Lambda u,y^*)+D_F(u,-\Lambda^*y^*), \label{primal_dual_parts_of_error_1}\\
&J(v)-I^*(p^*)=M_\oplus^2(v,p^*)=D_G(\Lambda v,p^*)+D_F(v,-\Lambda^*p^*).\label{primal_dual_parts_of_error_2} 
\end{align}
\end{subequations}
Using the fact that $J(u)=I^*(p^*)$ and that the above equalities \eqref{primal_dual_parts_of_error_1}, \eqref{primal_dual_parts_of_error_2} hold, we obtain another important relation (see \cite{Repin})
\begin{align}\label{important_relation_satisfied_by_the_majorant_M}
&\mathrel{\phantom{=}}M_\oplus^2(v,y^*)=J(v)-I^*(y^*)\nonumber\\
&=J(v)-I^*(p^*)+J(u)-I^*(y^*)=M_\oplus^2(v,p^*)+M_\oplus^2(u,y^*).
\end{align}
Notice that $M_\oplus^2(v,y^*)$ depends on the approximations
$v$ and $y^*$ only and, therefore, is fully computable. The  right-hand side of \eqref{important_relation_satisfied_by_the_majorant_M} can be viewed as a certain measure of the distance between $(u,p^*)$ and $(v,y^*)$, which vanishes if and only if $v=u$ and $y^*=p^*$. Hence the relation 
\begin{align}\label{Functional_error_equality}
D_G(\Lambda v, p^*)+D_F(v,-\Lambda^*p^*)+D_G(\Lambda u,y^*)+D_F(u,-\Lambda^*y^*)=M_\oplus^2(v,y^*)
\end{align}
establishes the equality of the computable term $M_\oplus^2(v,y^*)$ and an error measure natural for this class of variational problems. 

It is worth noting  that the   identity (\ref{Functional_error_equality}) can be represented
in terms of norms if $G$ and $F$ are quadratic functionals. For example,  if $V=H_0^1(\Omega)$, $V^*=H^{-1}(\Omega)$, $Y=[L^2(\Omega)]^d=Y^*$, $G(\Lambda v)=G(\nabla v)=\int\limits_{\Omega}{\frac{1}{2}A\nabla v\cdot \nabla v dx}$ and $F(v)=\int\limits_{\Omega}{\left(\frac{1}{2}v^2-lv\right)dx}$
(where $A$ is a symmetric positive definite matrix with bounded entries), then
\begin{align}\label{Compound_Functionals_Quadratic_Case}
\begin{aligned}
 D_G(\Lambda v,p^*)&=\frac{1}{2}\int\limits_{\Omega}{A\nabla (v-u)\cdot \nabla (v-u) dx},\\
 D_F(v,-\Lambda^*p^*)&=\frac{1}{2}\|v-u\|_{L^2(\Omega)}^2,\\
 D_G(\Lambda u,y^*) &= \frac{1}{2}\int\limits_{\Omega}{A^{-1}(y^*-p^*)\cdot (y^*-p^*) dx},\\
 D_F(u,-\Lambda^*y^*)&=\frac{1}{2}\|\div (y^* - p^*)\|_{L^2(\Omega)}^2.
 \end{aligned}
 \end{align}
 In this case,  the minimizer of  \eqref{general_variational_problem} solves  the linear elliptic problem $-\div(A\nabla u) + u = l$ in $\Omega$ and \eqref{Functional_error_equality} is reduced to the error identity 
\begin{align}\label{Functional_error_equality_quadratic_case}
\begin{aligned}
\int\limits_{\Omega}{A\nabla (v-u)\cdot \nabla (v-u) dx}&+\int\limits_{\Omega}{A^{-1}(y^*-p^*)\cdot (y^*-p^*) dx}\\
&+\|v-u\|_{L^2(\Omega)}^2+\|\div (y^* - p^*)\|_{L^2(\Omega)}^2=2M_\oplus(v,y^*)^2.
\end{aligned}
\end{align}
The sum of the first and the third term in \eqref{Functional_error_equality_quadratic_case}  represents the primal, the sum of the second and fourth term the dual error.

\subsection{Variational form of the problem}
\begin{definition}\label{weak_formulation_for_the_class_of_nonlinear_problems}
We call $u$ a weak solution of \eqref{PBE_special_form_regular_nonlinear_part} if $u\in H_0^1(\Omega)$ and $u$
is such that $b(x,u+w)v\in L^1(\Omega)$ for any  $v\in H_0^1(\Omega)\cap L^\infty(\Omega)$ and 
\begin{align}\label{un_weak_formulation}
a(u,v)+\int_{\Omega}{b(x,u+w)v dx}=\int\limits_{\Omega}{lvdx}, \,  \forall v\in H_0^1(\Omega)\cap L^\infty(\Omega),
\end{align}
where $a(u,v)=\int\limits_{\Omega}{\epsilon\nabla u\cdot\nabla v dx}$ and $b(x,z):=k^2(x)\sinh(z)$.
\end{definition}

The problem has the variational form \eqref{general_variational_problem} if we set $V=H_0^1(\Omega)$ and define $J:H_0^1(\Omega)\to\mathbb R\cup \{+\infty\}$
as follows:
\begin{align}\label{definition_of_J}
J(v):=\left\{
\begin{aligned}
&\int\limits_{\Omega}{\left[\frac{\epsilon(x)}{2}\abs{\nabla v}^2+k^2\cosh(v+w)-lv\right]dx},\text{ if }  k^2\cosh(v+w)\in L^1(\Omega_2),\\
&+\infty, \text{ if } k^2\cosh(v+w) \notin L^1(\Omega_2).
\end{aligned}
\right.
\end{align}
Using the Lebesgue dominated convergence theorem together with the fact that at the minimizer $u$ we have $\cosh(u+w)\in L^1(\Omega)$, it can be seen that the necessary condition for $u$ to be a minimizer of $J$ is
\begin{align*}
\int_{\Omega}{\epsilon \nabla u\cdot\nabla v dx}+\int_{\Omega}{k^2\sinh(u+w)vdx}=\int\limits_{\Omega}{lvdx},\,\forall v\in H_0^1(\Omega)\cap L^\infty(\Omega),
\end{align*}
which is exactly \eqref{un_weak_formulation}. Since $J(v)$ is strictly convex, coercive, and sequentially weakly lower semicontinuous (s.w.l.s) on $H_0^1(\Omega)$ it has a unique minimizer. We note that $J(v)$ is s.w.l.s. because the functional $\int\limits_{\Omega}{\left(\frac{\epsilon}{2}|\nabla v|^2-lv\right) dx}$ is convex and Gateaux differentiable and, therefore, s.w.l.s. over $H_0^1(\Omega)$  (see Corollary 2.4 in \cite{Showalter}). For $d=3$, the functional $\int\limits_{\Omega}{k^2(x)\cosh(v+w) dx}$ is not Gateaux differentiable. In view of Fatou's lemma and the compact embedding of $H_0^1(\Omega)$ into $L^2(\Omega)$ the functional $\int\limits_{\Omega}{k^2(x)\cosh(v+w) dx}$ is also s.w.l.s.. Uniqueness of the solution of \eqref{un_weak_formulation} follows from the monotonicity property of $b$, namely,
\begin{align*}
\int_{\Omega}{\left(b(x,v+w)-b(x,z+w)\right)\left(v-z\right)dx}\ge 0, \, \forall v,z\in H_0^1(\Omega).
\end{align*}
 If $u_1,u_2\in H_0^1(\Omega)$ are two different solutions of \eqref{un_weak_formulation}, then $$a(u_1-u_2,v)+\int_{\Omega}{\left(b(x,u_1+w)-b(x,u_2+w)\right)vdx}=0,\,\forall v\in H_0^1(\Omega)\cap L^\infty(\Omega).$$ Now, applying the theorem in \cite{Brezis_Browder_1978_One_Property_of_Sobolev_Spaces} to $b(x,u_1+w)-b(x,u_2+w)\in H^{-1}(\Omega)\cap L_{loc}^1(\Omega)$ and the function $v=u_1-u_2\in H_0^1(\Omega)$, we conclude that $$a(u_1-u_2,u_1-u_2)+\int_{\Omega}{\left(b(x,u_1+w)-b(x,u_2+w)\right)\left(u_1-u_2\right)dx}=0$$ and, consequently, $u_1=u_2$. We arrive at the following assertion:
\begin{proposition}\label{Proposition_1}
Problem \eqref{general_variational_problem} has a unique minimizer $u$, which coincides with the unique weak solution $u\in H_0^1(\Omega)$ of Problem \eqref{un_weak_formulation}.
\end{proposition}

Next, we show that the solution to Problem \eqref{un_weak_formulation} is essentially bounded.
To prove this, we need the following lemma (see \cite{Kinderlehrer_Stampacchia}).
\begin{lemma}\label{Lemma_B1}
Let $\varphi(t)$ denote a function which is nonnegative and nonincreasing for $s_0\leq t<\infty$. Further, let
\begin{align*}
\varphi(h)\leq C\frac{\varphi(s)^\beta}{(h-s)^\alpha},\,\forall h>s>s_0,
\end{align*}
where $C,\alpha$ are positive constants and $\beta>1$.
If $e \in  \mathbb R$ is defined by $e^\alpha:=C \varphi(s_0)^{\beta-1}2^{\frac{\alpha\beta}{\beta-1}}$,
then $\varphi(s_0+e)=0$.
\end{lemma}
Now, we present the main result of this section.  
\begin{proposition}\label{Proposition_2}
The unique weak solution $u$ to Problem \eqref{un_weak_formulation} belongs to $L^\infty(\Omega)$. Moreover, there is a positive constant $\overline e>0$, depending only on $d$, $\Omega$, $\|l\|_{L^2(\Omega)}$, $\epsilon_{\min}$, such that $\|u\|_{L^\infty(\Omega)}\leq \|w\|_{L^\infty(\Omega_2)}+\overline e$. If $l=0$, then the constant $\overline e$ is equal to zero.
\end{proposition}
\begin{proof}
To prove the boundedness of $u$ we apply the theorem in \cite{Brezis_Browder_1978_One_Property_of_Sobolev_Spaces} once again. 

The first step is to show that~\eqref{un_weak_formulation} holds for $v=G_s(u):=\text{sgn}(u)\max{\{|u|-s,0\}}$ and $s\ge \|w\|_{L^\infty(\Omega_2)}$. Similar test functions $G_s$ have been used in \cite[Theorem B.2]{Kinderlehrer_Stampacchia} in the context of linear elliptic problems.

 First, we note that by Stampacchia's theorem (Theorem 2.2.5 in \cite{Kesavan}) $G_s$ is Lipschitz continuous with $G_s(0)=0$ and hence $G_s(u)\in H_0^1(\Omega)$. From $a(u,\cdot)\in H^{-1}(\Omega)$, $(l,\cdot)\in H^{-1}(\Omega)$ and using \eqref{un_weak_formulation} it follows that $b(x,u+w)\in H^{-1}(\Omega)\cap L_{loc}^1(\Omega)$. Then, in view of Br{\'e}zis and Browder's theorem (\cite{Brezis_Browder_1978_One_Property_of_Sobolev_Spaces}), it suffices to show that 
 \begin{align}\label{Prop2_2:0}
 b(x,u+w)G_s(u)\ge f \text{ a.e. for some } f\in L^1(\Omega).
 \end{align}
Choosing $s\ge \|w\|_{L^\infty(\Omega_2)}$, using the monotonicity of $b(x,\cdot)$, and the fact that $b(x,0)=0$, we obtain 
 \begin{equation}\label{Prop2_2:1}
  b(x,u+w)G_s(u)=\left\{
  \begin{array}{rcl}
  b(x,u+w)(u-s)\ge &0 & \text{ for } u>s\\
  &0 &\text{ for } u\in[-s,s]\\
  b(x,u+w)(u+s)\ge &0 & \text{ for } u<-s,
  \end{array}
  \right.
  \end{equation}
  which shows the assumption \eqref{Prop2_2:0} for $f=0$.
  
  Now we are ready to prove that $u\in L^\infty(\Omega)$.
  First, we consider the case $l=0$. From \eqref{Prop2_2:1}, it follows that 
  \begin{align}\label{Prop2_2:2}
  \int_{\Omega}{ b(x,u+w)G_s(u)dx}\ge 0.
  \end{align}
   Moreover,
  \begin{eqnarray}\label{Prop2_2:3}
  a(u,G_s(u))&=&\int_{\Omega}{\epsilon\nabla u\cdot\nabla G_s(u)}=\int_{\Omega}{\epsilon\nabla G_s(u)\cdot\nabla G_s(u) dx}\nonumber\\
  &\ge& \epsilon_{\min}\|\nabla G_s(u)\|_{L^2(\Omega)}^2\ge \frac{\epsilon_{\min}}{C_{F}^2}\|G_s(u)\|_{L^2(\Omega)}^2,
  \end{eqnarray}
  where $C_{F}$ is the constant in Friedrichs' inequality $\|v\|_{L^2(\Omega)}\leq C_{F}\|\nabla v\|_{L^2(\Omega)}$ that holds for all $v\in H_0^1(\Omega)$.
  Finally, using \eqref{un_weak_formulation}, \eqref{Prop2_2:2}, and \eqref{Prop2_2:3} we get
  \begin{align*}
  \|G_s(u)\|_{L^2(\Omega)}^2\leq 0, \text{ for all } s\ge \|w\|_{L^\infty(\Omega_2)}.
  \end{align*}
Consequently $|u|\leq s$ almost everywhere and for all $s\ge \|w\|_{L^\infty(\Omega_2)}$.\\
In the case where $l$ is not identically zero in $\Omega$, we further estimate $a(u,G_s(u))$ from below and $\int\limits_{\Omega}{l G_s(u) dx}$ from above using the Sobolev embedding $H^1(\Omega)\hookrightarrow L^q(\Omega)$ where $q=\infty$ for $d=1$, $q<\infty$ for $d=2$, and $q=\frac{2d}{d-2}$ for $d\ge 3$. With $q^*$ we will denote the H{\"o}lder conjugate to $q$.  Thus, $q^*=1$ for $d=1$, $q^*=\frac{q}{q-1}>1$ for $d=2$, and $q^*=\frac{2d}{d+2}$ for $d>2$. In order to treat both cases in which we are interested simultaneously, namely $d=2,3$, we can take $q=6$ and $q^*=6/5$. With $C_E$ we denote the embedding constant in the inequality $\|v\|_{L^6(\Omega)}\leq C_E\|v\|_{H^1(\Omega)},\,\forall v\in H^1(\Omega)$, which depends only on the domain $\Omega$ and~$d$. Moreover, we define $A(s):=\{x\in\Omega: |u(x)|> s\}$. For $a(u,G_s(u))$, we have
\begin{align}\label{lower_bound_a_u_Gsu}
a(u,G_s(u))&=\int_{\Omega}{\epsilon\nabla G_s(u)\cdot\nabla G_s(u) dx}\ge \frac{\epsilon_{\min}}{1+C_{F}^2}\|G_s(u)\|_{H^1(\Omega)}^2
\end{align}
and for $\int\limits_{\Omega}{lG_s(u)dx}$ we obtain
\begin{align}\label{upper_bound_l_Gsu}
\int\limits_{\Omega}{lG_s(u)dx}&=\int\limits_{A(s)}{lG_s(u)dx}\leq \|l\|_{L^{q^*}(A(s))}\|G_s(u)\|_{L^{q}(\Omega)}\nonumber\\
&\leq C_E\|l\|_{L^{q^*}(A(s))}\|G_s(u)\|_{H^1(\Omega)}.
\end{align}
Combining \eqref{lower_bound_a_u_Gsu}, \eqref{upper_bound_l_Gsu}, \eqref{Prop2_2:2}, and \eqref{un_weak_formulation}, we obtain
\begin{align}\label{last_inequality_before_estimate_in_terms_of_As_measures}
\frac{\epsilon_{\min}}{1+C_{F}^2}\|G_s(u)\|_{H^1(\Omega)}\leq C_E\|l\|_{L^{q^*}(A(s))}.
\end{align}
The final step before applying Lemma \ref{Lemma_B1} is to estimate the left-hand side of \eqref{last_inequality_before_estimate_in_terms_of_As_measures} from below in terms of $|A(h)|$ for $h>s\ge \|w\|_{L^\infty(\Omega_2)}$ and the right-hand side of \eqref{last_inequality_before_estimate_in_terms_of_As_measures} from above in terms of $|A(s)|$.
Again using the Sobolev embedding $H^1(\Omega)\hookrightarrow L^q(\Omega)$ and H{\"o}lder's inequality yields
\begin{align}\label{lower_bound_in_terms_of_Ah}
&\|G_s(u)\|_{H^1(\Omega)}\ge \frac{1}{C_E}\left(\int\limits_{\Omega}{|G_s(u)|^{q}dx}\right)^{\frac{1}{q}}=\frac{1}{C_E}\left(\int\limits_{A(s)}{||u|-s|^{q}dx}\right)^{\frac{1}{q}}\nonumber\\
&\ge\frac{1}{C_E}\left(\int\limits_{A(h)}{(h-s)^{q}dx}\right)^{\frac{1}{q}}=\frac{1}{C_E}(h-s)|A(h)|^{\frac{1}{q}}
\end{align}
and 
\begin{align}\label{upper_bound_in_terms_of_As}
\|l\|_{L^{q^*}(A(s))}\leq \|l\|_{L^2(\Omega)}|A(s)|^{\frac{2-q^*}{2q^*}}.
\end{align}
Combining \eqref{lower_bound_in_terms_of_Ah}, \eqref{upper_bound_in_terms_of_As}, and \eqref{last_inequality_before_estimate_in_terms_of_As_measures}, we obtain the following inequality for the nonnegative and nonincreasing function $\varphi(t):=|A(t)|$
\begin{align} 
|A(h)|\leq \left(\frac{C_E^2\left(1+C_{F}^2\right)}{\epsilon_{\min}}\|l\|_{L^2(\Omega)}\right)^{q}\frac{|A(s)|^{\frac{q-2}{2}}}{(h-s)^q}, \text{ for all } h>s\ge \|w\|_{L^\infty(\Omega_2)}.
\end{align}
Since $\frac{q-2}{q}=2>1$, by applying Lemma \ref{Lemma_B1} we conclude that there is some $e>0$ such that
\begin{align*}
0<e^{q}&=\left(\frac{C_E^2\left(1+C_{F}^2\right)}{\epsilon_{\min}}\|l\|_{L^2(\Omega)}\right)^{q}|A(\|w\|_{L^\infty(\Omega_2)})|^{\frac{q-4}{2}}2^{\frac{q(q-2)}{q-4}}\\
&\leq \left(\frac{C_E^2\left(1+C_{F}^2\right)}{\epsilon_{\min}}\|l\|_{L^2(\Omega)}\right)^{q}|\Omega|^{\frac{q-4}{2}}2^{\frac{q(q-2)}{q-4}}=:\overline e^{q}
\end{align*}
and $|A(\|w\|_{L^\infty(\Omega_2)}+\overline e)|=0$. Hence $\|u\|_{L^\infty(\Omega)}\leq \|w\|_{L^\infty(\Omega_2)}+\overline e$.
\end{proof}
\begin{remark}
Since $k=0$ in $\Omega_1$, $w\in L^\infty(\Omega_2)$ and $u\in L^\infty(\Omega)$, we conclude that \eqref{un_weak_formulation} holds for all $v\in H_0^1(\Omega)$ resulting in a standard weak formulation. If $k^2$ is uniformly positive in the whole domain $\Omega$ and $w\in L^\infty(\Omega)$, then we have that $\|u\|_{L^\infty(\Omega)}\leq \|w\|_{L^\infty(\Omega)}+\overline e$. On the other hand, if $k=0$ in $\Omega_2$, $k^2$ is uniformly positive in $\Omega_1$, and $w\in L^\infty(\Omega_1)$, we have $\|u\|_{L^\infty(\Omega)}\leq \|w\|_{L^\infty(\Omega_1)}+\overline e$.
\end{remark}

\section{A posteriori error estimates}
We set $V:=H_0^1(\Omega)$, $Y:=[L^2(\Omega)]^d$ ($d=2,3$), and $\Lambda$ the gradient operator $\nabla: H_0^1(\Omega)\to [L^2(\Omega)]^d$. We further denote 
$g: \Omega \times \mathbb R^3\rightarrow \mathbb R, \quad g(x,\xi):=\frac{\epsilon(x)}{2}\abs{\xi}^2$, and $B:\Omega\times  \mathbb R\rightarrow \mathbb R,\quad B(x,\xi):=k^2(x)\cosh(\xi)$. With this notation, we have 
\begin{align*}
&G(\Lambda v):=\int\limits_{\Omega}{g(x,\nabla v(x))dx}=\int\limits_{\Omega}{\frac{\epsilon}{2}\abs{\nabla v}^2dx},\\
& F(v):=\int\limits_{\Omega}{B(x,v(x)+w(x))dx}=\int\limits_{\Omega}{k^2 \cosh(v+w)dx}-\int\limits_{\Omega}{lv dx}.
\end{align*}
For any $v\in V$ the functional $G(\Lambda v)$ is finite, while $F:V\to\mathbb R\cup\{+\infty\}$ may take the value $+\infty$ for some $v\in V$ if $d\ge 3$ (e.g $v=\log{\frac{1}{|x|^\alpha}},\,\alpha\ge d$ on the unit ball in $\mathbb R^d$). However, if $d\leq 2$, then $\exp(v)\in L^1(\Omega),\,\forall v\in H_0^1(\Omega)$ and $F:V\to\mathbb R$ (see \cite{Best_constants_in_some_exponential_Sobolev_inequalities}). Also, $F(0_V)$ is obviously finite since $w\in L^\infty(\Omega_2)$. We set $V^*=H^{-1}(\Omega)$ and $Y^*=Y=[L^2(\Omega)]^d$. In this case, $\Lambda^*$ coincides with $-\div$ considered as an operator from $[L^2(\Omega)]^d$ to $H^{-1}(\Omega)$. First we will give an explicit form of the error estimate in terms of forcing functionals using \eqref{Func_a_posteriori_estimate_with_forcing_functionals} and then we will present the particular form of the error equality \eqref{Functional_error_equality} where the error is measured in a special ''nonlinear norm''. This measure contains the usual combined energy norm terms, i.e. the sum of the energy norms of the errors for the primal and dual problem, but also two additional nonnegative terms due to the nonlinearity $B(x,\xi)$ (or equivalently $b(x,\xi)$) which in some cases may dominate the usual energy norm terms. We start by deriving explicit expressions for $G^*,F^*,\Upsilon_G, \Upsilon_{G^*}, \Upsilon_F, \Upsilon_{F^*}$ and then we will use these expressions to get an explicit form of the abstract error estimates \eqref{Func_a_posteriori_estimate_with_forcing_functionals} and \eqref{Functional_error_equality}.
\subsection{Fenchel Conjugates of $G$ and $F$}
It is easy to find that $G^*(y^*)=\int\limits_{\Omega}{\frac{1}{2\epsilon(x)}|y^*(x)|^2dx}$.
For $y^*\in H(\div;\Omega)$ and an arbitrary function $z:\Omega_2\to \mathbb R$, we introduce 
\begin{align}\label{notation_I_y*}
I_{y^*}(z):=\int\limits_{\Omega_2}{\left[(\div y^*+l)z-B(x,z+w)\right]dx} .
\end{align}
Recalling that the nonlinearity $B$~is supported on $\Omega_2$, we have
\begin{align}\label{Computing_F_star}
&F^*(-\Lambda^*y^*)=\sup\limits_{z\in H_0^1(\Omega)}{\left[\langle -\Lambda^* y^*,z\rangle - F(z)\right]}=
\sup\limits_{z\in H_0^1(\Omega)}{\left[( -y^*,\Lambda z) - F(z)\right]}\nonumber\\
&=\sup\limits_{z\in H_0^1(\Omega)}{\int\limits_{\Omega}{\left[-y^*\cdot\nabla z-B(x,z+w)+lz\right]dx}}=\quad(\text{if $y^*\in H(\div;\Omega)$})\nonumber\\
&=\sup\limits_{z\in H_0^1(\Omega)}{\int\limits_{\Omega}{\left[\div y^*z-B(x,z+w)+lz\right]dx}}\quad (\text{finite if } \div y^*+l=0 \text{ in } \Omega_1)\nonumber\\
&=\sup\limits_{z\in H_0^1(\Omega)}{I_{y^*}(z)}\leq \int\limits_{\Omega_2}{\sup\limits_{\xi\in\mathbb R}{\left[\left(\div y^*(x)+l(x)\right)\xi-B\left(x,\xi+w(x)\right)\right]}dx}\nonumber\\
&=\int\limits_{\Omega_2}{\left[\left(\div y^*(x)+l(x)\right)\xi_0(x)-B\left(x,\xi_0(x)+w(x)\right)\right]dx}=I_{y^*}(\xi_0).
\end{align}
\noindent Here $\xi_0:\Omega_2\to\mathbb R$ is computed from the condition
\begin{align}\label{Necessary_Condition_for_Maximum_in_Ksi}
&\frac{d}{d\xi}\left[\left(\div y^*(x)+l(x)\right)\xi-B\left(x,\xi+w(x)\right)\right]=0, \,\, \text{for a.e.}\,\,x\in \Omega_2,
\end{align}
which is equivalent to
\begin{align*}
&\div y^*(x)+l(x)-k^2(x)\sinh\left(\xi+w(x)\right)=0 \,\, \text{for a.e.}\,\,x\in \Omega_2.
\end{align*}
We notice that \eqref{Necessary_Condition_for_Maximum_in_Ksi} is a necessary condition for a maximum which is also sufficient since $B(x,\cdot)$ is convex.
The solution of the last equation exists, is unique, and is given by
\begin{align}\label{definition_ksi0}
&\xi_0(x)=\arsinh\left(\rho_k(y^*)\right)-w(x)\nonumber\\
&=\ln\left(\rho_k(y^*)+\sqrt{\rho_k^2(y^*)+1}\right)-w(x)=\ln\left(\Theta\left(\rho_k(y^*)\right)\right)-w(x),
\end{align}
where $\rho_k(y^*):=\frac{\div y^*(x)+l(x)}{k^2(x)}$ and $\Theta(s):=s+\sqrt{s^2+1}$ for $s\in\mathbb R$. Note that the exact solution $p^*=\epsilon\nabla u$ of
the dual Problem $(P^*)$ also satisfies the relation $\div(\epsilon\nabla u)+l=0$
because for any $x \in \Omega_1$ it holds $k(x)=0$.
Moreover, since $u\in L^\infty(\Omega)$, $w\in L^\infty(\Omega_2)$, and $l\in L^2(\Omega)$, we see that the
$\div p^*=k^2 \sinh(u+w)+l \in L^2(\Omega)$
and thus $p^*\in H(\div;\Omega)$.
In Proposition \ref{prop_L2_assumption_on_div_k_zero_in_Omega_m}, we will later prove that we have
not overestimated the supremum over $z\in H_0^1(\Omega)$ in \eqref{Computing_F_star} and that we actually have equalities everywhere.
Denoting $S:=\arsinh\left(\rho_k(y^*)\right)$, and using the expression for $\xi_0(x)$ and the formula $\cosh(\arsinh(x))=\sqrt{x^2+1},\,\forall x\in\mathbb R$,
for any $y^*\in H(\div;\Omega)\subset [L^2(\Omega)]^d=Y^*$ with $\div y^*+l=0$ in $\Omega_1$ we obtain an explicit formula for $F^*(-\Lambda^*y^*)$:
\begin{align}\label{Expression_for_Fstar}
\begin{aligned}
&F^*(-\Lambda^* y^*)=\int\limits_{\Omega_2}{\left[k^2\rho_k(y^*)\left(\ln{\left(\Theta\left(\rho_k(y^*)\right)\right)}-w\right)-k^2\sqrt{\rho_k^2(y^*)+1}\right]dx}\\
&=\int\limits_{\Omega_2}{\left[k^2\sinh(S)(S-w)-k^2\cosh(S)\right]dx}
\end{aligned}
\end{align}

\begin{remark}\label{Remark_log_inequality_on_R}
Since $\left|\ln\left(t+\sqrt{t^2+1}\right)\right|\leq |t|,\, \forall t\in\mathbb R$, the function $\ln\left(\Theta(f(x))\right)-w(x)$ belongs to $L^2(\Omega_2)$ for any $f\in L^2(\Omega_2)$ and we conclude that $\xi_0(x)\in L^2(\Omega_2)$ if $y^*\in H(\div;\Omega)$. Therefore the integral in \eqref{Expression_for_Fstar} is well defined. 
\end{remark}

%

Now our goal is to prove that the inequality $\sup\limits_{z\in H_0^1(\Omega)}{I_{y^*}(z)}\leq I_{y^*}(\xi_0)$ holds
as the equality. In other words, we want to prove that the error estimate remains sharp and that the computed majorant
$M_\oplus^2(v,y^*)$ will be indeed zero if approximations $(v,y^*)$ coincide with the exact solution $(u, p^*)$.
\begin{proposition}\label{prop_L2_assumption_on_div_k_zero_in_Omega_m}
For any $y^*\in H(\div;\Omega)$ with $\div y^*+l=0$ in $\Omega_1$ it holds
 $$\sup\limits_{z\in H_0^1(\Omega)}{I_{y^*}(z)}=I_{y^*}(\xi_0)<\infty.$$ 
\end{proposition}
\begin{proof}
The idea is to approximate $f=\frac{\div y^*+l}{k^2}\in L^2(\Omega_2)$ and $w_{\restriction_{\Omega_2}}\in L^\infty(\Omega_2)$ by $C_0^\infty(\Omega_2)$ functions (in the a.e. sense) and use the Lebesgue dominated convergence theorem. Let $f_n\in C_0^\infty(\Omega_2)$ and $w_n\in C_0^\infty(\Omega_2)$ be such that $f_n(x)\to f(x),\, a.e.$ in $\Omega_2$, $|f_n(x)|\leq h(x)\in L^2(\Omega_2)$ (see Theorem 4.9 in \cite{Brezis_FA}), $w_n(x)\to w(x),\,a.e.$ in $\Omega_2$, $|w_n(x)|\leq m+2$, where $m:=\|w\|_{L^\infty(\Omega_2)}$. Then $z_n(x):=\ln\left(\Theta\left(f_n(x)\right)\right)-w_n(x)\to \xi_0(x),\,a.e.$ in $\Omega_2$ and $z_n\in C_0^\infty(\Omega_2)\subset H_0^1(\Omega_2)\subset H_0^1(\Omega)$ (by extending the functions by zero in $\Omega_1$). Since $B(x,\cdot)$ is continuous, we have the pointwise a.e. in $\Omega_2$ convergence
\begin{eqnarray*}
\left(\div y^*(x)+l(x)\right)z_n(x)-B\left(x,z_n+w(x)\right)\to\left(\div y^*(x)+l(x)\right)\xi_0(x)-B(x,\xi_0(x)+w(x))
\end{eqnarray*}
Now we search for a function in $L^1(\Omega_2)$ that majorates the function $\vert \left(\div y^*(x)+l(x)\right)z_n(x)-B\left(x,z_n+w(x)\right)\vert$:
\begin{align}\label{inequality_1}
&\mathrel{\phantom{=}}\left|\left(\div y^*(x)+l(x)\right)z_n(x)-k^2(x)\cosh\left(z_n(x)+w(x)\right)\right|\nonumber\\
&\leq  |\div y^*(x)+l(x)||z_n(x)|+k^2(x)e^{\|w\|_{L^\infty(\Omega_2)}}e^{|z_n(x)|}
\end{align}
Our next goal is to bound $|z_n(x)|$ in \eqref{inequality_1}. For the first summand, we have
$$|z_n(x)|=\left|\ln\left(\Theta\left(f_n(x)\right)\right)-w_n(x)\right|\leq |f_n(x)|+m+2\leq h(x)+m+2\in L^2(\Omega_2)$$
where Remark \ref{Remark_log_inequality_on_R} has been used. However, this bound cannot be used in the second
term because $e^h$ might not belong even to $L^1(\Omega_2)$. In order to find an $L^1$-majorant for the second
summand in \eqref{inequality_1}, we distinguish the following two cases:

In the first case $ f_n(x)>0$.
Then $\left|\ln\left(\Theta\left(f_n(x)\right)\right)\right|\leq \left|\ln\left(\Theta\left(h(x)\right)\right)\right|.$

In the second case ($f_n(x)\leq 0$), we have $\Theta\left(f_n(x)\right)\leq 1$. Therefore, $0\ge~ f_n(x)\ge -h(x)$.
Since $\Theta(s)$ is a monotonically increasing function,
\begin{align*}
&\Theta(0)=1\ge \Theta(f_n(x))\ge \Theta(-h(x))>0.
\end{align*}
From here we obtain
\begin{align*}
&\ln(1)=0\ge \ln\left(\Theta\left(f_n(x)\right)\right)\ge \ln\left(\Theta\left(-h(x)\right)\right)
\end{align*}
and using the relation $\Theta(-h)=\frac{1}{\Theta(h)}$ we conclude that
\begin{align*}
& \left|\ln\left(\Theta\left(f_n(x)\right)\right)\right|\leq \left|\ln\left(\Theta\left(-h(x)\right)\right)\right|=\left|\ln\left(\Theta\left(h(x)\right)\right)\right|.
\end{align*}
Finally, for almost all $x\in\Omega_2$ we have 
\begin{align*}
&\mathrel{\phantom{=}}|z_n(x)|=\left|\ln\left(\Theta\left(f_n(x)\right)\right)-w_n(x)\right|\leq  \left|\ln\left(\Theta\left(h(x)\right)\right)\right|+m+2\\
&= \ln\left(\Theta\left(h(x)\right)\right)+m+2, \text{ because } h(x)\ge 0,\, \text{for a.e.} \, x\in\Omega_2.
\end{align*}
Therefore, 
\begin{align*}
&\mathrel{\phantom{=}}\left|\left(\div y^*(x)+l(x)\right)z_n(x)-k^2(x)\cosh\left(z_n(x)+w(x)\right)\right|\\
&\leq |\div y^*(x)+l(x)|\left(h(x)+\|w\|_{L^\infty(\Omega_2)}+2\right)\\
&+k^2(x)e^{2\|w\|_{L^\infty(\Omega_2)}+2}\Theta\left(h(x)\right):=H(x)\in L^2(\Omega_2),
\end{align*}
where in the last line we used the fact that $\Theta\left(h(x)\right)\in L^2(\Omega_2)$.
All the conditions of the Lebesgue's dominated convergence theorem are satisfied and we see that $I_{y^*}(z_n)\to I_{y^*}(\xi_0)$ and, consequently, $\sup\limits_{z\in H_0^1(\Omega)}{I_{y^*}(z)}=I_{y^*}(\xi_0)$.
\end{proof}

\subsection{Forcing Functionals}

Now our goal is to compute the forcing functionals $\Upsilon_G,\Upsilon_{G^*},\Upsilon_F,\Upsilon_{F^*}$.
For any $y_1,y_2\in Y=[L^2(\Omega)]^d$, we have
\begin{align}\label{Upsilon_G}
&\frac{1}{2}G(y_1)+\frac{1}{2}G(y_2)-G\left(\frac{y_1+y_2}{2}\right)=\frac{1}{2}\int\limits_{\Omega}{\frac{\epsilon}{2}|y_1|^2dx}+\frac{1}{2}\int\limits_{\Omega}{\frac{\epsilon}{2}|y_2|^2dx}-\int\limits_{\Omega}{\frac{\epsilon}{2}\left(\frac{y_1+y_2}{2}\right)^2dx}\nonumber\\
&=\frac{1}{8}\int\limits_{\Omega}{\epsilon|y_1-y_2|^2dx}=:\Upsilon_G(y_1-y_2).
\end{align}
Similarly, for any $y_1^*,y_2^*\in Y^*=[L^2(\Omega)]^d$ we get
\begin{align}\label{Upsilon_Gstar}
\Upsilon_{G^*}(y_1^*-y_2^*)=\frac{1}{8}\int\limits_{\Omega}{\frac{1}{\epsilon}|y_1^*-y_2^*|^2dx}.
\end{align}
We note that according to the definition of uniformly convex functional, $F$ is not uniformly convex because $k=0$ in $\Omega_1$ and therefore $F$ is affine on the linear subspace $H_0^1(\Omega_m)\subset H_0^1(\Omega)=V$,
where functions in $H_0^1(\Omega_1)$ are extended by zero into $\Omega_2$.
Thus, if there is a nonnegative functional $\Upsilon_F$ such that \eqref{Defining_inequality_for_uniform_convexity}
is satisfied for all $v_1,v_2\in V$, it will be necessarily zero for all $v_1,v_2$ such that $v_1=v_2$ in $\Omega_2$
(see Remark \ref{Remark_forcing_functional}). If $F_1$ is uniformly convex with forcing functional $\Upsilon_{F_1}$
and $F_2$ is convex, then $F=F_1+F_2$ is uniformly convex with a forcing functional $\Upsilon_F=\Upsilon_{F_1}$.
Since $F(v)=\int\limits_{\Omega}{B(x,v+w)dx}-\int\limits_{\Omega}{lvdx}$ it is enough to find a forcing functional
$\Upsilon_{\tilde B_x}:\mathbb R\to\mathbb R$ for $\tilde B_x(.):=B(x,.+w(x)):\mathbb R\to\mathbb R$ for a.e.
$x\in \Omega$. In this case, we will define $\Upsilon_F(v_1-v_2):=\int_{\Omega}{\Upsilon_{\tilde B_x}(v_1-v_2)dx}$.
We need to find $\Upsilon_{\tilde B_x}$ such that
\begin{align}
\Upsilon_{\tilde B_x}(\xi_1-\xi_2)\leq \frac{1}{2}\tilde B_x(\xi_1)+\frac{1}{2}\tilde B_x(\xi_2)-\tilde B_x\left(\frac{\xi_1+\xi_2}{2}\right),\,\forall \xi_1,\xi_2\in\mathbb R.
\end{align}
If we denote $\zeta:=\xi_1-\xi_2$ then $\xi_2=\xi_1-\zeta$ and we have
\begin{align*}
&\Upsilon_{\tilde B_x}(\zeta)\leq \frac{1}{2}\tilde B_x(\xi_1)+\frac{1}{2}\tilde B_x(\xi_1-\zeta)-\tilde B_x\left(\frac{2\xi_1-\zeta}{2}\right) ,\,\forall \xi_1,\zeta\in\mathbb R.
\end{align*}
Thus, we define 
\begin{align*}
&\Upsilon_{\tilde B_x}(\zeta):=\inf\limits_{\xi_1\in\mathbb R}{\left[\frac{1}{2}\tilde B_x(\xi_1)+\frac{1}{2}\tilde B_x(\xi_1-\zeta)-\tilde B_x\left(\frac{2\xi_1-\zeta}{2}\right) \right]}\nonumber\\
&=\inf\limits_{\xi_1\in\mathbb R}{\left[\frac{k^2(x)}{2}\left(\cosh(\xi_1+w(x))+\cosh(\xi_1-\zeta+w(x))-2\cosh\left(\frac{2\xi_1-\eta}{2}+w(x)\right)\right)\right]}.
\end{align*}
A necessary condition for $\bar \xi_1$ to minimize the above expression is that the first derivative with respect to $\xi_1$ vanishes at $\bar \xi_1$, i.e.,
\begin{equation}\label{forcing_functional_for_B_necessary_condition_for_min}
    \begin{split}
   \frac{d}{d\xi_1}\left[\frac{k^2(x)}{2}\left(\cosh(\xi_1+w(x))+\cosh(\xi_1-\zeta+w(x))\right.\right.\\
\left.\left.-2\cosh\left(\frac{2\xi_1-\zeta}{2}+w(x)\right)\right)\right]=0.
    \end{split}
\end{equation}
Since the second derivative with respect to $\xi_1$
is always positive because of the convexity of $\cosh$, we see that
the necessary condition for a minimum is also a sufficient condition.
After using the formula $\sinh(\xi)=\frac{\exp(\xi)-\exp(-\xi)}{2}$ and making the substitutions $\exp(\xi_1+w(x))=r>0$
and $\exp(\zeta/2)=a>0$ in \eqref{forcing_functional_for_B_necessary_condition_for_min} we get
\begin{align*}
\frac{1}{2}\left(r-\frac{1}{r}\right)+\frac{1}{2}\left(\frac{r}{a^2}-\frac{a^2}{r}\right)-\left(\frac{r}{a}-\frac{a}{r}\right)=0
\end{align*}
from which after solving for $r$ we get $r=a$ and $\bar\xi_1(x)=\frac{\zeta}{2}-w(x)$. Therefore,
\begin{align}
\Upsilon_{\tilde B_x}(\zeta)=k^2(x)\left[\cosh\left(\frac{\zeta}{2}\right)-1\right]
\end{align}
and thus
\begin{align}
\Upsilon_F(v-u)=\int\limits_{\Omega}{k^2\left[\cosh\left(\frac{v-u}{2}\right)-1\right]dx}.
\end{align}
Here we note that
\begin{align}\label{Quadratic_lower_bound_for_Upsilon_F}
 \Upsilon_F(v-u)\ge \int\limits_{\Omega}{\frac{k^2}{8}(v-u)^2dx}:=\underline{\Upsilon}_F(v-u)
 \end{align}
 since $\cosh(\frac{\xi}{2})-1\ge \frac{\xi^2}{8},\,\forall \xi\in\mathbb R$ and thus $\underline{\Upsilon}_F$ is also
 a forcing functional (see Figure \ref{Lower_Bound_Upsilon_F}). If $k^2$ is uniformly positive on the whole
 domain $\Omega$, then $\sqrt{\underline{\Upsilon}_F(v-u)}$ is equivalent to the $L^2(\Omega)$ norm of $v-u$.
 \begin{figure}
 \centering
        \centering
        \captionsetup{width=.7\linewidth}
        \includegraphics[width=0.8\linewidth]{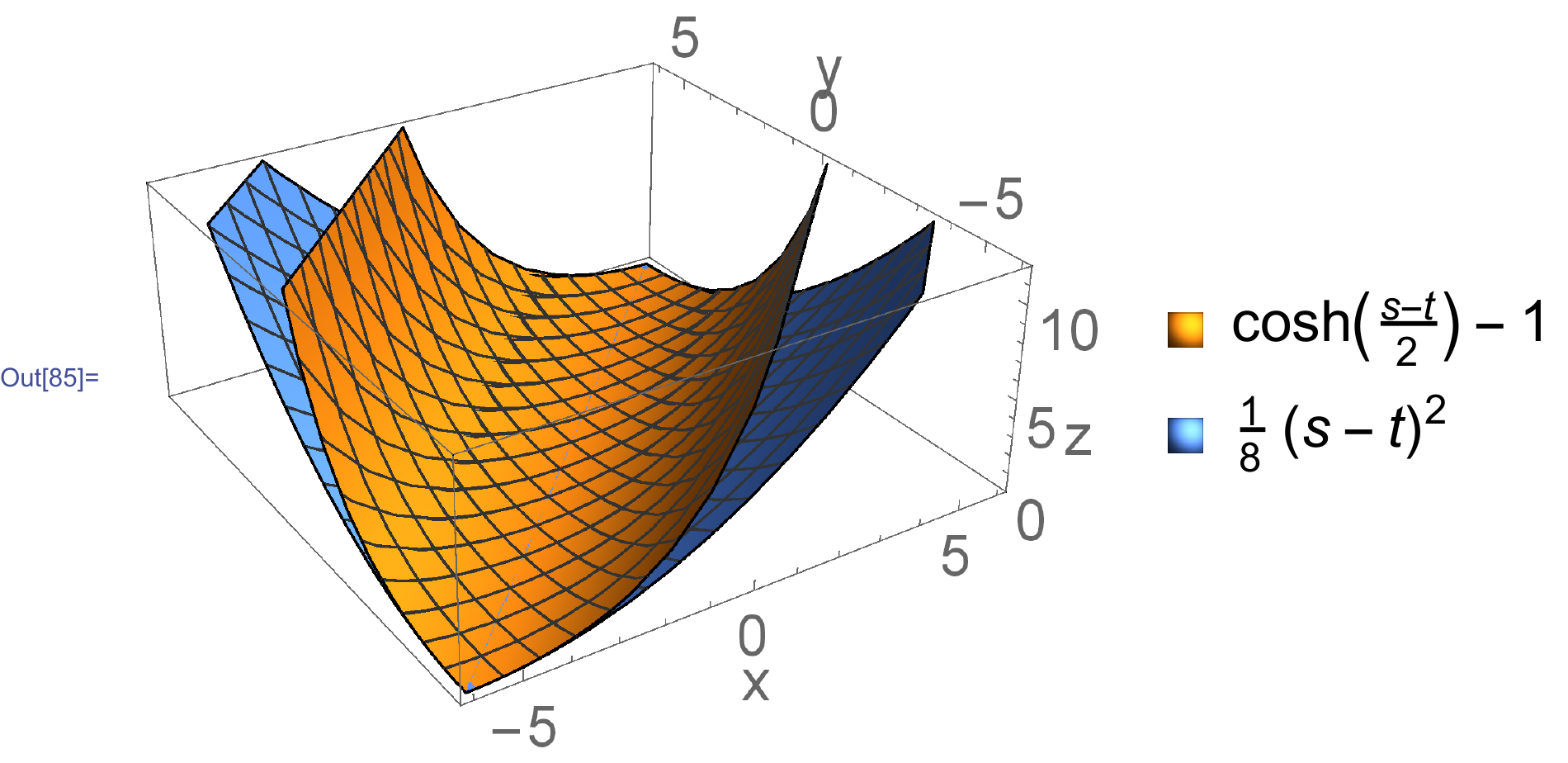}
        \caption{Lower bound of $\cosh\left(\frac{s-t}{2}\right)-1$ formed by a quadratic function.}
        \label{Lower_Bound_Upsilon_F}
\end{figure}
Finally, we obtain a forcing functional for $F^*$. We have derived an explicit expression for $F^*$ only for arguments
of the form $-\Lambda^*y^*$ where $y^*\in H(\div;\Omega)\subset [L^2(\Omega)]^d=Y^*$. Therefore, we will search
only for a forcing functional $\Upsilon_{F^*}$ which takes such arguments, i.e. $F^*(-\Lambda^*y^*+\Lambda^*p^*)$.
We should also note that these arguments are abstract elements from the dual space of $H_0^1(\Omega)$.
However (as we already saw in \eqref{Computing_F_star}), if $y^*\in H(\div;\Omega)$, then 
\begin{eqnarray}
F^*(-\Lambda^*y^*)&=&\sup\limits_{z\in H_0^1(\Omega)}{\int\limits_{\Omega}{\left[\div y^*z-B(x,z+w)+lz\right]dx}}\nonumber\\
&=&\sup\limits_{z\in L^2(\Omega)}{\int\limits_{\Omega}{\left[\div y^*z-B(x,z+w)+lz\right]dx}}\nonumber\\
&=:&\tilde F^*(\div y^*)
\end{eqnarray}
where $\tilde F^*$ is the Fenchel conjugate of the functional $\tilde F: L^2(\Omega)\to\mathbb R\cup\{+\infty\}$ defined by 
\begin{align*}
\tilde F(z)=\int\limits_{\Omega}{\left[B(x,z+w)-lz\right]dx},\,\forall z\in L^2(\Omega)
\end{align*}
and the supremum over $z\in H_0^1(\Omega)$ is equal to the supremum over $z\in L^2(\Omega)$ due to Remark \ref{Remark_log_inequality_on_R} and Proposition \ref{prop_L2_assumption_on_div_k_zero_in_Omega_m}\footnote{Note that the functional $(\div y^*,z)-\tilde F(z)$ is only upper semicontinuous and is not continuous over $L^2(\Omega)$ and thus we cannot use a density argument to prove that both suprema are equal.}.
Thus, for all $y^*\in H(\div;\Omega)$ we have
\begin{align}\label{defining_equality_for_tildeF_change_of_variables}
F^*(-\Lambda^*y^*)=\tilde F^*(\div y^*)
\end{align}
and additionally $\tilde F^*(\div y^*)$ is equal to the expression in \eqref{Expression_for_Fstar}:
\begin{align*}
&\Tilde F^*(\div y^*)=I_{y^*}(\xi_0)=:\int\limits_{\Omega_2}{f_x^*(\div y^*)dx}
\end{align*}
where $f_x^*:\mathbb R\to\mathbb R$ and for a.e. $x\in\Omega_2$ and for all $ r\in\mathbb R$
\begin{align*}
&f_x^*(r)=\left(r+l(x)\right) \left(\arsinh\left(\frac{r+l(x)}{k^2(x)}\right)-w(x)\right)-k^2(x)\cosh\left(\arsinh\left(\frac{r+l(x)}{k^2(x)}\right)\right).
\end{align*}
  If $\Upsilon_{\tilde F^*}$ is a forcing functional for $\tilde F^*$, then we have
 \begin{align*}
\tilde F^*\left(\frac{\div y^*-\div p^*}{2}\right)+\Upsilon_{\tilde F^*}(\div y^*-\div p^*)\leq \frac{1}{2}\tilde F^*(\div y^*)+\frac{1}{2}\tilde F^*(\div p^*),
\end{align*}
which due to \eqref{defining_equality_for_tildeF_change_of_variables} is exactly the same as 
\begin{align*}
F^*\left(\frac{-\Lambda^*y^*+\Lambda^*p^*}{2}\right)+\Upsilon_{F^*}(-\Lambda^*y^*+\Lambda^*p^*)\leq \frac{1}{2}F^*(-\Lambda^*y^*)+\frac{1}{2}F^*(-\Lambda^*p^*).
\end{align*}
Here $\Upsilon_{F^*}(-\Lambda^*(y^*-p^*)):=\Upsilon_{\tilde F^*}(\div(y^*-p^*)),\,\forall y^*,p^*\in H(\div;\Omega)$.
With this remark, it is clear that we only need to find the forcing functional
$\Upsilon_{\tilde F^*}: L^2(\Omega)\supset\text{R}(\div)\to\mathbb {\overline R}$ where $\text{R}(\div)$ is the range
of the divergence operator as an operator from $[L^2(\Omega)]^d$ to $L^2(\Omega)$. Again, since $\tilde F^*$ is an
integral functional, i.e. $\tilde F^*(\div y^*)=\int\limits_{\Omega_2}{f_x^*(\div y^*)dx}$ it is enough to find a forcing
functional $\Upsilon_{f_x^*}:\mathbb R\to\mathbb R$  for $f_x^*$.

For any $\xi_1,\xi_2\in\mathbb R$ it holds
\begin{align*}
\Upsilon_{f_x^*}(\xi_1-\xi_2)\leq \frac{1}{2}f_x^*(\xi_1)+\frac{1}{2}f_x^*(\xi_2)-f_x^*\left(\frac{\xi_1+\xi_2}{2}\right).
\end{align*}
Since $f_x^*(r)=:h_x^*(r+l(x))$, where 
\[
h_x^*(r)=r \left(\arsinh\left(\frac{r}{k^2(x)}\right)-w(x)\right)-k^2(x)\cosh\left(\arsinh\left(\frac{r}{k^2(x)}\right)\right),\forall r\in\mathbb R,
\]
it suffices to find a forcing functional $\Upsilon_{h_x^*}$ for $h_x^*(r)$ and then define
$\Upsilon_{f_x^*}(\zeta)=\Upsilon_{h_x^*}(\zeta),\,\forall\zeta\in\mathbb R$.
Again, we make the substitution $\zeta=\xi_1-\xi_2$. We have
\begin{align*}
\inf\limits_{\xi_1\in\mathbb R}{\left[\frac{1}{2}h_x^*(\xi_1)+\frac{1}{2}h_x^*(\xi_1-\zeta)-h_x^*\left(\frac{2\xi_1-\zeta}{2}\right)\right]}=0,\,\forall \zeta\in\mathbb R.
\end{align*}
This means that the function $h_x^*$ is not uniformly convex over $\mathbb R$. However, if we restrict $\xi_1,\xi_2$ to a ball $B(0,\delta)\subset \mathbb R^2,\,\delta>0$, then the above infimum is greater than zero, and according to the definition of uniform convexity, $h_x^*$ will be uniformly convex over the ball $B(0,\delta)$. This can be useful in the context of our particular problem when $l\in L^\infty(\Omega_2)$, since from Proposition~\ref{Proposition_2} we have that $\div(\epsilon\nabla  u)+l=\div p^*+l=k^2\sinh(u+w)\in L^\infty(\Omega_2)$. Therefore, if we pick the approximations $y^*\in H(\div;\Omega)$ for the solution $p^*$ of the dual Problem $(P^*)$  with $\div y^*+l=0$ in  $\Omega_1$ and additionally such that $\div y^*+l\in L^\infty(\Omega_2)$ with $-M\leq \div y^*+l\leq M$ for some $M>0$, then for almost each $x\in\Omega_2$
\begin{align*}
&-\|k^2\|_{L^\infty(\Omega_2)}\sinh\left(2\|w\|_{L^\infty(\Omega_2)}+\overline e\right)\leq \div p^*+l\leq \|k^2\|_{L^\infty(\Omega_2)}\sinh\left(2\|w\|_{L^\infty(\Omega_2)}+\overline e\right),\\ 
&-\|k^2\|_{L^\infty(\Omega_2)}\sinh\left(2\|w\|_{L^\infty(\Omega_2)}+\overline e\right)-M\leq \div y^*-\div p^*\\
&\leq \|k^2\|_{L^\infty(\Omega_2)}\sinh\left(2\|w\|_{L^\infty(\Omega_2)}+\overline e\right)+M ,
\end{align*}
and thus we can choose $\delta=\max{\{\|k^2\|_{L^\infty(\Omega_2)}\sinh\left(2\|w\|_{L^\infty(\Omega_2)}+\overline e\right),M\}}$ and 
$B(0,\delta)=[-\delta,\delta]^2$ in $\mathbb R^2$.
In this case, for any $-2\delta\leq\zeta\leq 2\delta$
\begin{align*}
\Upsilon_{h_x^*}(\zeta)=\inf\limits_{-\delta\leq \xi_1\leq\delta}{\left[\frac{1}{2}h_x^*(\xi_1)+\frac{1}{2}h_x^*(\xi_1-\zeta)-h_x^*\left(\frac{2\xi_1-\zeta}{2}\right)\right]}>0.
\end{align*}
Now, depending on $k$ and the above defined $\delta$, one can find a constant $C_1$ such that for all $\xi_1\in[-\delta,\delta]$ and $\zeta\in [-2\delta,2\delta]$ the following inequality is satisfied
\begin{align*}
C_1\zeta^2&\leq \frac{1}{2}h_x^*(\xi_1)+\frac{1}{2}h_x^*(\xi_1-\zeta)-h_x^*\left(\frac{2\xi_1-\zeta}{2}\right).
\end{align*}
This means that 
we can define $\Upsilon_{f_x^*}(\zeta)=\Upsilon_{h_x^*}(\zeta)=C_1\zeta^2$ and consequently the forcing functional
$\Upsilon_{F^*}$ as follows
\begin{align}\label{Forcing_functional_FStar}
&\Upsilon_{F^*}(-\Lambda^* y^*+\Lambda^* p^*)=\Upsilon_{\tilde F^*}(\div y^*-\div p^*)\nonumber\\
&=\int\limits_{\Omega_2}{\Upsilon_{f_x^*}(\div y^*-\div p^*) dx}=\int\limits_{\Omega_2}{C_1(\div y^*-\div p^*)^2dx}.
\end{align}

\subsection{Error measures}\label{Error_measures}
In this section, we apply the abstract framework from Section \ref{Section_Abstract_Framework} and derive explicit forms of
relations~\eqref{Func_a_posteriori_estimate_with_forcing_functionals} and~\eqref{Functional_error_equality} adapted to our problem.
Using \eqref{Quadratic_lower_bound_for_Upsilon_F} and
\eqref{Forcing_functional_FStar}, for any $v\in H_0^1(\Omega)$ and $y^*\in Y_M^*$, where
$$Y_M^*:=\{y^*\in H(\div;\Omega), \text{ s.t. } \div y^*+l=0 \text{ in } \Omega_1 \text{ and } -M\leq \div y^*+l\leq M \text{ in } \Omega_2\},$$ the estimate \eqref{Func_a_posteriori_estimate_with_forcing_functionals} takes the form
\begin{align}\label{Explicit_form_of_Error_Estimate_with_Forcing_Functionals}
&\frac{1}{8}\int\limits_{\Omega}{\epsilon|\nabla (v-u)|^2dx}+\frac{1}{8}\int\limits_{\Omega}{\frac{1}{\epsilon}|y^*-p^*|^2dx}+\frac{1}{8}\int\limits_{\Omega}{k^2(v- u)^2dx}+C_1\int\limits_{\Omega}{|\div y^*-\div  p^*|^2dx}\nonumber\\
&\leq \frac{1}{8}\int\limits_{\Omega}{\epsilon|\nabla (v-u)|^2dx}+\frac{1}{8}\int\limits_{\Omega}{\frac{1}{\epsilon}|y^*-p^*|^2dx}\\
&+\int\limits_{\Omega}{k^2\left[\cosh\left(\frac{v-u}{2}\right)-1\right]dx}+C_1\int\limits_{\Omega}{|\div y^*-\div p^*|^2dx}\leq \frac{1}{2}M_{\oplus}^2(v,y^*)\nonumber,
\end{align}
where the constant $C_1$ depends on $k$, $\|w\|_{L^\infty(\Omega_2)}$, $\overline e$, and $M$. The quantity $M_{\oplus}^2(v,y^*)$ is fully computable and is given by the relation
\begin{align}\label{Explicit_form_of_Nonlinear_Majorant}
\begin{aligned}
M_{\oplus}^2(v,y^*)&=D_G(\Lambda v,y^*)+D_F(v,-\Lambda^*y^*)\\
&=G(\Lambda v)+G^*(y^*)-\langle y^*,\Lambda v\rangle+F(v)+F^*(-\Lambda^*y^*)+\langle \Lambda^*y^*,v\rangle\\
&=\int\limits_{\Omega}{\eta^2(x)dx}=\frac{1}{2}\vertiii{\epsilon\nabla v-y^*}_*^2+D_F(v,-\Lambda^*y^*),
\end{aligned}
\end{align}
where
\begin{align}\label{integrand_of_the_Majorant}
 \begin{aligned}
    \eta^2(x)=\left\{
                \begin{array}{ll}
                \frac{1}{2\epsilon}|\epsilon\nabla v-y^*|^2,\, x\in \Omega_1\\
                 \frac{1}{2\epsilon}|\epsilon\nabla v-y^*|^2+k^2\cosh(v+w)-lv\\
+k^2\rho_k(y^*)\left(\ln{\left(\Theta\left(\rho_k(y^*)\right)\right)}-w\right)
-k^2\sqrt{\rho_k^2(y^*)+1}-\div y^* v,\, x\in \Omega_2
                \end{array}
              \right.
\end{aligned}
\end{align}
It is clear that $\eta^2(x)\ge 0$ since it is the sum of the compound functionals generated by $\tilde g_x(s):=g(x,s)$ and $\tilde B_x(s)-l(x)s=B(x,s+w(x))-l(x)s$ evaluated at $(\nabla v(x),y^*(x))$ and $(v(x),\div y^*(x))$ respectively. It therefore qualifies as an error indicator, provided that $y^*$ is chosen appropriately, which we demonstrate with numerical experiments in the next section.
One can also work with the space $Y_\infty^*:=\{y^*\in H(\div;\Omega) \text{ s.t. } \div y^*+l=0 \text{ in }\Omega_2\}$ instead of $Y_M^*$. In this case, $F^*$ does not posses a nonzero forcing functional and we skip the term with $C_1$ in \eqref{Explicit_form_of_Error_Estimate_with_Forcing_Functionals}.

Using the expression for $G^*$, we obtain
\begin{align}\label{DGvpStar}
D_G(\Lambda v,p^*)=\frac{1}{2}\int\limits_{\Omega}{\epsilon |\nabla (v-u)|^2 dx}=:\frac{1}{2}\vertiii{\nabla(v-u)}^2
\end{align}
and 
\begin{align}\label{DGuyStar}
D_G(\Lambda u,y^*) =\frac{1}{2}\int\limits_{\Omega}{\frac{1}{\epsilon}|y^*-p^*|^2 dx}=:\frac{1}{2}\vertiii{y^*-p^*}_*^2.
\end{align}
Now, we find explicit expressions for the nonlinear measures $D_F(v,-\Lambda^*p^*)$ and $D_F(u,-\Lambda^*y^*)$ similar to the ones for the case of quadratic $F$ in \eqref{Compound_Functionals_Quadratic_Case} for the linear elliptic equation $-\div(A\nabla u) + u =l$. First, we prove the following assertion:
\begin{proposition}\label{Proposition_Main_Algebraic_Inequality_for_DF}
For all $s,t\in\mathbb R$ it holds 
\begin{align}\label{Main_Algebraic_Inequality_for_DF}
\frac{(t-s)^2}{2}\leq A(s,t) \leq \frac{(\sinh(t)-\sinh(s))^2}{2},
\end{align}
where $A(s,t)=\cosh(t)-\cosh(s)+s\sinh(s)-t\sinh(s)$.
\end{proposition}
\begin{proof}
For the first inequality, denote $A_1(s,t):=A(s,t)-\frac{(t-s)^2}{2}$. We prove that for any fixed $s\in\mathbb R$, $A_1(s,t)\ge 0$ for all $t\in\mathbb R$. If $s=0$, we have $\cosh(t)-1\ge \frac{t^2}{2}$ for all $t\in\mathbb R$. If $s\neq 0$, the necessary condition for a minimum in $t$ is $\frac{\partial A_1}{\partial t}(s,t)=0$ which is equivalent to $\sinh(t)-\sinh(s)-t+s=0$. The only solution of this equation is $t=s$ because the function $\sinh(t)-t$ is strictly monotonically increasing. It is left to observe that at $t=s$ we have $\frac{\partial^2 A_1}{\partial t^2}=\cosh(s)-1>0$ and that $A_1(s,t=s) = 0$.
For the second inequality, denote $A_2(s,t):=\frac{(\sinh(t)-\sinh(s))^2}{2}- A(s,t)$.  If $t=0$, the inequality $A_2(s,0)\ge 0$ reduces to the inequality $q(s):=\frac{\sinh^2(s)}{2}-1+\cosh(s)-s\sinh(s)\ge 0$ which is true since  the minimum of the function $q(s)$ is $0$. If $t\neq 0$, the necessary condition for a minimum in $s$ is $\frac{\partial A_2}{\partial s}=0$ which is equivalent to $\cosh(s)(\sinh(s)-\sinh(t)-s+t)=0$. The only solution of this equation is $s=t$. Now, it is left to observe that at $s=t$ we have $\frac{\partial^2 A_2}{\partial s^2}=\cosh(t)(\cosh(t)-1)>0$ and that $A_2(s=t,t)=0$.
\end{proof}
Since for the exact solution $u$ we have $\rho_k(p^*)=\sinh(u+w)$ and $u=\arsinh{\left(\rho_k(p^*)\right)}-w$ a.e. in $\Omega_2$, we find that
$$
\begin{aligned}
D_F(v,-\Lambda^*p^*)
&=\int\limits_{\Omega_2}{\left(k^2\cosh(v+w)-lv+k^2\sinh(u+w)u-k^2\cosh(u+w)-\div p^* v\right)dx} \\
&=\int\limits_{\Omega_2}{k^2\left(\cosh(v+w)-\cosh(u+w)+u\sinh(u+w)-v\sinh(u+w)\right)dx}.
\end{aligned}
$$
Similarly,
%
$
\begin{aligned}
&D_F(u,-\Lambda^*y^*)=\int_{\Omega_2}{k^2\left(\cosh(T)-\cosh(S)+S\sinh(S) -T\sinh(S)\right)dx},
\end{aligned}
$
where  $T:=\arsinh\left(\rho_k(p^*)\right)$.
The nonlinear quantities $D_F(v,-\Lambda^*p^*)$ and $D_F(u,-\Lambda^*y^*)$
measure the error in $v$ and in $\div y^*$, respectively.
Using inequality~\eqref{Main_Algebraic_Inequality_for_DF}, we can represent these two measures in a form, which resembles the corresponding
estimates in the case~\eqref{Compound_Functionals_Quadratic_Case} of a quadratic functional $F$, namely,
\begin{align}\label{inequality_for_DFvMinusLambdaStarpStar}
\int\limits_{\Omega_2}{\frac{k^2}{2}(v-u)^2 dx}\leq D_F(v,-\Lambda^*p^*)\leq \int\limits_{\Omega_2}{\frac{k^2}{2}(\sinh(v+w)-\sinh(u+w))^2 dx}
\end{align}
and 
\begin{align}\label{inequality_for_DFuMinusLambdaStaryStar}
\int\limits_{\Omega_2}{\frac{k^2}{2}(T-S)^2 dx}\leq D_F(u,-\Lambda^*y^*)\leq \int\limits_{\Omega_2}{\frac{1}{2k^2}(\div p^*-\div y^*)^2 dx}.
\end{align}
Note that for $k\ge k_{\min}>0$ in $\Omega$ the equivalences
$\int\limits_{\Omega}{\frac{k^2}{2}(v-u)^2 dx} \eqsim \|v-u\|_{L^2(\Omega)}^2$ and
$\int\limits_{\Omega}{\frac{1}{2k^2}(\div p^*-\div y^*)^2 dx} \eqsim \|\div y^*-\div p^*\|_{L^2(\Omega)}^2$
hold.
Moreover, replacing the nonlinear term $k^2\sinh(u+w)$ with $u$, the inequalities \eqref{inequality_for_DFvMinusLambdaStarpStar}
and \eqref{inequality_for_DFuMinusLambdaStaryStar} reduce to the equalities for $D_F(v,-\Lambda^*p^*)$ and $D_F(u,-\Lambda^*y^*)$
in~\eqref{Compound_Functionals_Quadratic_Case} because in this case the inverse function of $f(x)=x$ is again $f(x)$. The functions on the left-hand side, in the middle, and on the right-hand side in the inequality \eqref{Proposition_Main_Algebraic_Inequality_for_DF} are depicted on Figure \ref{Proposition_3p2_Inequality}.
\begin{figure}
 \centering
        \centering
        \captionsetup{width=.7\linewidth}
        \includegraphics[width=0.8\linewidth]{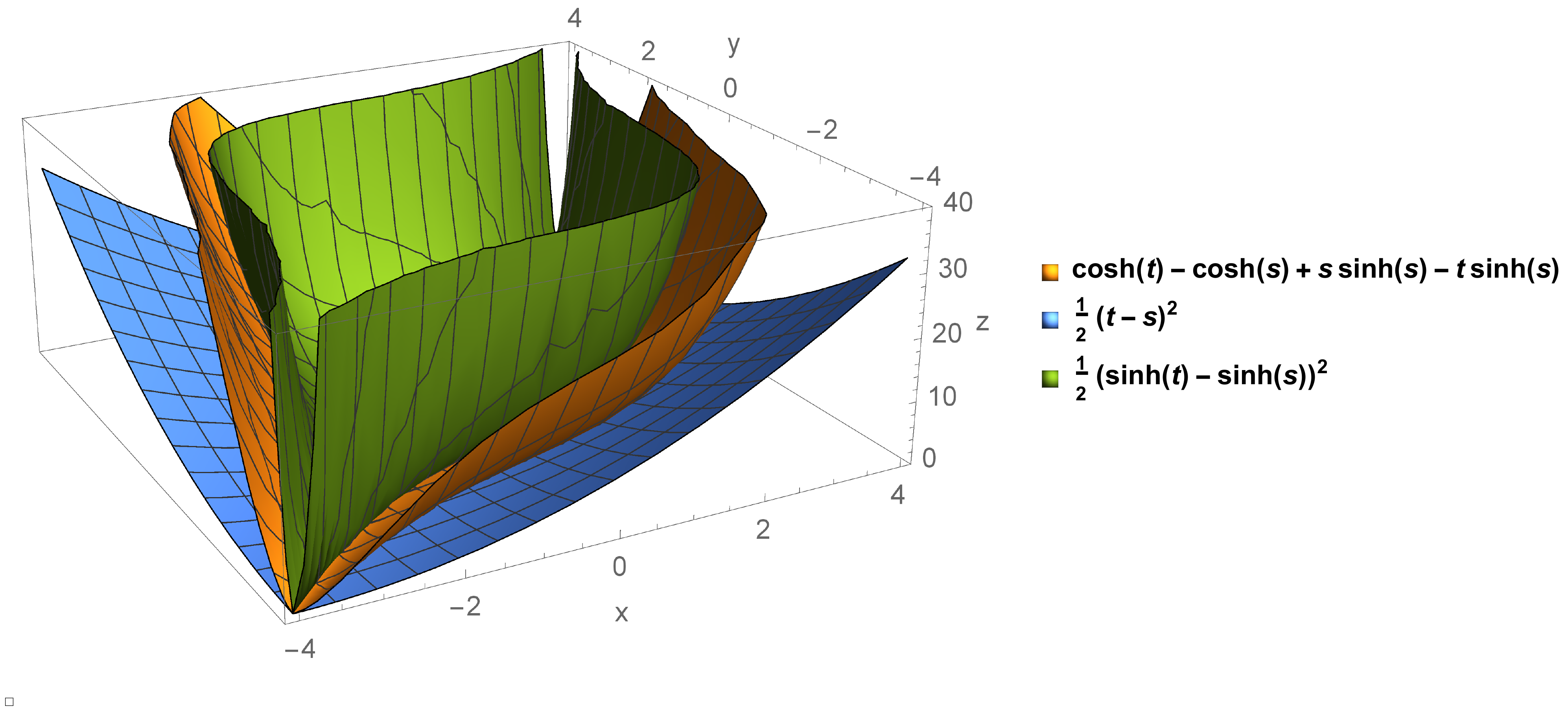}
        \caption{Constituting functions in the inequality \eqref{Main_Algebraic_Inequality_for_DF}.}
        \label{Proposition_3p2_Inequality}
\end{figure}
Further, if $v$ is in a $\delta_1$-neighborhood of $u$ in $L^\infty(\Omega)$ norm, then we can find a constant 
$C_1\left(\delta_1,\|u\|_{L^\infty(\Omega)}\right)>1$ such that
\begin{align}\label{Upper_Quadratic_Bound_for_DFvMinusLambdaStarpStar}
\int\limits_{\Omega_2}{\frac{k^2}{2}(\sinh(v+w)-\sinh(u+w))^2 dx}\leq C_1\left(\delta_1,\|u\|_{L^\infty(\Omega)}\right)\int\limits_{\Omega_2}{\frac{k^2}{2}(v-u)^2 dx}.
\end{align}
 Analogously, if $l\in L^\infty(\Omega_2)$ and $\|\div (y^*-p^*)\|_{L^\infty(\Omega)}\leq \delta_2$ (recall that when $l\in L^\infty(\Omega_2)$, $\div p^* $ is in $L^\infty(\Omega)$), then we can find a constant $C_2\left(\delta_2,\|\div p^*\|_{L^\infty(\Omega)}\right)<1$ such that 
 \begin{align}\label{Lower_Quadratic_Bound_for_DFuMinusLambdaStaryStar}
 C_2\left(\delta_2,\|\div p^*\|_{L^\infty(\Omega)}\right)\int\limits_{\Omega_2}{\frac{1}{2k^2}(\div p^*-\div y^*)^2 dx}\leq \int\limits_{\Omega_2}{\frac{k^2}{2}(T-S)^2 dx}.
\end{align}
Notice that if $k^2\ge k_{\min}>0$ in $\Omega$, then everywhere in \eqref{inequality_for_DFvMinusLambdaStarpStar}, \eqref{inequality_for_DFuMinusLambdaStaryStar}, \eqref{Upper_Quadratic_Bound_for_DFvMinusLambdaStarpStar}, and \eqref{Lower_Quadratic_Bound_for_DFuMinusLambdaStaryStar},
the integrals are taken over the entire domain $\Omega$. Now, the abstract error identity \eqref{Functional_error_equality} takes the form
\begin{align}\label{Explicit_form_of_Functional_error_equality}
&\mathrel{\phantom{=}}\frac{1}{2}\vertiii{\nabla(u-v)}^2+\frac{1}{2}\vertiii{p^*-y^*}_*^2\nonumber\\
&+ \int\limits_{\Omega_2}{\frac{k^2}{2}(v-u)^2 dx}+C_2\left(\delta_2,\|\div p^*\|_{L^\infty(\Omega)}\right)\int\limits_{\Omega_2}{\frac{1}{2k^2}(\div p^*-\div y^*)^2 dx}\nonumber\\
&\leq \underline{\frac{1}{2}\vertiii{\nabla(u-v)}^2+\frac{1}{2}\vertiii{p^*-y^*}_*^2+D_F(v,-\Lambda^*p^*)+D_F(u,-\Lambda^*y^*)=M_{\oplus}^2(v,y^*)}\\
&\leq \frac{1}{2}\vertiii{\nabla(u-v)}^2+\frac{1}{2}\vertiii{p^*-y^*}_*^2\nonumber\\
&+C_1\left(\delta_1,\|u\|_{L^\infty(\Omega)}\right)\int\limits_{\Omega_2}{\frac{k^2}{2}(v-u)^2 dx}+\int\limits_{\Omega_2}{\frac{1}{2k^2}(\div y^*-\div p^*)^2 dx}\nonumber
\end{align}
where we have used that $p^*=\epsilon\Lambda u=\epsilon\nabla  u$.
Relation \eqref{Explicit_form_of_Functional_error_equality} shows that the computable majorant $M_\oplus^2(v,y^*)$
is bounded from below and above  by a multiple of one and the same error norm. Note that the left-hand side inequality
in~\eqref{Explicit_form_of_Functional_error_equality} is a stronger version of the left-hand inequality
in~\eqref{Explicit_form_of_Error_Estimate_with_Forcing_Functionals}.
Since $D_F(v,-\Lambda^*p^*)\ge 0$ and $D_F(u,-\Lambda^*y^*)\ge 0$ we also obtain a guaranteed bound on the error
in the combined energy norm:
\begin{align}\label{Upper_Bound_CEN_Error}
&\vertiii{\nabla( u-v)}^2+\vertiii{p^*-y^*}_*^2\leq 2M_{\oplus}^2(v,y^*).
\end{align}
From the pointwise equality
\begin{align}\label{Second_pointwise_error_equality}
&\mathrel{\phantom{=}}\frac{1}{\epsilon}|\epsilon\nabla v-y^*|^2=\frac{1}{\epsilon}\left|\epsilon\nabla(v-u)-(y^*- p^*)\right|^2\nonumber\\
&=\epsilon|\nabla(v-u)|^2+\frac{1}{\epsilon}|y^*- p^*|^2-2(y^*-p^*)\cdot\nabla(v- u),
\end{align}
after applying Young's inequality and integrating over $\Omega$, we obtain a lower bound for the error in combined energy norm:
\begin{align}\label{two_sided_bound_for_combined_energy_norm_error}
\begin{aligned}
&\frac{1}{2}\vertiii{\epsilon \nabla v-y^*}_*^2\leq \vertiii{\nabla (v- u)}^2+\vertiii{y^*- p^*}_*^2
\end{aligned}
\end{align}
\begin{remark}
Integrating \eqref{Second_pointwise_error_equality} over $\Omega$  we obtain the algebraic identity
\begin{align}\label{Prager_Synge_Equality}
\vertiii{\epsilon\nabla v-y^*}_*^2=\vertiii{\nabla(v-u)}^2 + \vertiii{y^*-p^*}_*^2 - 2\int\limits_{\Omega}{(y^*-p^*)\cdot\nabla(v-u) dx},
\end{align}
from which the Prager-Synge identity is derived.
Comparing the last relation with \eqref{Explicit_form_of_Functional_error_equality}, by using the fact that
$M_\oplus(v,y^*)^2=\frac{1}{2}\vertiii{\epsilon \nabla v-y^*}_*^2+D_F(v,-\Lambda^*y^*)$, we arrive at the relation
\begin{align}\label{relation_DFs_and_integral_term_in_Prager_Synge_extended_equality_2}
D_F(v,-\Lambda^*y^*)=D_F(v,-\Lambda^* p^*)+D_F( u,-\Lambda^*y^*)+\int\limits_{\Omega}{(y^*- p^*)\cdot \nabla (v-u) dx}.
\end{align}
From here, it is seen that if the integral on the right-hand side is small compared to the other terms, then the error
in $v$ and $\div y^*$ measured with $D_F(v,-\Lambda^*p^*)+D_F(u,-\Lambda^*y^*)$ is controlled mainly by the
computable term $D_F(v,-\Lambda^*y^*)$ in the majorant $M_\oplus^2(v,y^*)$. Moreover, \eqref{Prager_Synge_Equality}
enables us to give a practical estimation of the error in combined energy norm, which is very close to the real error in all
of the experiments that we have conducted.
\end{remark}
We end this section by presenting a near best approximation result. Contrary to the result in \cite[Theorem 6.2]{Chen2006b},
we do not make any restrictive assumptions on the meshes to ensure that the finite element approximations $u_h$ are uniformly
bounded in $L^\infty$ norm. In our considerations, let $V_h$ be a closed subspace of $H_0^1(\Omega)$ and let $u_h$ be the
unique minimizer of $J$ over $V_h$, which is also the unique solution of the Galerkin problem
\begin{align}\label{Discrete_Weak_Formulation}
&\text{Find } u_h\in V_h \text{ such that }\nonumber\\
&a(u_h,v)+\int\limits_{\Omega}{b(x,u_h+w)v dx}=(l,v),\,\text{ for all } v\in V_h\cap L^\infty(\Omega).
\end{align}
Then, using \eqref{primal_dual_parts_of_error_2} and the expression \eqref{DGvpStar} for $D_G(\Lambda v,p^*)$,
for any $v\in V_h$ we can write
\begin{align*}
&\mathrel{\phantom{=}}\vertiii{\nabla (u_h-u)}^2+2D_F(u_h,-\Lambda^*p^*)=2\left(J(u_h)-J(u)\right)\\
&\leq 2\left(J(v)-J(u)\right)=\vertiii{\nabla (v-u)}^2+2D_F(v,-\Lambda^*p^*).
\end{align*}
Since $2D_F(u_h,-\Lambda^*p^*)\ge 0$, we obtain
\begin{align}\label{quasi_optimal_a_priori_error_estimate_1}
\vertiii{\nabla (u_h-u)}^2\leq \inf\limits_{v\in V_h}{\bigg\{\vertiii{\nabla (v-u)}^2+\int\limits_{\Omega_2}{k^2(\sinh(v+w)-\sinh(u+w))^2 dx}\bigg\}}
\end{align}
where we have used \eqref{inequality_for_DFvMinusLambdaStarpStar}. Since we use the finite element method with $P_1$ Lagrange elements, let $V_h$ be the corresponding space where $h$ refers to the maximum element size. With $I_h(\varphi)$ we denote the Lagrange finite element interpolant of $\varphi\in C^0(\Omega)$. Using \eqref{quasi_optimal_a_priori_error_estimate_1} we can show unqualified convergence of the finite element approximations $u_h$ to $u$ when $h\to 0$. Let $\varepsilon >0$ and $\bar u\in C_0^\infty(\Omega)$ is such that $\|\nabla(\bar u-u)\|_{L^2(\Omega)}\leq \varepsilon$ and $\|\bar u\|_{L^\infty(\Omega)}\leq \|u\|_{L^\infty(\Omega)}+2$. Also, let $L$ be the Lipschitz constant in the inequality $|\sinh(s)-\sinh(t)|\leq L|s-t|$ for all $s,t \in \left[-2\|w\|_{L^\infty(\Omega_2)}-\overline e-2, 2\|w\|_{L^\infty(\Omega_2)}+\overline e+2\right]$.
Then by applying the triangle inequality together with Young's inequality, we obtain
\begin{align}\label{inequality_2}
&\mathrel{\phantom{=}}\vertiii{\nabla (u_h-u)}^2\leq 2\left(\vertiii{\nabla (I_h(\bar u)-\bar u)}^2+\vertiii{\nabla (\bar u-u)}^2\right)\\
&+2\left(\int\limits_{\Omega}{k^2(\sinh(I_h(\bar u)+w)-\sinh(\bar u+w))^2dx}+\int\limits_{\Omega}{k^2(\sinh(\bar u+w)-\sinh(u+w))^2dx}\right).\nonumber
\end{align}
For the first term in \eqref{inequality_2}, by assuming mesh regularity, we have 
\begin{align*}
&\vertiii{\nabla (I_h(\bar u)-\bar u)}^2+\vertiii{\nabla (\bar u-u)}^2\leq \epsilon_{\max}\left(C|\bar u|_2^2h^2+\varepsilon^2\right)
\end{align*}
where $|\bar u|_2$ denotes the $H^2$ seminorm of $\bar u$ and $C>0$ is a constant depending on the mesh regularity. Using the fact that $\|I_h(\bar u)\|_{L^\infty(\Omega)}\leq \|\bar u\|_{L^\infty(\Omega)}\leq \|u\|_{L^\infty(\Omega)}+2$, for the second term in \eqref{inequality_2} we obtain the upper bound
\begin{align*}
 &2k_{\max}^2L^2\left( \|I_h(\bar u)-\bar u\|_{L^2(\Omega)}^2+ \|\bar u- u\|_{L^2(\Omega)}^2\right)\\
 &\leq 2 k_{\max}^2L^2 C_F^2\left( \|\nabla(I_h(\bar u)-\bar u)\|_{L^2(\Omega)}^2+ \|\nabla(\bar u- u)\|_{L^2(\Omega)}^2 \right)\leq  2 k_{\max}^2L^2 C_F^2 \left(C|\bar u|_2^2 h^2+\varepsilon^2\right).
\end{align*}
This shows that the right-hand side of \eqref{inequality_2} can be made as small as desired provided that we choose $\varepsilon$ and $h$ small enough and therefore $\vertiii{\nabla (u_h-u)}\to 0$ when $h\to 0$. Moreover, \eqref{quasi_optimal_a_priori_error_estimate_1} can be also used to obtain qualified convergence of $u_h$ in energy norm under additional assumptions on the interface $\Gamma$, the meshes, and the regularity of $u$.

\section{Numerical results}
In the following we present numerical examples illustrating the functional a posteriori error equality \eqref{Explicit_form_of_Functional_error_equality} as well as the constituting terms of the equality. All numerical experiments are carried out in FreeFem++ developed and maintained by Frederich Hecht \cite{FreeFem} and all pictures are generated in VisIt \cite{VisIt}. We solve adaptively the homogeneous nonlinear Problem \eqref{PBE_special_form_regular_nonlinear_part} with $w:=w_{h_{ref}}=g-z_{h_{ref}}$ where $z_{h_{ref}}$ is a good Galerkin finite element approximation of the solution $z$ of
\begin{subequations}\label{PBE_Reference_Solution}
\begin{eqnarray}
-\nabla\cdot(\epsilon \nabla z)&=&-k^2\sinh(g)+l  \quad\text{in } \Omega_1\cup \Omega_2,\\
	\left[z\right]_\Gamma&=&0,\\
	 \left[\epsilon\frac{\partial z}{\partial n}\right]_\Gamma&=&0,\\
	 z&=&0,\quad \text{on } \partial \Omega,
	  \end{eqnarray}
 \end{subequations}
for given functions $g$ and $l$. We compare the accuracy of the adaptively computed solution $u_{h}$ of \eqref{PBE_special_form_regular_nonlinear_part} for $w=w_{h_{ref}}$ to the reference solution $z_{h_{ref}}$. The adaptive mesh refinement is based on the error indicator $\|\sqrt{2}\eta\|_{L^2(O_i)}$ where the function $\eta$ is defined in \eqref{integrand_of_the_Majorant} and $\eta^2$ is the integrand of the majorant $M_\oplus^2(v,y^*)$. The factor $\sqrt{2}$ accounts for the factor $2$ in \eqref{Upper_Bound_CEN_Error}. More precisely, we find approximations $u_h$ to the exact solution $u\in H_0^1(\Omega)$ of
\begin{align}\label{test_examples_setup}
\int\limits_{\Omega}{\epsilon\nabla u\cdot\nabla v dx}+\int\limits_{\Omega}{b(x,u+w_{h_{ref}})vdx}=\int\limits_{\Omega}{lv dx}=0,\,\forall v\in H_0^1(\Omega).
\end{align}
In all examples, we used piecewise constant parameters $\epsilon$ and $k$, and for $y^*\in H(\div;\Omega)$, we used a patchwise equilibrated reconstruction of the numerical flux $\epsilon \nabla u_{h}$ based on \cite{Braess_Schoberl_2006}. More precisely, we find $y^*$ in the Raviart-Thomas space $RT_0$ over the same mesh, such that its divergence is equal to the $L^2$ orthogonal projection of $k^2\sinh(u_h+w)+l$ onto the space of piecewise constants. 

Recall that 
\[
M_\oplus^2(v,y^*)=M_\oplus^2(v,p^*)+M_\oplus^2(u,y^*),
\]
where $M_\oplus^2(v,y^*)=\frac{1}{2}\vertiii{\epsilon\nabla v-y^*}_*^2+D_F(v,-\Lambda^*y^*)$ and $M_\oplus^2(v,p^*)=J(v)-J(u) = \frac{1}{2}\vertiii{\nabla(v-u)}^2+D_F(v,-\Lambda^*p^*)$ is the primal error, whereas $M_\oplus^2(u,y^*)=I^*(p^*)-I^*(y^*)=\frac{1}{2}\vertiii{y^*-p^*}_*^2+D_F(u,-\Lambda^*y^*)$ is the dual error. Further, we use $v$ for the approximate solution $u_{h}$ and $u$ for the reference solution $z_{h_{ref}}$ and define the efficiency index of the lower bound for the error in combined energy norm \eqref{two_sided_bound_for_combined_energy_norm_error} by 
\[
I_{\text{Eff}}^{\text{CEN,Low}}:=\frac{\frac{\sqrt{2}}{2}\vertiii{\epsilon\nabla v-y^*}_* }{\sqrt{\vertiii{\nabla(v-u)}^2 +\vertiii{y^*-p^*}_*^2 }}.
\]
Similarly, 
\[
I_{\text{Eff}}^{\text{CEN,Up}}:=\frac{\sqrt{2M_\oplus^2(v,y^*)}}{\sqrt{ \vertiii{\nabla(v-u)}^2+\vertiii{y^*-p^*}_*^2  }}
\]
defines the efficiency index of the upper bound \eqref{Upper_Bound_CEN_Error} for the error in combined energy norm,
\[
I_{\text{Eff}}^{\text{E}}:=\frac{\sqrt{2M_\oplus^2(v,y^*) } }{\vertiii{\nabla(v-u)}}
\]
defines the efficiency index of the upper bound for the error in energy norm, and  
\[
P_{\text{rel}}^{\text{CEN}}:=\frac{\vertiii{\epsilon\nabla v-y^*}_*}{ \sqrt{\vertiii{\nabla v}^2 +\vertiii{y^*}_*^2 }  }
\]
defines the practical estimate of the relative error in combined energy norm.
\subsection{Example 1 (2D)}
In the first example, the domain $\Omega$ is a square with a side $20$ with $\Omega_1$ being a regular 15-sided polygon with a radius of its circumscribed circle equal to $2$. The coefficients $\epsilon$ and $k$ are\\
\begin{minipage}{0.5\textwidth}
\begin{align*}
&\epsilon(x)=\left\{
\begin{aligned}
&\epsilon_1=1,\quad x\in \Omega_1,\\
&\epsilon_2=100,\quad x\in\Omega_2.
\end{aligned}
\right.
\end{align*}
\end{minipage}
\begin{minipage}{0.5\textwidth}
\begin{align*}
&k(x)=\left\{
\begin{aligned}
&k_1=0.15,\quad x\in \Omega_1,\\
&k_2=0.4,\quad x\in\Omega_2.
\end{aligned}
\right.
\end{align*}
\end{minipage}
and 
\begin{align*}
g=L\left(  \exp\left(  -b_1 \left( \frac{(x_1-c_1))^2}{\sigma_1^2}-1\right) \right)-\exp\left(  -b_2 \left( \frac{(x_2-c_2)^2}{\sigma_2^2}-1\right) \right) \right),
\end{align*}
$l=0$, where $b_1=2=b_2=2$, $c_1=-1$, $c_2=6$, $\sigma_1=\sigma_2=1.5$, $L=0.8$.
The reference solution $z_{h_{ref}}$ is computed on an adapted mesh with $50\,086\,142$ triangles.
Note that $k^2=0.0225$ in $\Omega_1$ and $k^2=0.16$ in $\Omega_2$. The mesh adaptation is done
with the built in function ''adaptmesh'' of freefem++.
The localized error indicator $\|\sqrt{2}\eta\|_{L^2(O_i)}$, computed on each vertex patch $O_i$ of
the mesh, is compared to its average value over all patches and the local mesh size is divided by
two if this average is smaller then the local value.

Table~\ref{Example1_Table1} illustrates the main error
identity~\eqref{important_relation_satisfied_by_the_majorant_M} and the convergence of its constituent
parts. Further, it is seen that the dual error $2M_\oplus(u,y^*)$ dominates the primal error in this example.
This is due to the fact that the term $2D_F(u,-\Lambda^*y^*)$, measuring the error in
$\div y^*$ (cf. \eqref{inequality_for_DFuMinusLambdaStaryStar}
and~\eqref{Lower_Quadratic_Bound_for_DFuMinusLambdaStaryStar}), is much larger than
$\vertiii{\nabla( v-u)}^2+D_F(v,-\Lambda^*p^*)$, where $D_F(v,-\Lambda^*p^*)$ behaves like
$\|v-u\|_{L^2(\Omega_2)}^2$ (cf. \eqref{inequality_for_DFvMinusLambdaStarpStar}
and~\eqref{Upper_Quadratic_Bound_for_DFvMinusLambdaStarpStar}). 
 As we mentioned earlier, for $y^*$ we use a partially equilibrated reconstruction of the numerical
 flux $\epsilon\nabla v$ which is the reason why the integral term in~\eqref{Prager_Synge_Equality}
 is negligible compared to the combined energy norm of the error. This fact is confirmed by the values
 of the efficiency index of the lower bound \eqref{two_sided_bound_for_combined_energy_norm_error}.

{\small
\begin{center}
\captionof{table}{Example 1 (2D)} \label{Example1_Table1}
\vspace{1ex} 
\begin{tabular}{ |c|c|c|c|c|c|c| }
\hline
\multicolumn{7}{|c|}{Example 1 (2D): $k_1=0.15,\,k_2=0.4,\,\epsilon_1=1,\,\epsilon_2=100$} \\
\hline
\#elts &$\frac{\|v-u\|_0}{\|u\|_0}[\%]$ &$\frac{\vertiii{\nabla(v-u)}}{\vertiii{\nabla u}}[\%]$ & $\frac{\vertiii{y^*-p^*}_*}{\vertiii{p^*}_*}[\%]$ & $2M_\oplus^2(v,y^*)$& $2M_\oplus^2(v,p^*)$ &$2M_\oplus^2(u,y^*)$\\
\hline
196 		      &15.0077 	 &51.5582 		&86.1021			& 1778.14		&66.5980 	        &1711.54\\
347 		      &5.69339         &30.8534 	        &41.7241			& 703.594		&20.7780 		&682.816\\
630 		      &4.20384         &21.7715 		&31.4858			&217.719		&10.2201		&207.498 \\
1315     	      &2.39552	 &15.8532 		&23.1244			& 76.8018		&5.37574		&71.4261\\
2865            &1.87075         & 11.7353 		&17.1655 		        &33.9310			&2.94414		&30.9869\\
5938            &0.64611 	 &7.93001 		&11.4692			&16.0812		&1.33874 		&14.7425\\
12006          &0.36985	 &5.64786  		& 8.23544			&7.75232		&0.67872		&7.07360\\
24571          &0.16023 	 &3.94241 		& 5.76054			&3.85268		&0.33039		&3.52229\\
48483          &0.08909	 &2.80265 		& 4.09366			&1.90043		&0.16682		&1.73361\\
97423          &0.03961	 &1.97875 		& 2.88455			&0.96275		&0.08304 		&0.87970\\
192905       &0.02230 	 &1.39832 		& 2.03200				&0.47524		&0.04136 		&0.43388\\
386185       &0.01015 	 &0.99471 		&1.44616			&0.24134		&0.02082   	&0.22052\\
\hline
\end{tabular}
\end{center}
}

{\small
\begin{center}
\captionof{table}{Example 1 (2D)} \label{Example1_Table2} 
\vspace{1ex}
\begin{tabular}{ |c|c|c|c|c| }
\hline
\multicolumn{5}{|c|}{Example 1 (2D): $k_1=0.15,\,k_2=0.4,\,\epsilon_1=1,\,\epsilon_2=100$} \\
\hline
\#elts & $\vertiii{\nabla(v-u)}^2$ &$\vertiii{y^*-p^*}_*^2$ & $2D_F(v,-\Lambda^*p^*)$ &$2D_F(u,-\Lambda^*y^*)$\\
\hline
196 		   &56.5057 	 &157.588 				&10.0923			&1553.95		\\
347 		   &20.2350     &37.0058 	     			&0.54296			&645.811		\\
630 		   &10.0756     &21.0729				&0.14450			&186.425		\\
1315     	   &5.34235	 &11.3668				&0.03338			&60.0593		\\
2865            &2.92742     &6.26338 	 		&0.01671			&24.7235		\\
5938            &1.33673 	 &2.79619 			&0.00200			&11.9462		\\
12006          &0.67805	 &1.44169 			&0.00067			&5.63191		\\
24571          &0.33038 	 &0.70538				&0.00001  		&2.81691		\\
48483          &0.16696	 &0.35622 			&0.00000			&1.37739		\\
97423          &0.08323	 &0.17687 			&0.00000			&0.70283		\\
192905       &0.04156 	 &0.08777 			&0.00000			&0.34611		\\
386185       &0.02103 	 &0.04445 			&0.00000			&0.17606		\\
\hline
\end{tabular}
\end{center}
}

{\small
\begin{center}
\captionof{table}{Example 1 (2D)} \label{Example1_Table3} 
\vspace{1ex}
\begin{tabular}{ | c|c|c|c|c|c|c| }
\hline
\multicolumn{7}{|c|}{Example 1 (2D): $k_1=0.15,\,k_2=0.4,\,\epsilon_1=1,\,\epsilon_2=100$} \\
\hline
\#elts & $\frac{D_F(v,-\Lambda^*y^*)}{M_\oplus^2(v,y^*)}[\%]$ & $I_{\text{Eff}}^{\text{CEN,Low}}$ &
$I_{\text{Eff}}^{\text{CEN,Up}}$ & $I_{\text{Eff}}^{\text{E,Up}}$ & $P^{\text{CEN}}_{\text{rel}}~[\%]$ &
\parbox{2.4 cm}{True rel. error in $\text{CEN}~[\%]$}
\\
\hline
196 		      &89.0701  	 &0.67371 		&2.88191			&5.60966		&74.6973	&70.9641	\\
347 		      &92.4942         &0.67919 	        &3.50597			&5.89671		&36.2638	&36.6935	\\
630 		      &85.9525         &0.70066		&2.64380				&4.64848		&27.1574	&27.0680		\\
1315     	      &78.2616	 	&0.70681 		&2.14392			&3.79158		&19.9383	&19.8250		\\
2865            &72.8992         &0.70729 		&1.92142		        &3.40452		&14.7523	&14.7032	\\
5938            &74.3009		 &0.70708 		&1.97256			& 3.46846		&9.87419	&9.85973	\\
12006          &72.6473		 &0.70722 		&1.91238			&3.38130			&7.06762	&7.06119	\\
24571          &73.1176 	 &0.70708		&1.92864			&3.41485		&4.93753	&4.93591	\\
48483          &72.4826		 &0.70694 		&1.90588			&3.37371		&3.50789	&3.50805	\\
97423          &73.0084 	 &0.70678		&1.92392			&3.40108		&2.47256	&2.47347	\\
192905       &72.8486	 	 &0.70629 		&1.91692			&3.38145		&1.74226	&1.74418	\\
386185       &72.9912  	 &0.70546 		&1.91972			&3.38748		&1.23829	&1.24114	\\
\hline
\end{tabular}
\end{center}
}
\vspace{1ex}

 In Table~\ref{Example1_Table3} we can see that $I_{\text{Eff}}^{\text{CEN,Low}}$ is approximately
 equal to $0.7071\approx \frac{\sqrt{2}}{2}$. The value of the efficiency index with respect to the combined
 energy norm and the value of the ratio $D_F(v,-\Lambda^*y^*)/M_\oplus^2(v,y^*)$ are also coupled in
 the sense that if we have only one of these two quantities, we can estimate the other one by using the
 main error equality~\eqref{Explicit_form_of_Functional_error_equality}. This estimation is accurate
 because the integral term
 in~\eqref{relation_DFs_and_integral_term_in_Prager_Synge_extended_equality_2} is very close
 to zero and therefore $D_F(v,-\Lambda^*y^*) \approx D_F(v-\Lambda^*p^*)+D_F(u-\Lambda^*y^*)$.
 One more consequence of using a partially equilibrated flux is that we obtain a very accurate practical
 estimate of the absolute and relative error in combined energy norm as illustrated in the last two columns
 of Table~\ref{Example1_Table3}.
 
Figure~\ref{run35_2D_Th_Adapted_With_Functional_Error_Indicator0009} depicts a mesh that
is a part of a sequence of meshes obtained by mesh adaptation using the localized functional
error indicator~$\|\sqrt{2}\eta\|_{L^2(O_i)}$.
Figure~\ref{run35_2D_Th_Adapted_With_EpsGradvMinusyStar_Error_Indicator0009} depicts a
mesh with approximately the same number of elements but obtained by mesh adaptation using
the error indicator $\vertiii{\epsilon\nabla v-y^*}_{*(O_i)}$.
The mesh in Figure~\ref{run35_2D_Th_Adapted_With_Functional_Error_Indicator0009} is refined
mainly where the error in $\div y^*$ is the dominant part of the error
$M_\oplus^2(v,-\Lambda^*p^*)+M_\oplus^2(u,-\Lambda^*y^*)$. On the other hand, the mesh
in Figure~\ref{run35_2D_Th_Adapted_With_EpsGradvMinusyStar_Error_Indicator0009} is refined
most around the extrema of the solution.
Figure~\ref{Th_Adapted_With_EpsGradvMinusyStar_Indicator_AdptLevel_7} depicts the minimal
set of elements $K$ of a mesh $T_h$ that contains at least $30\%$ of the total indicated error
$\sum_{K\in T_h}{\vertiii{\epsilon\nabla v-y^*}_{*(K)} }$ (greedy algorithm with a bulk factor of $0.3$),
where $T_h$ is part of the same sequence as the mesh illustrated in
Figure~\ref{run35_2D_Th_Adapted_With_EpsGradvMinusyStar_Error_Indicator0009}. 

Figure \ref{run35_greedy1_level_02_marked_true_nonlinear_error_bulkfactor_0p50000} depicts
the elements marked by the greedy algorithm using a bulk factor of $0.5$ and employing the true
error $\sqrt{2M_\oplus^2(v,p^*)+2M_\oplus^2(u,y^*)}$ as indicator.
Figure~\ref{run35_greedy1_level_02_Differently_marked_by_FuncInd_compared_to_true_nonlinear_error_bulkfactor_0p50000}
depicts elements which are marked additionally or fail to be marked by the same greedy algorithm
when employing the functional error indicator $\|\sqrt{2}\eta\|_{L^2(O_i)}$ for the same bulk factor.
The ratio of the number of these differently marked elements, that is, elements which are marked by one
of the two methods but not by the other one, and the total number of elements is $0.022$ and the ratio
of the number of differently marked elements to the number of marked elements using the true error
is $0.048$ (see Table~\ref{Example1_Table_Marked_Elements}). 
Comparing the indicated error and the true error elementwise, one finds that the error indicator generated by the majorant $M_\oplus^2(v,y^*)$
reproduces the local distribution of the error with a very high accuracy.
This is also confirmed by Figure~\ref{Errors_FuncInd_vs_FullNonlinearNorm_Refinement} where it can be seen that
all error measures are almost identical in both cases of adaptive mesh refinement. Mesh adaptation based on the
functional error indicator $\|\sqrt{2}\eta\|_{L^2(O_i)}$ instead of the
error indicator $\vertiii{\epsilon\nabla v-y^*}_{*(O_i)}$ (see Figure~\ref{Majorant_Errors_vs_DOFs}) yields
approximately twice smaller efficiency indexes in energy and combined energy norms and approximately
twice smaller values for the full error $M_\oplus^2(v,p^*)+M_\oplus^2(u,y^*)$ on meshes with a comparable
number of elements. The reason for the higher efficiency indexes is that no adaptive control is applied on the
nonlinear part of the error measure in \eqref{Explicit_form_of_Functional_error_equality}, and consequently,
the ratio $D_F(v,-\Lambda^*y^*)/M_\oplus^2(v,y^*)$ is increasing, reaching values close to $100\%$ on fine
meshes. However, the error in $\vertiii{\nabla(v-u)}$ and $\vertiii{y^*-p^*}_*$ might be a little higher in the case
of the functional error indicator $\|\sqrt{2}\eta\|_{L^2(O_i)}$. For example, on the mesh from
Figure~\ref{Th_Adapted_With_EpsGradvMinusyStar_Indicator_AdptLevel_7} with $24\,122$ elements,
$M_\oplus^2(v,p^*)+M_\oplus^2(u,y^*)=3.8314$, $\vertiii{\nabla(v-u)} = 0.4674$, $\vertiii{y^*-p^*}_*=0.6540$,
whereas on a mesh with $24\,571$ elements from the sequence adapted with the indicator
$\|\sqrt{2}\eta\|_{L^2(O_i)}$, we obtained a value of $1.9263$ for $M_\oplus^2(v,y^*)$, and $0.574791$
and $0.8399$ for $\vertiii{\nabla(v-u)}$ and $\vertiii{y^*-p^*}_*$, respectively. This shows that by reducing
the error in $\div y^*$ the functional error indicator $\|\sqrt{2}\eta\|_{L^2(O_i)}$ provides a better approximation
for the primal and dual problem together.
        \begin{figure}[!htb]
    \centering
    \begin{minipage}{1\textwidth}
        \centering
        \captionsetup{width=0.8\linewidth}
      \includegraphics[width=0.8\linewidth]{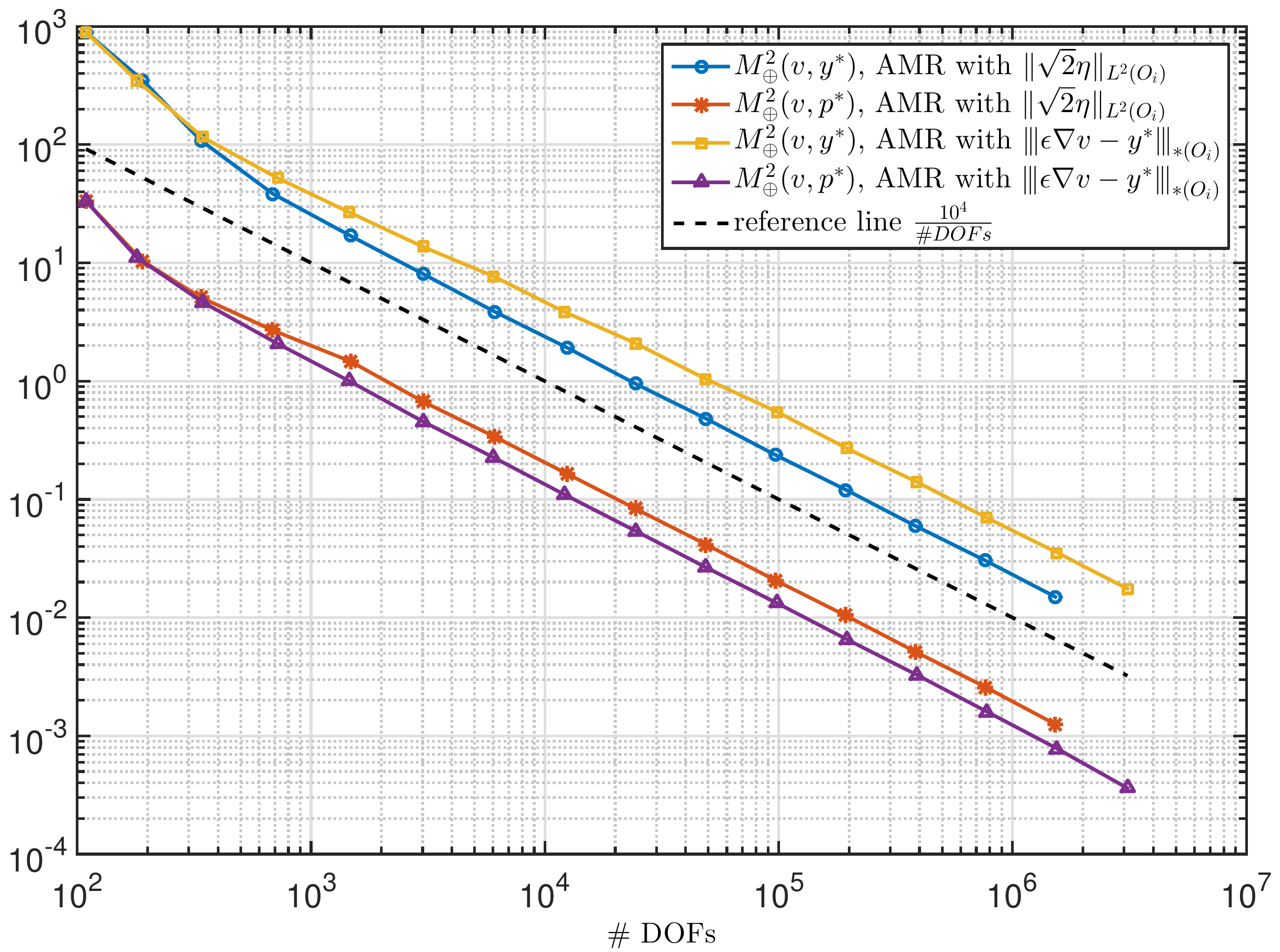}
        \caption{{\scriptsize Comparison of errors for AMR based on the functional error indicator
        $\|\sqrt{2}\eta\|_{L^2(O_i)}$ versus AMR based on the indicator $\vertiii{\epsilon\nabla v-y^*}_{*(O_i)}$. }}
        \label{Majorant_Errors_vs_DOFs}
    \end{minipage}%
    \end{figure}
    \begin{figure}
    \begin{minipage}{1\textwidth}
        \centering
        \captionsetup{width=0.9\linewidth}
        \includegraphics[width=0.9\linewidth]{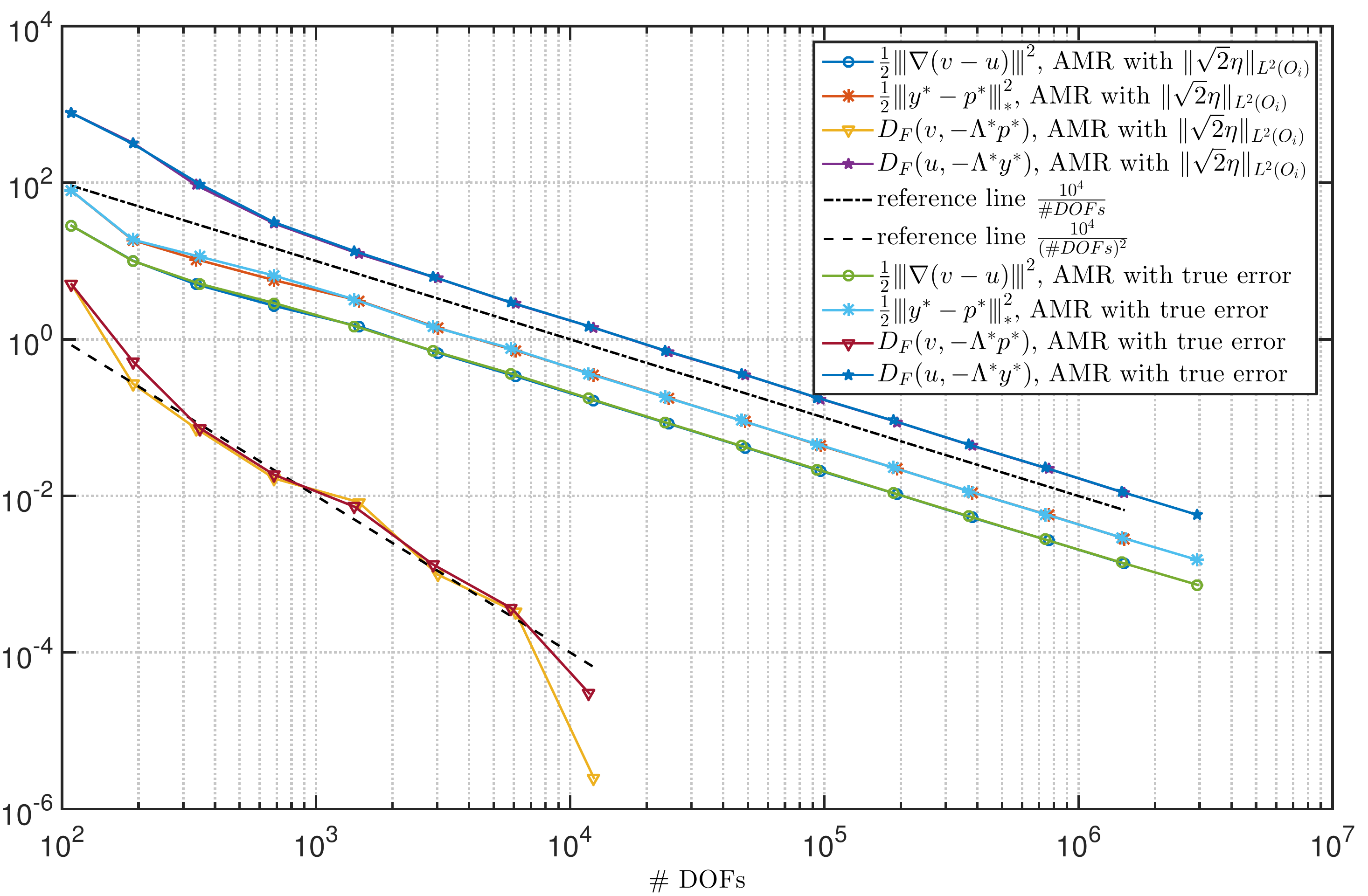}
        \caption{{\scriptsize Comparison of errors for AMR based on the functional error indicator
        $\|\sqrt{2}\eta\|_{L^2(O_i)}$ versus AMR based on the indicator generated by the true error
        $\sqrt{2M_\oplus^2(v,p^*)+2M_\oplus^2(u,y^*)}$.}}
        \label{Errors_FuncInd_vs_FullNonlinearNorm_Refinement}
    \end{minipage}
\end{figure}
\begin{figure}[!htb]
    \centering
    \begin{minipage}{0.5\textwidth}
        \centering
        \captionsetup{width=.9\linewidth}
        \includegraphics[width=1\linewidth]{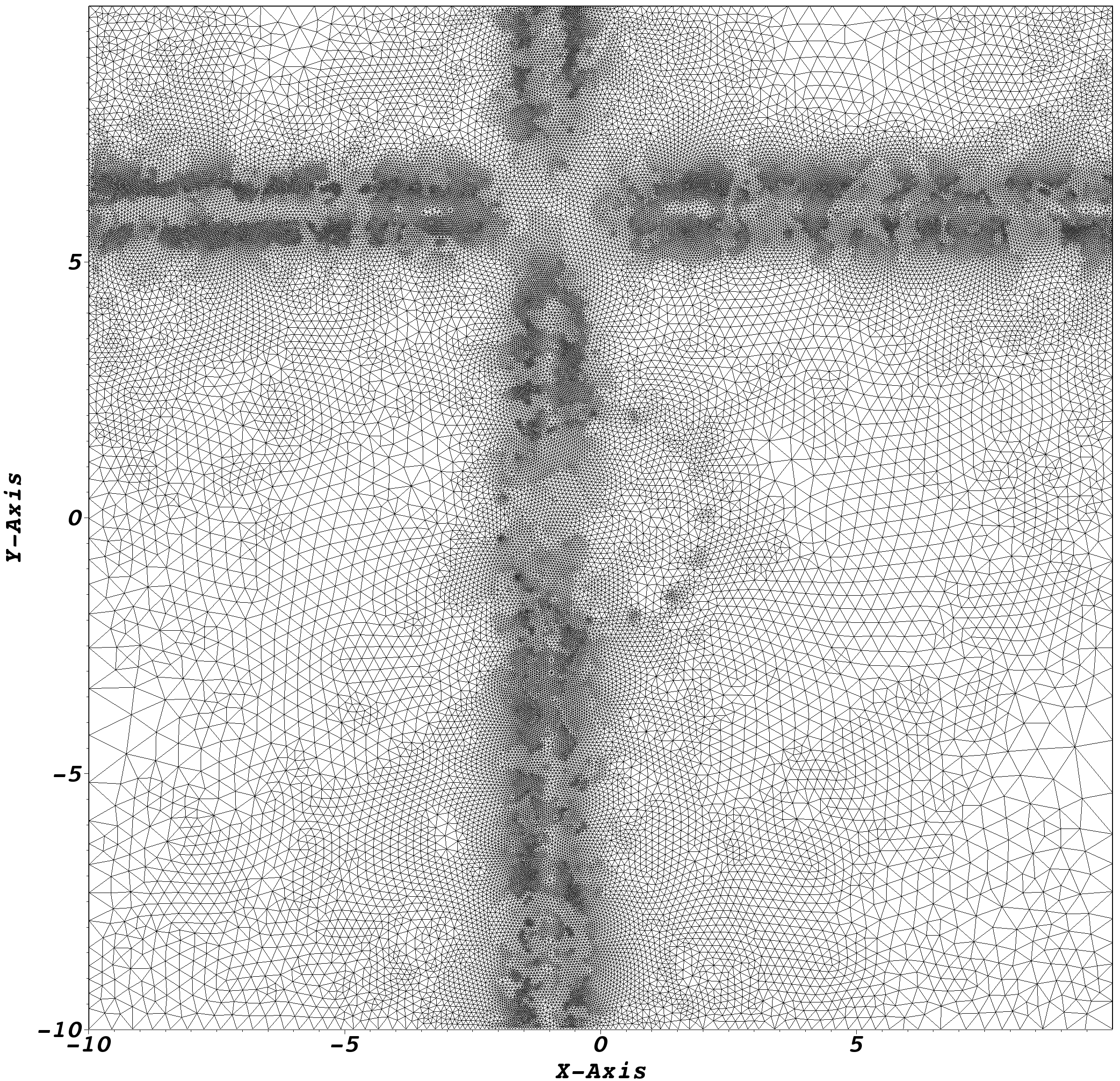}
        \caption{{\scriptsize Mesh on the $9$th level of AMR ($97\,423$ elements) based on the error
        indicator $\|\sqrt{2}\eta\|_{L^2(O_i)}$ with flux equilibration for~$y^*$.}}
        \label{run35_2D_Th_Adapted_With_Functional_Error_Indicator0009}
    \end{minipage}%
    \begin{minipage}{0.5\textwidth}
        \centering
        \captionsetup{width=.9\linewidth}
        \includegraphics[width=1\linewidth]{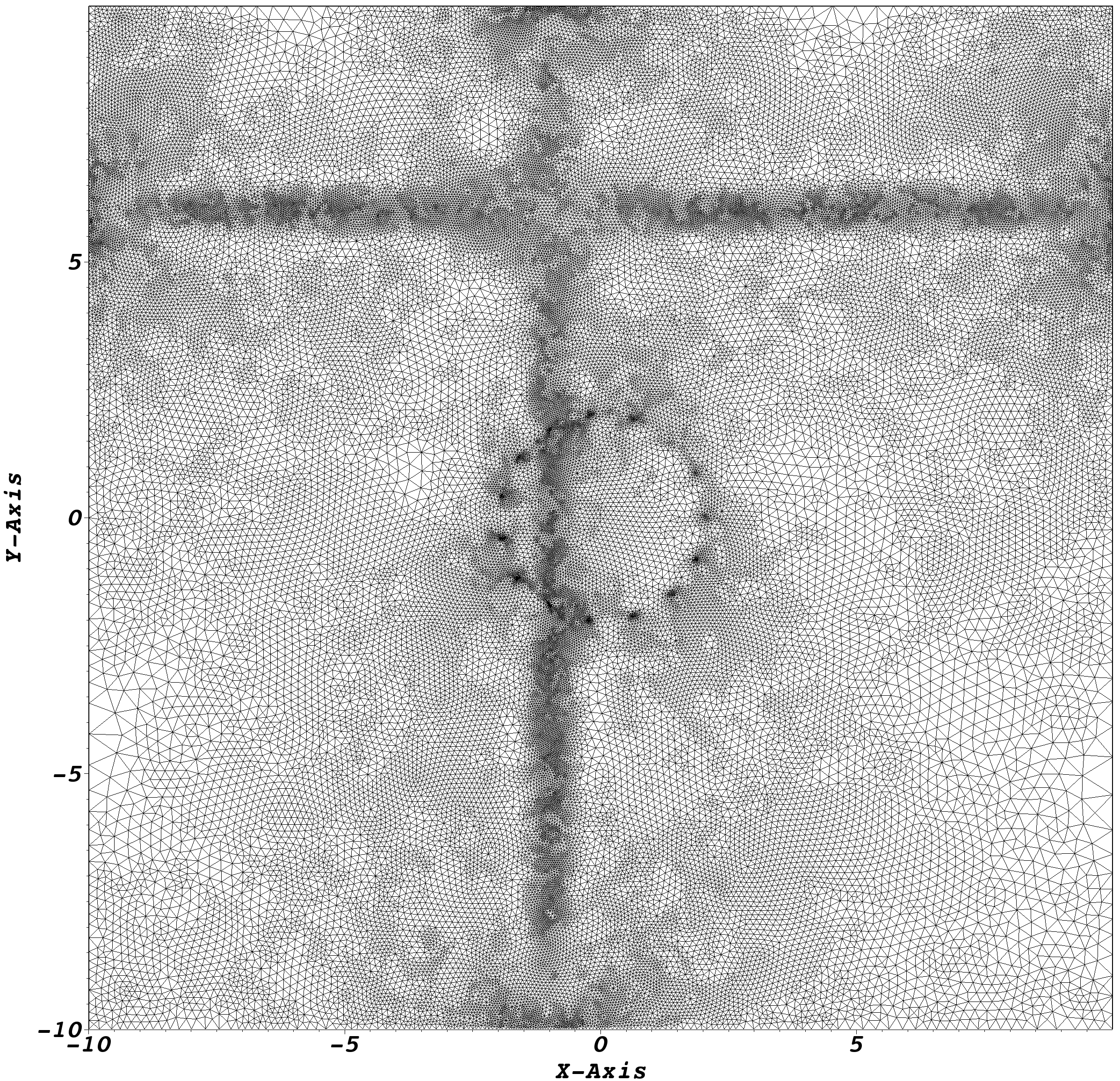}
        \caption{{\scriptsize Mesh on level $9$th level of AMR ($97\,353$ elements) based on the error
        indicator $\vertiii{\epsilon\nabla v-y^*}_{*(O_i)}$ with flux equilibration for~$y^*$.}}
        \label{run35_2D_Th_Adapted_With_EpsGradvMinusyStar_Error_Indicator0009}
    \end{minipage}
\end{figure}
{\small
\begin{center}
\captionof{table}{Example 1 (2D)} \label{Example1_Table_Marked_Elements} 
\vspace{1ex}
\begin{tabular}{ |c|c|c|c| }
\hline
\multicolumn{4}{|c|}{Example 1 (2D): $k_1=0.15,\,k_2=0.4,\,\epsilon_1=1,\,\epsilon_2=100$} \\
\hline
\#elts &\parbox{2.8cm}{\#marked elts\\ with true error}& \parbox{2.5cm}{\#differently\\ marked elts}&
\parbox{4.2cm}{differently marked elts \\ in \% of all mesh elts }\\
\hline
196 		     	 &62 	     &6 			&3.06122			\\
347 		      &150       &10 	     &2.88184			\\
630 		      &288       &14 		&2.22222			\\
1315     	      &632       &39 		&2.96578			\\
2865         	&1439      &113 		&3.94415 		     \\
5938       		&2949	     &216 		&3.63759			\\
12006       	&5981      &534  		&4.44778			\\
24571       	&12099    &961 		&3.91111			\\
48483       	&24194    &2233 		&4.60574			\\
97423       	&47784    &4012 		&4.11812			\\
\hline
\end{tabular}
\end{center}
}
\begin{figure}[h!t]
    \centering
    \begin{minipage}[t]{0.52\textwidth}
        \centering
        \captionsetup{width=0.9\linewidth}
        \includegraphics[width=1\linewidth,valign=t]{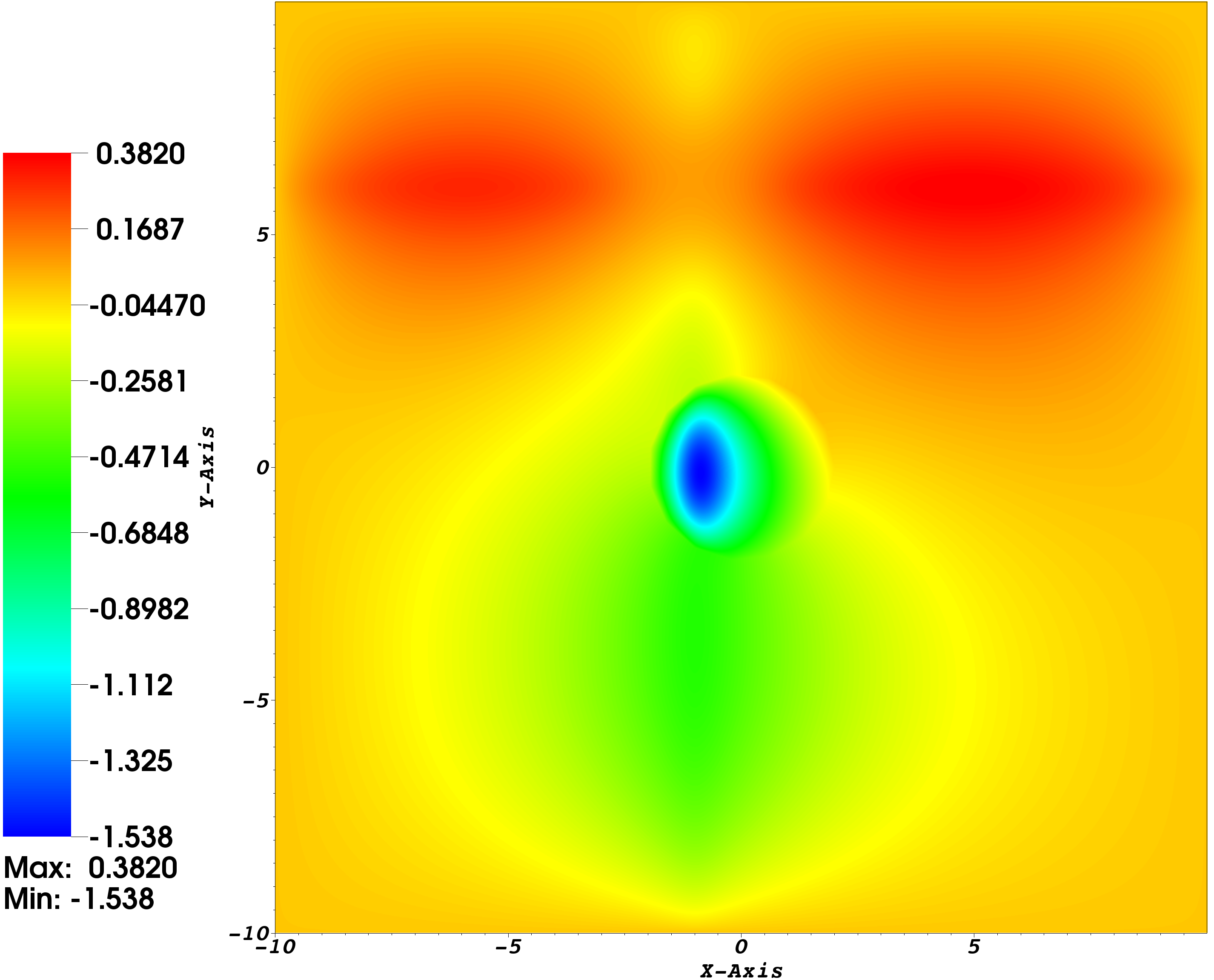}
        \caption{{\scriptsize Reference solution for Example 1 (2D).}}
        \label{Reference_Solution}
    \end{minipage}%
    \begin{minipage}[t]{0.434\textwidth}
        \centering
        \captionsetup{width=.9\linewidth}
        \includegraphics[width=1\linewidth,valign=t]{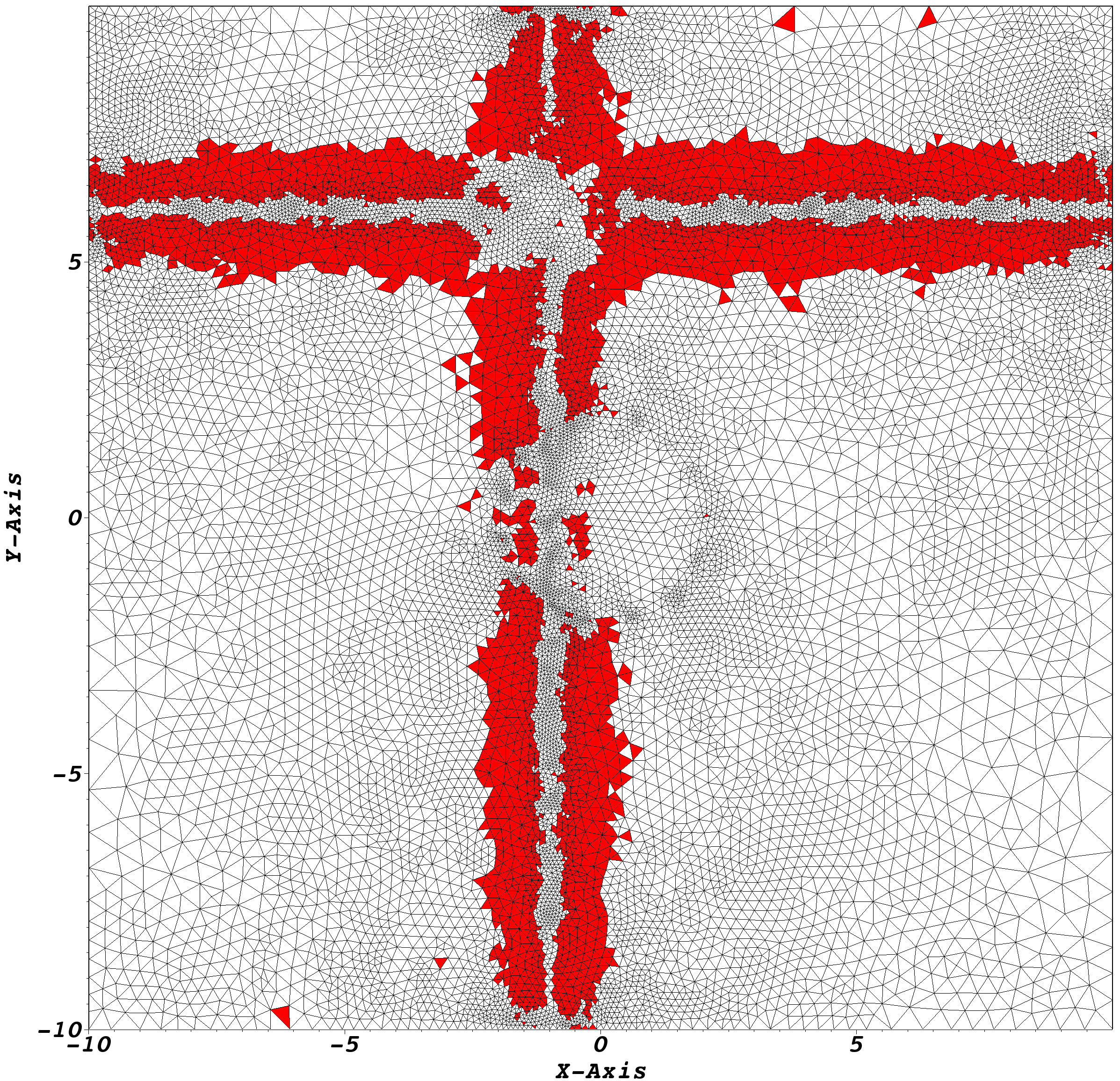}
        \caption{{\scriptsize Mesh on the $7$th level of AMR ($24\,122$ elements) based on the error indicator
        $\vertiii{\epsilon\nabla v-y^*}_{*(O_i)}$ with flux equilibration for~$y^*$.
        Marked elements using the error indicator $\|\sqrt{2}\eta\|_{L^2(K)}$ applying greedy algorithm with bulk factor~$0.3$.}}
        \label{Th_Adapted_With_EpsGradvMinusyStar_Indicator_AdptLevel_7}
    \end{minipage}
\end{figure}
\vspace{1ex}
\begin{figure}[h!t]
    \centering
    \begin{minipage}[t]{0.5\textwidth}
        \centering
        \captionsetup{width=.9\linewidth}
        \includegraphics[width=1\linewidth]{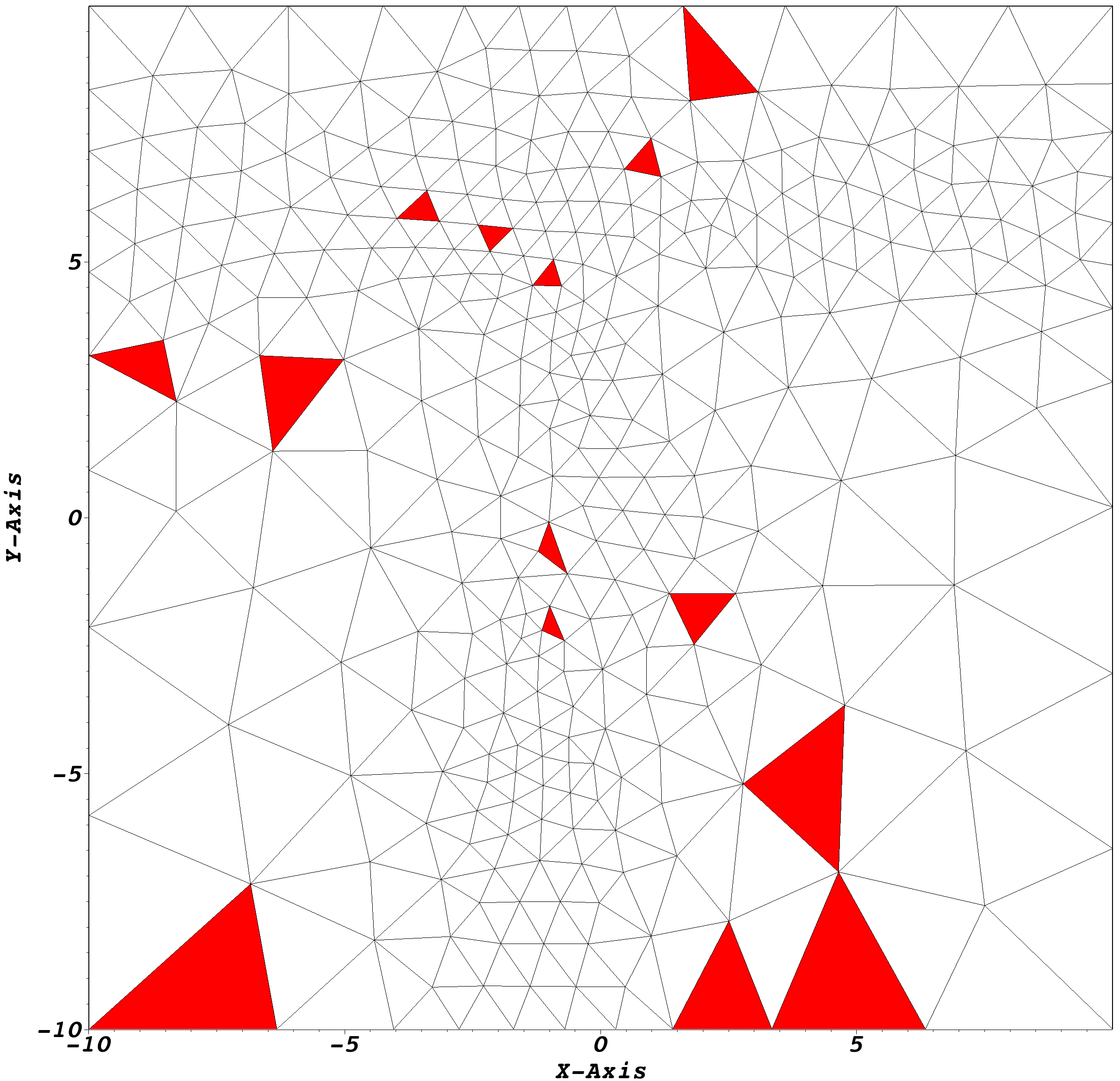}
 \caption{{\scriptsize Mesh on the $2$nd level of AMR ($630$ elements) based on the error indicator
 $\|\sqrt{2}\eta\|_{L^2(O_i)}$ with flux equilibration for~$y^*$. Differently marked elements using the
 error indicator $\|\sqrt{2}\eta\|_{L^2(K)}$ as compared to the elements marked when using the true
 error $M_\oplus^2(v,-\Lambda^*p^*)+M_\oplus^2(u,-\Lambda^*y^*)$ applying greedy algorithm with
 bulk factor~$0.5$.}}
        \label{run35_greedy1_level_02_Differently_marked_by_FuncInd_compared_to_true_nonlinear_error_bulkfactor_0p50000}
    \end{minipage}%
    \begin{minipage}[t]{0.5\textwidth}
        \centering
        \captionsetup{width=.9\linewidth}
        \includegraphics[width=1\linewidth]{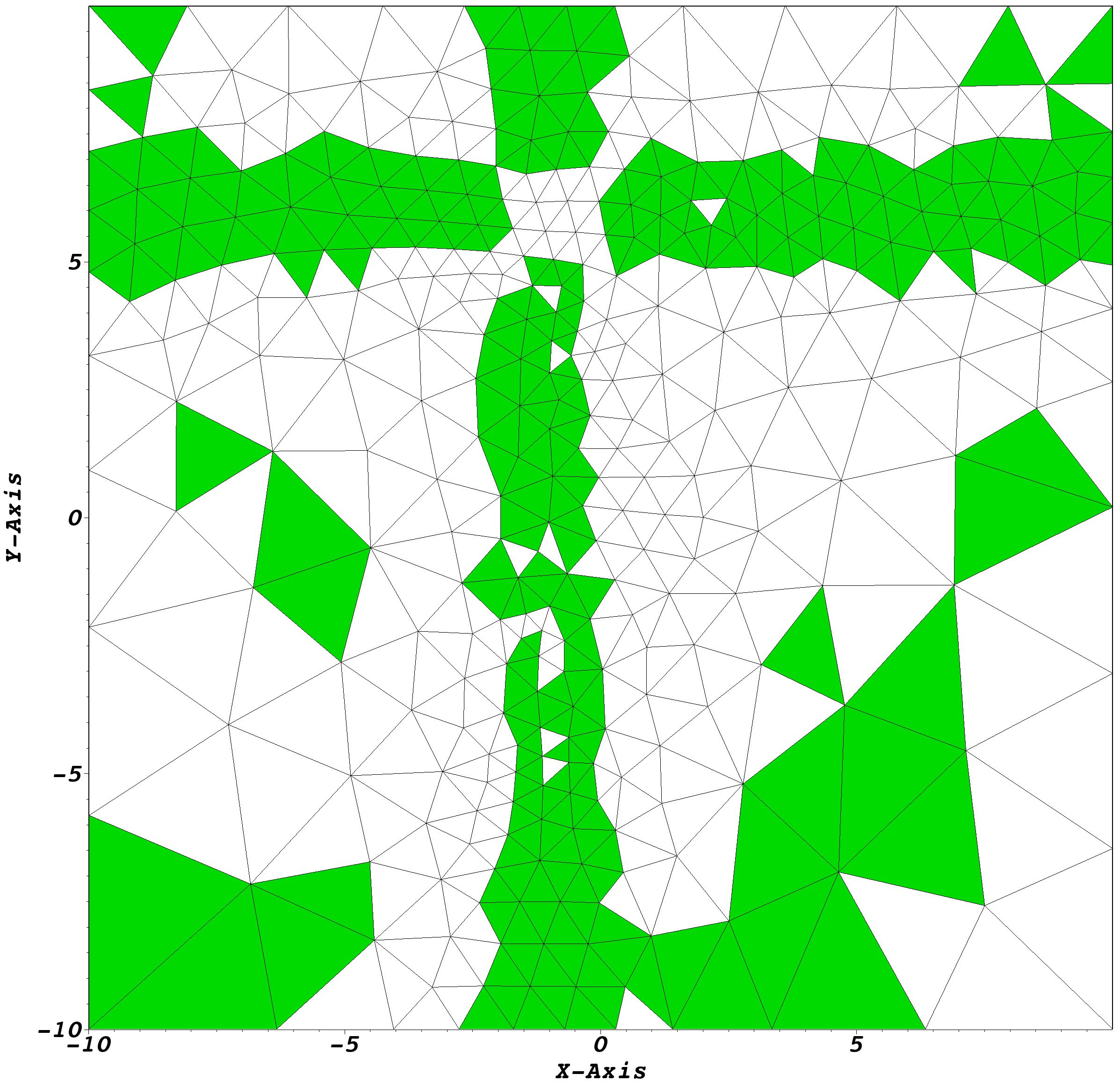}
        \caption{{\scriptsize Mesh on the $2$nd level of AMR  ($630$ elements) based on the error indicator
        $\|\sqrt{2}\eta\|_{L^2(O_i)}$ with flux equilibration for~$y^*$. Marked elements using the true error
        $\sqrt{2M_\oplus^2(v,p^*)+2M_\oplus^2(u,y^*)}$ applying greedy algorithm with bulk factor~$0.5$.}}
        \label{run35_greedy1_level_02_marked_true_nonlinear_error_bulkfactor_0p50000}
    \end{minipage}
\end{figure}

In the following we want to demonstrate that flux equilibration is indeed an important subtask
to make the proposed error bounds reliable and efficient. For this purpose, we use a simple
global gradient averaging procedure, i.e. project the numerical flux $\epsilon\nabla v\in L^2(\Omega)$
onto the subspace $\left[V_h\right]^2$, where $V_h$ is the finite element space of continuous piecewise
linear functions. We then solve adaptively Example 1 once by applying the functional error indicator
$\|\sqrt{2}\eta\|_{L^2(O_i)}$ and once by applying the error indicator $\vertiii{\epsilon\nabla v-y^*}_{*(O_i)}$.
Figure~\ref{run35_2D_Th_Adapted_With_Functional_Error_Indicator_Using_Gradient_Averaging} shows
an adapted mesh with $563\,965$ elements which is a part of a sequence of meshes obtained by applying
the functional error indicator with gradient averaging for $y^*$
while Figure~\ref{run35_2D_Th_Adapted_With_EpsGradvMinusyStar_Error_Indicator_Using_Gradient_Averaging_for_yStar}
shows a mesh with $444\,092$ elements which is part of a sequence of meshes adapted using the second indicator
with gradient averaging for $y^*$. It can be seen by comparing with the results based on flux equilibration for~$y^*$
that the mesh in $\Omega_2$ close to the interface $\Gamma$ is refined too much for both error indicators.
Apart from that, the meshes on
Figures~\ref{run35_2D_Th_Adapted_With_EpsGradvMinusyStar_Error_Indicator_Using_Gradient_Averaging_for_yStar}
and~\ref{run35_2D_Th_Adapted_With_EpsGradvMinusyStar_Error_Indicator0009} look quite similar,
unlike the meshes on Figures~\ref{run35_2D_Th_Adapted_With_Functional_Error_Indicator_Using_Gradient_Averaging}
and~\ref{run35_2D_Th_Adapted_With_Functional_Error_Indicator0009}.
For meshes with similar number of elements, by applying the indicator $\vertiii{\epsilon\nabla v-y^*}_{*(O_i)}$
using flux equilibration versus gradient averaging we obtained around $30 \%$ larger values for the error
$\vertiii{\nabla(v-u)}$ and $60\%$ larger values for the error $\vertiii{y^*-p^*}_*$. The difference in the errors
when applying the functional indicator $\|\sqrt{2}\eta\|_{L^2(O_i)}$ with flux equilibration versus with gradient
averaging for $y^*$ is even more drastic--between $40\%$ and $180\%$ larger error $\vertiii{\nabla(v-u)}$
and between  $64\%$ and $66\%$ larger error $\vertiii{y^*-p^*}_*$ for meshes with between $21\,528$ and
$563\,965$ elements.
In both cases we obtained an increasing sequence of efficiency indexes with respect to energy and combined
energy norms reaching values of $133$ and $107$ with the functional error indicator on a mesh with $2\,089\,022$
elements, and $ 570$ and $269$ with the error indicator $\vertiii{\epsilon\nabla v-y^*}_{*(O_i)}$ on a mesh with
$2\,954\,218$ elements. This is due to the fact that the nonlinear term $D_F(u,-\Lambda^*y^*)$, which measures
the error in $\div y^*$ (see \eqref{inequality_for_DFuMinusLambdaStaryStar} and
\eqref{Lower_Quadratic_Bound_for_DFuMinusLambdaStaryStar}),
dominates the other terms in the nonlinear measure $M_\oplus^2(v,p^*)+M_\oplus^2(u,y^*)$ for the error,
reaching more than $99.99 \%$ of it in both cases. In both experiments with gradient averaging for $y^*$,
increasing values of $D_F(u,-\Lambda^*y^*)$ are in correspondence with increasing error
$\|\div y^*-\div p^*\|_{L^2(\Omega)}$ and increasing efficiency indexes.
\vspace{1ex}

\begin{figure}[!ht]
    \centering
    \begin{minipage}{0.5\textwidth}
        \centering
        \captionsetup{width=.9\linewidth}
        \includegraphics[width=1\linewidth]{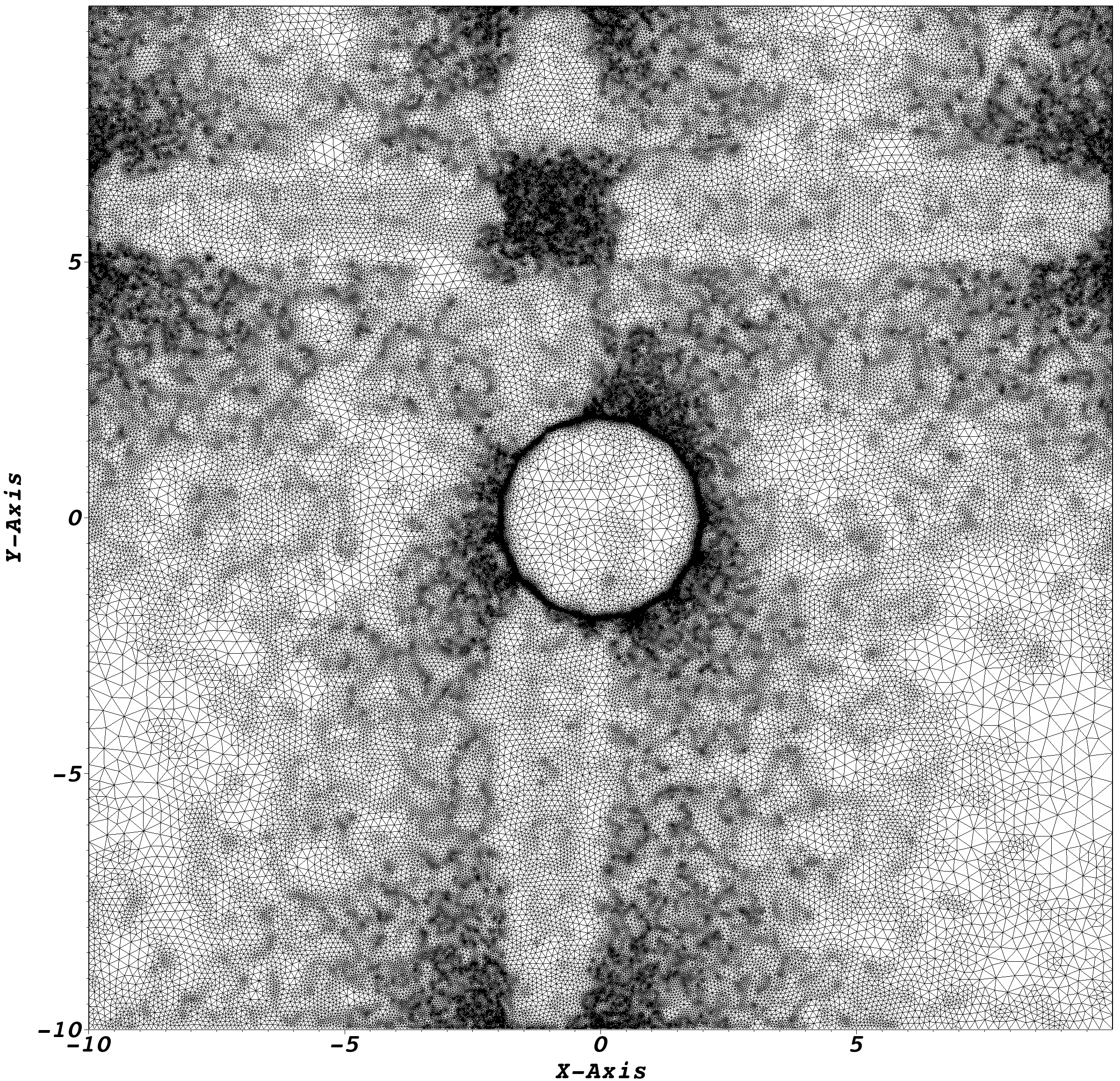}
        \caption{{\scriptsize Mesh with $563\,965$ elements, adapted using the error indicator $\|\sqrt{2}\eta\|_{L^2(O_i)}$ with gradient averaging for $y^*$. }}
        \label{run35_2D_Th_Adapted_With_Functional_Error_Indicator_Using_Gradient_Averaging}
    \end{minipage}%
    \begin{minipage}{0.5\textwidth}
        \centering
        \captionsetup{width=.9\linewidth}
        \includegraphics[width=1\linewidth]{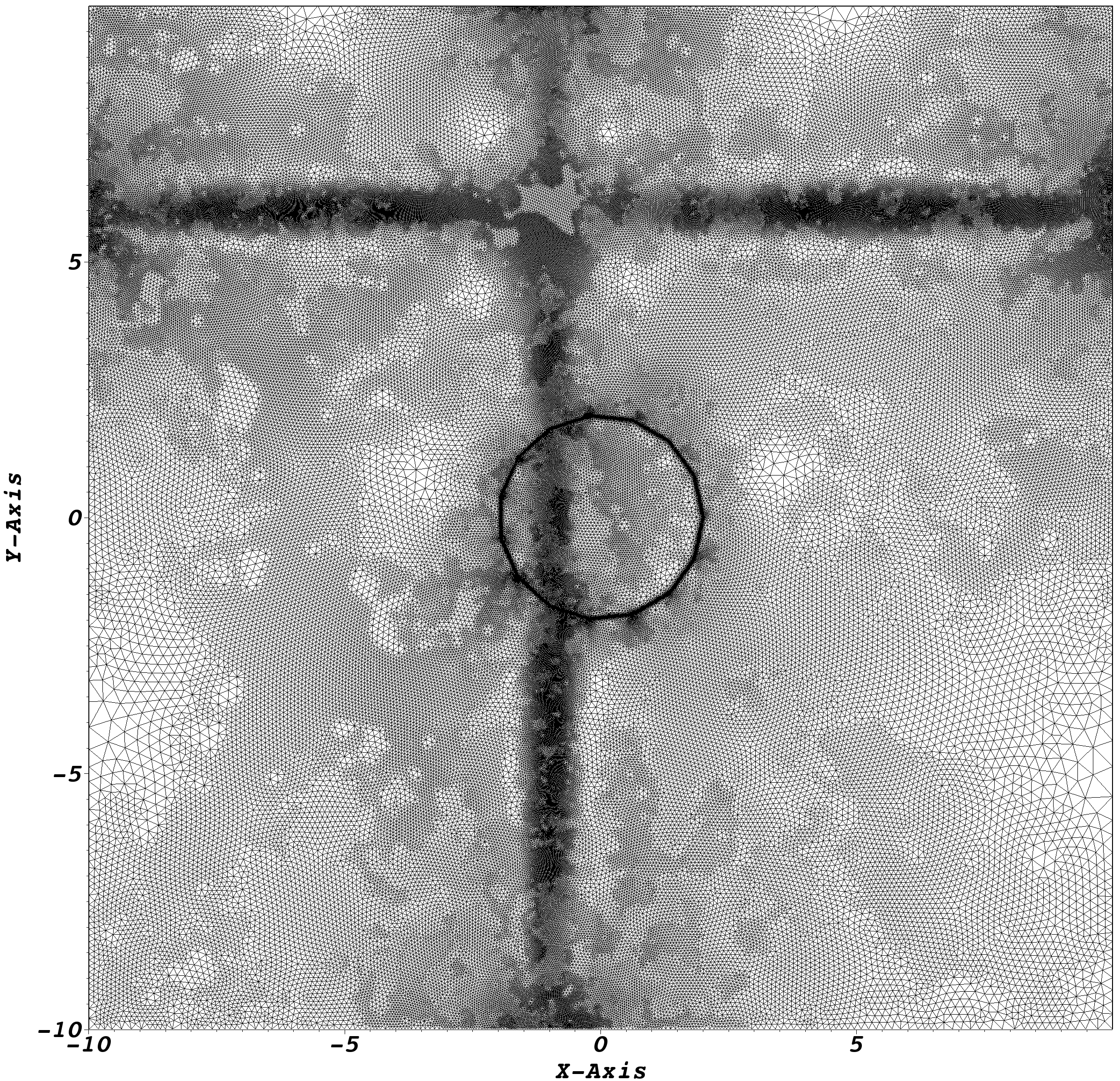}
        \caption{{\scriptsize Mesh with $444\,092$ elements, adapted using the error indicator $\vertiii{\epsilon\nabla v-y^*}_{*(O_i)}$ with gradient averaging for $y^*$.}}
        \label{run35_2D_Th_Adapted_With_EpsGradvMinusyStar_Error_Indicator_Using_Gradient_Averaging_for_yStar}
    \end{minipage}
\end{figure}
\subsection{Example 2 (2D)}
Figures~\ref{run26_2D_Th_Adapted_With_Functional_Error_Indicator0011} and \ref{run26_2D_Th_Adapted_With_EpsGradvMinusyStar_Error_Indicator0012} show the dependence of the meshes on the indicator for another example. Here, $\epsilon_1=1\, \epsilon_2=100$, $k_1=0.2,\,k_2=0.3$. The function $g=\exp\left(-b_1\left(\frac{|x-c_1|^2}{\sigma_1^2}-1\right)\right)-\exp\left(-b_2\left(\frac{|x-c_2|^2}{\sigma_2^2}-1\right)\right)$ and $l=\exp\left(-b_3\left(\frac{|x|^2}{\sigma_3^2}-1\right)\right)\sin\left(\frac{x_1x_2}{4}\right)$, where $b_1=2.2$, $b_2=2.5$, $b_3=6$, $c_1=(-1,0)$, $c_2=(5,5)$, $\sigma_1=\sigma_2=2$, $\sigma_3=10$.
The indicator $\vertiii{\epsilon\nabla v-y^*}_{*(O_i)}$ approximates well the elementwise error in combined energy norm but does not capture the rest of the error which is a result from the nonlinearity $k^2\sinh(u+w)$ and the right-hand side $l$ in \eqref{PBE_special_form_regular_nonlinear_part}. On the other hand, the term $D_F(v,-\Lambda^*y^*)$ controls the error $D_F(v,-\Lambda^*p^*)+D_F(u,-\Lambda^*y^*)$ and this is the reason why the mesh on Figure~\ref{run26_2D_Th_Adapted_With_Functional_Error_Indicator0011} resembles the wavy features of the function $f=-k^2\sinh(u+w)+l$.
 
\noindent 
The isolines of the reference solution and of the function $f$ are depicted on
Figures~\ref{Reference_Solution_50_Iso_Lines0000} and~\ref{Example_run26_2D_Function_f_200_Iso_Lines0000}.
\vspace{1ex} 
 
\begin{figure}[h!t]
    \centering
        \begin{minipage}[t]{0.516\textwidth}
        \captionsetup{width=0.9\linewidth}
        \includegraphics[width=1\linewidth,valign=t]{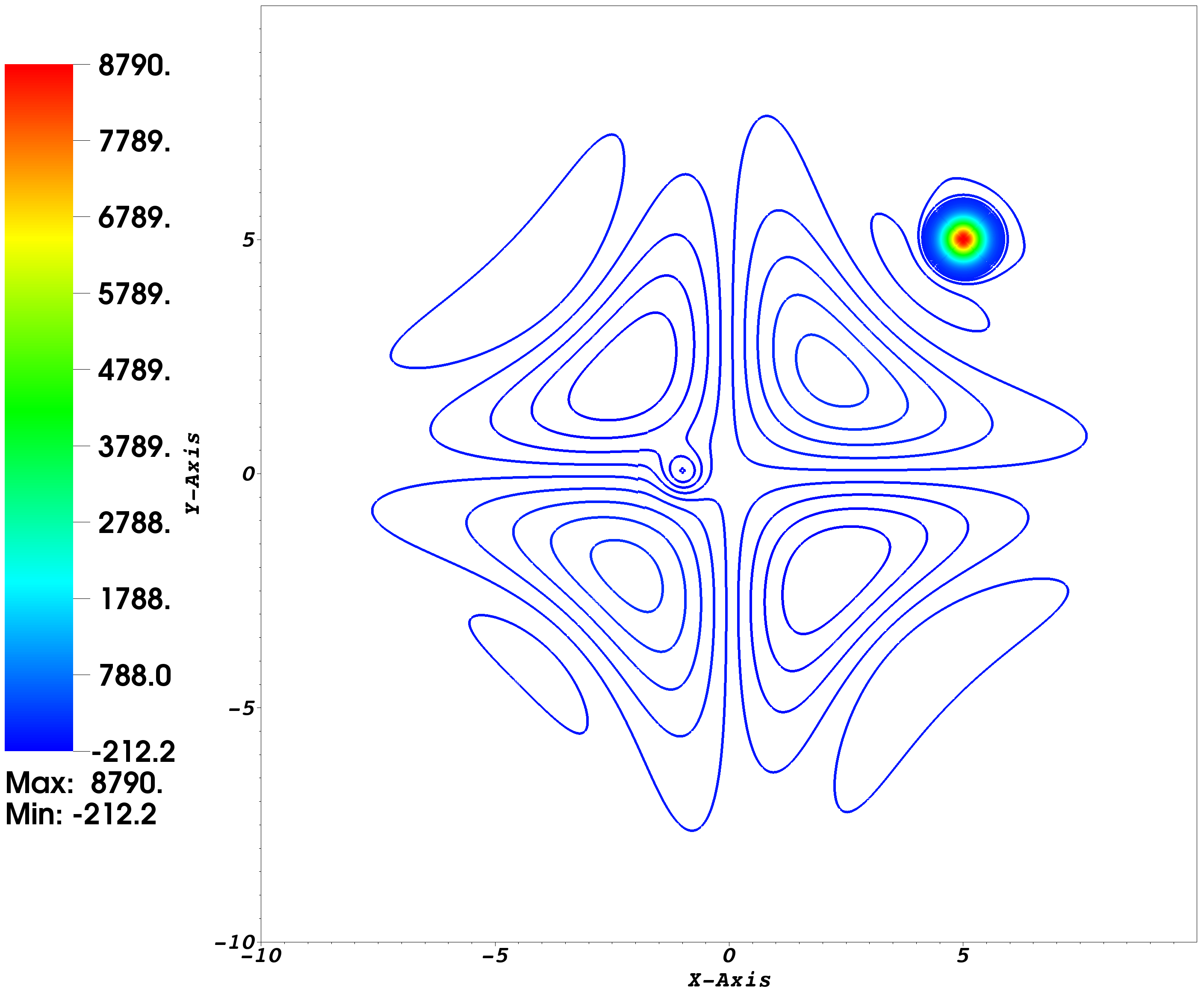}\vspace{-1ex}
        \caption{{\scriptsize Function $f=-k^2\sinh(u+w)+l$.}}
        \label{Example_run26_2D_Function_f_200_Iso_Lines0000}
    \end{minipage}
    \begin{minipage}[t]{0.434\textwidth}
        \captionsetup{width=.9\linewidth}
        \includegraphics[width=1\linewidth,valign=t]{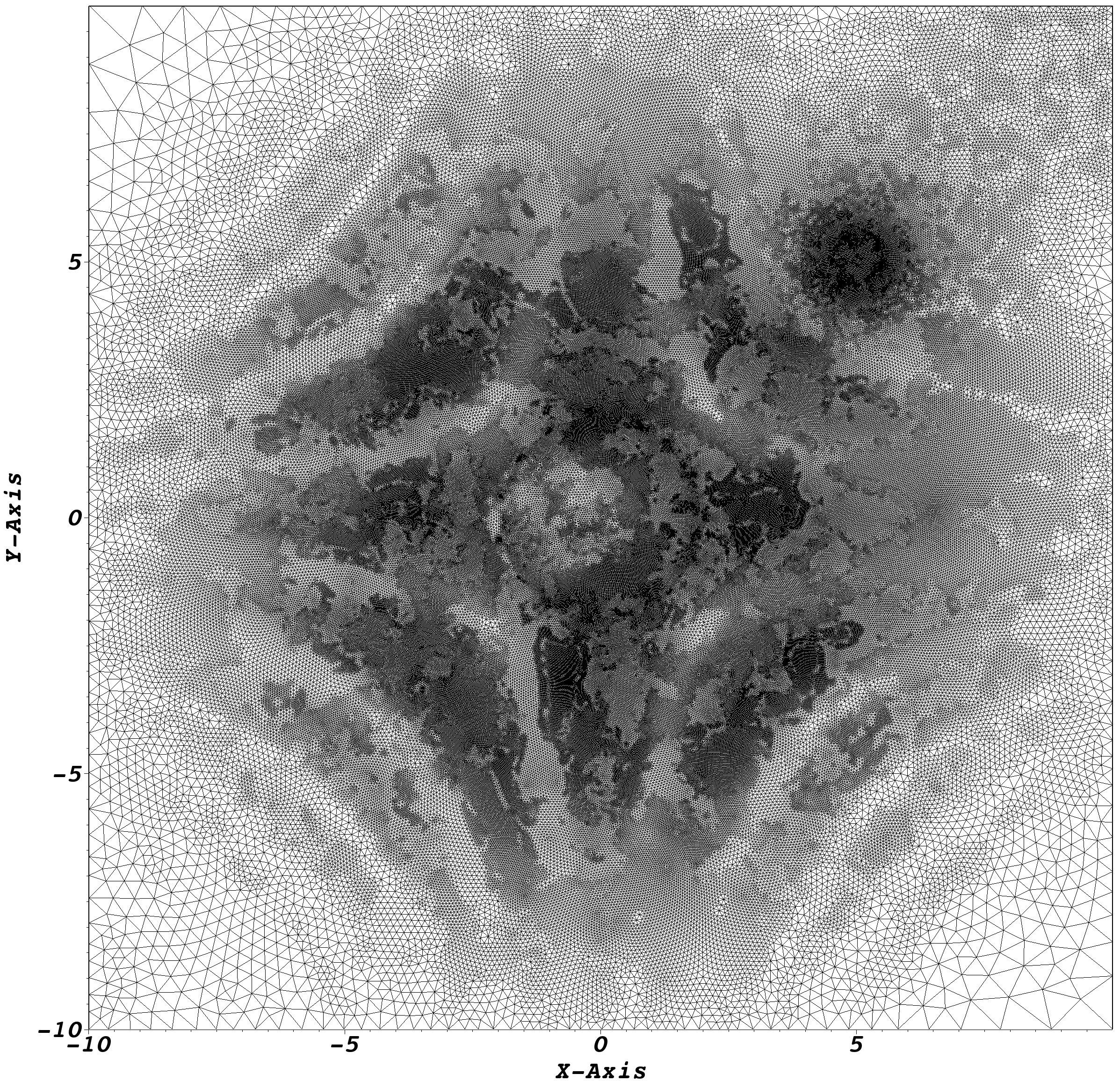}\vspace{-1ex}
        \caption{{\scriptsize Mesh with $395\,935$ elements, obtained by AMR using the error indicator $\|\sqrt{2}\eta\|_{L^2(O_i)}$ with
        flux equilibration for~$y^*$.}}
        \label{run26_2D_Th_Adapted_With_Functional_Error_Indicator0011}
    \end{minipage}\\[4ex]%
    \centering
    \begin{minipage}[t]{0.516\textwidth}
        \captionsetup{width=0.9\linewidth}
        \includegraphics[width=1\linewidth,valign=t]{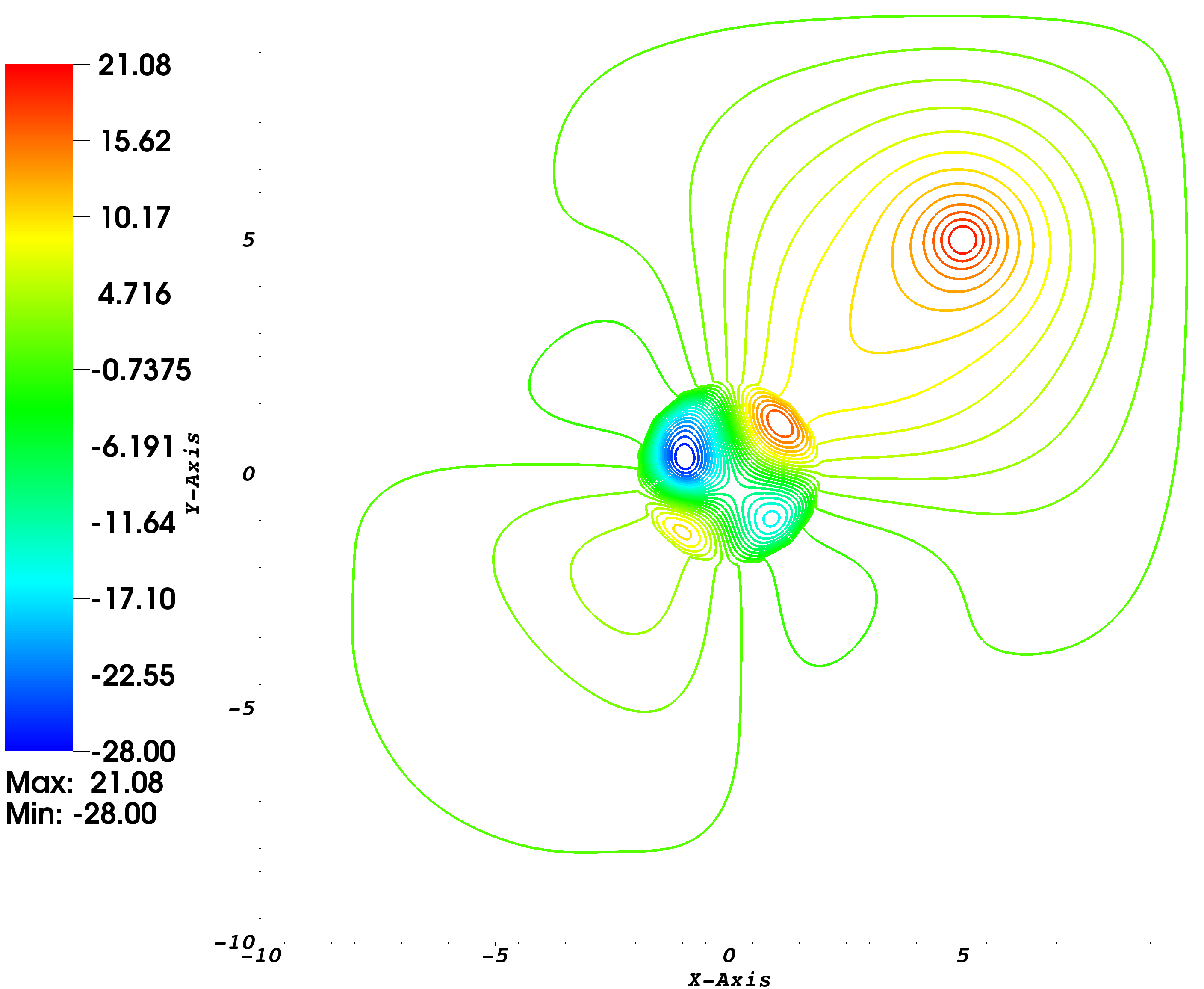}\vspace{-1ex}
        \caption{{\scriptsize Reference solution.}}
        \label{Reference_Solution_50_Iso_Lines0000}
    \end{minipage}%
    \begin{minipage}[t]{0.434\textwidth}
        \captionsetup{width=0.9\linewidth}
        \includegraphics[width=1\linewidth,valign=t]{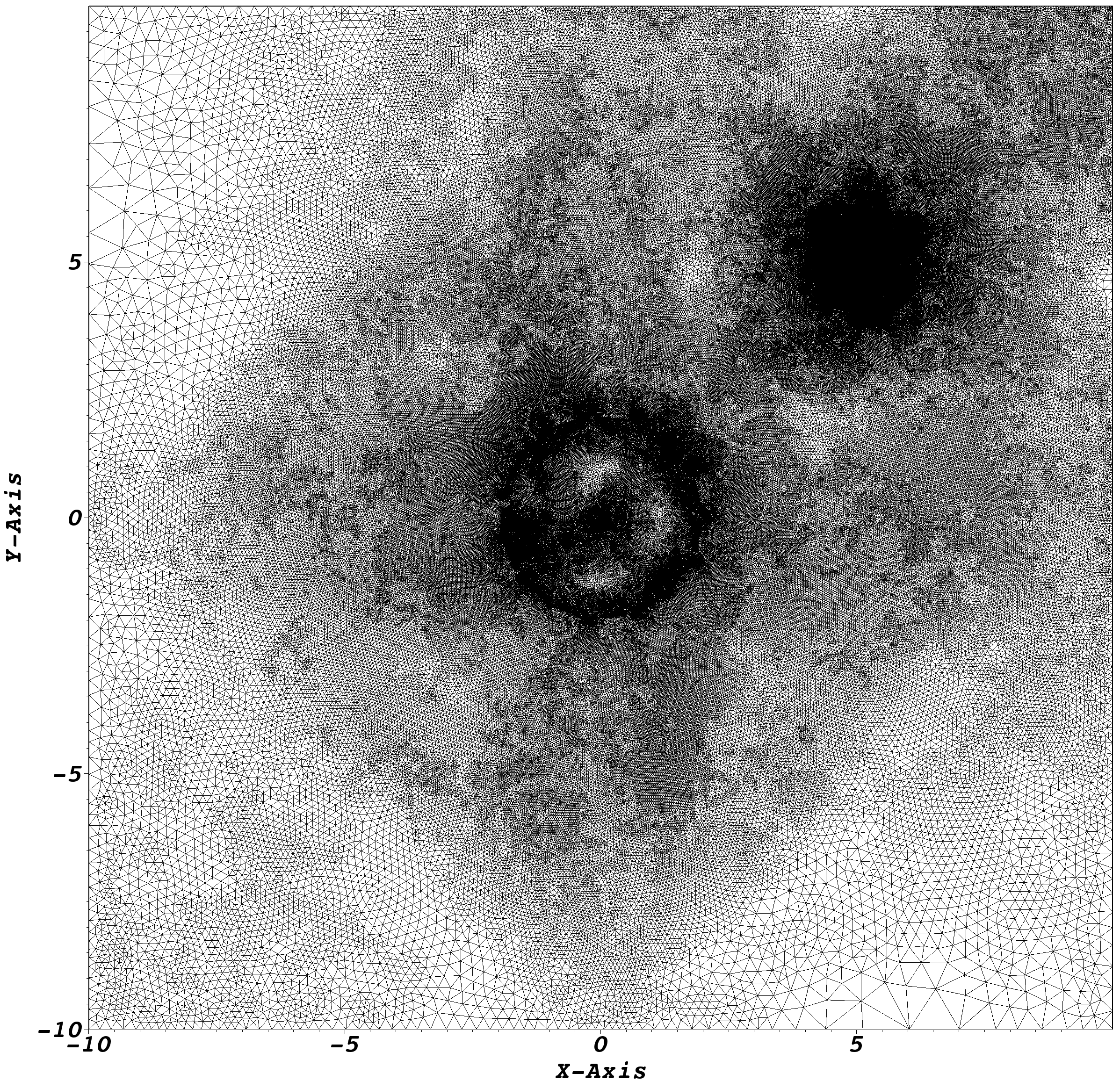}\vspace{-1ex}
        \caption{{\scriptsize Mesh with $555\,489$ elements, obtained by AMR using the error indicator $\vertiii{\epsilon\nabla v-y^*}_{*(O_i)}$ with
        flux equilibration for~$y^*$.}}
        \label{run26_2D_Th_Adapted_With_EpsGradvMinusyStar_Error_Indicator0012}
    \end{minipage}
\end{figure}
\vspace{1ex}

\subsection{Example 3 (3D)}

 In the third example, the computational domain $\Omega$ is a cube of side length $20$ Angstroms with
 a triangulated water molecule $\Omega_1$ in it. The diameter of the water molecule, which is positioned
 in the center of the cube, is about $2.75$ Angstroms. Its shape is not changed during the mesh adaptation
 process. The surface mesh of the water molecule is taken from~\cite{WaterMoleculeSurfaceMesh}. 
 Figure~\ref{Initial_Tetrahedral_Mesh_run17_3D} illustrates the initial tetrahedral mesh, which consists of
 $60\,222$ elements. It is generated using TetGen~\cite{TetGen} and adaptively refined with the help of
 mmg3d~\cite{mmg3d_dobrzynski}.
 Using the localized error indicator $\|\sqrt{2}\eta\|_{L^2(O_i)}$ computed on each vertex patch $O_i$ of the
 mesh, a new local mesh size at each vertex is defined by the formula 
 \begin{align*}
 h_i^{\text{new}}=h_i^{\text{old}}\left(\max\biggl\{\min\biggl\{\frac{\text{AM}\left\{\|\sqrt{2}\eta\|_{L^2(O_j)}\right\}}{\|\sqrt{2}\eta\|_{L^2(O_i)}},1 \biggr\}, 0.35 \biggr\}\right)
 \end{align*}
 and supplied to mmg3d, where $\text{AM}\left\{\|\sqrt{2}\eta\|_{L^2(O_j)}\right\}$ is the arithmetic mean of $\left\{\|\sqrt{2}\eta\|_{L^2(O_j)}\right\}$ over all vertex patches $O_j$. The coefficients $\epsilon$ and $k$ for this example are typical for electrostatic computations in biophysics using the PBE and are given by\\
\begin{minipage}{0.5\textwidth}
\begin{align*}
&\epsilon(x)=\left\{
\begin{aligned}
&\epsilon_1=2,\quad &x\in \Omega_1,\\
&\epsilon_2=80,\quad &x\in\Omega_2.
\end{aligned}
\right.
\end{align*}
\end{minipage}
\begin{minipage}{0.5\textwidth}
\begin{align*}
&k(x)=\left\{
\begin{aligned}
&k_1=0,\quad &x\in \Omega_1,\\
&k_2=0.84,\quad &x\in\Omega_2.
\end{aligned}
\right.
\end{align*}
\end{minipage}
Moreover, we assume that the problem is homogeneous, i.e., $l=0$, and
\begin{align*} 
&g=\exp\left(-b_1\left(\frac{|x-c_1|^2}{\sigma_1^2}-1\right)\right)-\exp\left(-b_2\left(\frac{|x-c_2|^2}{\sigma_2^2}-1\right)\right)\\
&+\exp\left(-b_3\left(\frac{|x-c_3|^2}{\sigma_3^2}-1\right)\right)+\exp\left(-b_4\left(\frac{|x-c_4|^2}{\sigma_4^2}-1\right)\right),
\end{align*}
where $b_1=b_2=b_3=b_4=2.3$, $c_1=(1,1,0)$, $c_2=(4,4,0)$,  $c_2=(0,6,0)$, $c_2=(-5,0,0)$,
 $\sigma_1=\sigma_2=\sigma_3=\sigma_4=2$.
 The reference solution $z_{h_{ref}}$ is computed on an adapted mesh with $79\,917\,007$ tetrahedrons.
 \begin{figure}[H]
    \centering
    \begin{minipage}[t]{0.45\textwidth}
        \captionsetup{width=0.9\linewidth}
        \includegraphics[width=1\linewidth,valign=t]{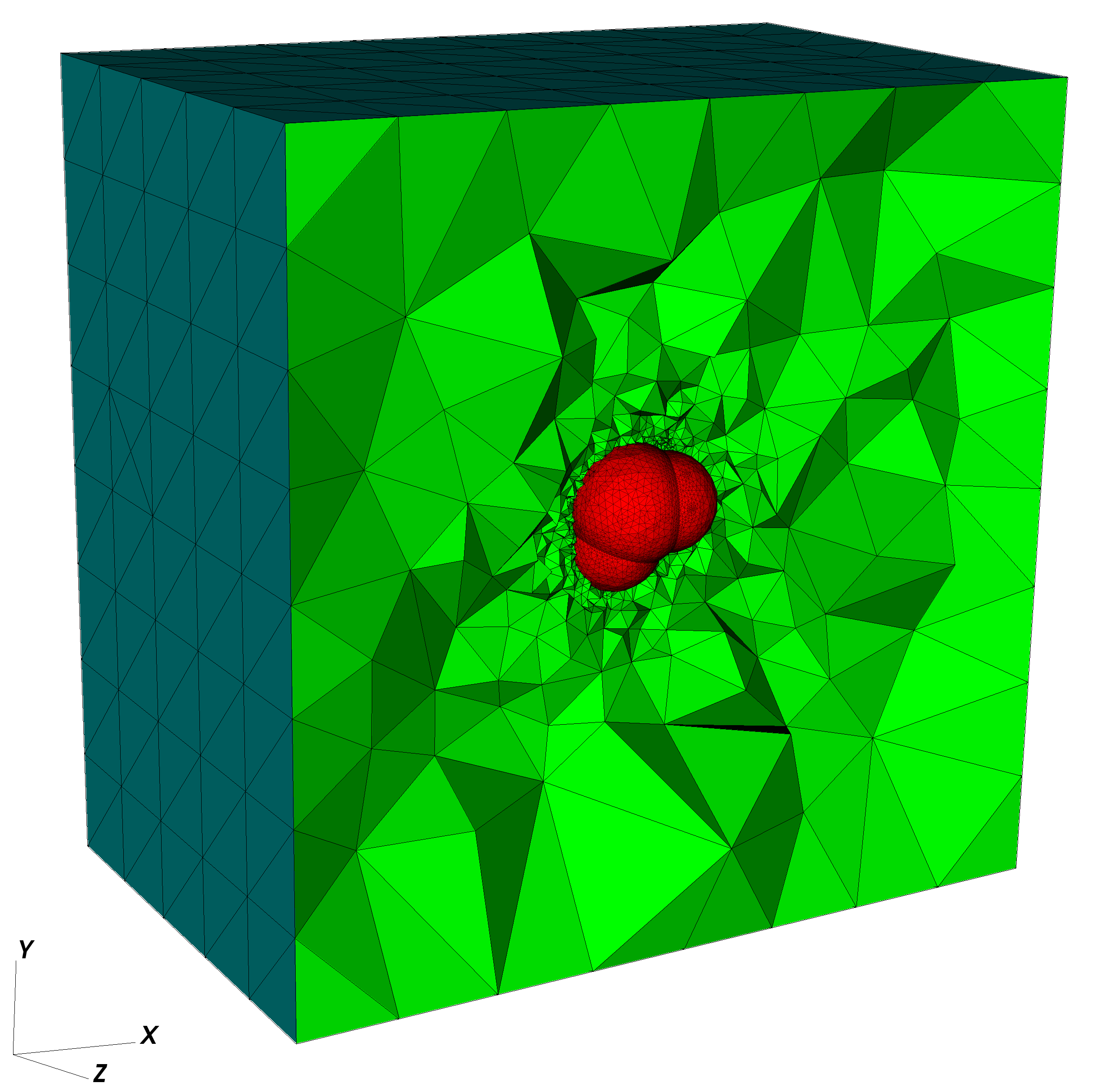}
        \caption{{\scriptsize Initial mesh in Example 3 consisting of $60\,222$ tetrahedrons.}}
        \label{Initial_Tetrahedral_Mesh_run17_3D}
    \end{minipage}%
    \begin{minipage}[t]{0.5\textwidth}
        \captionsetup{width=0.9\linewidth}
        \includegraphics[width=1\linewidth,valign=t]{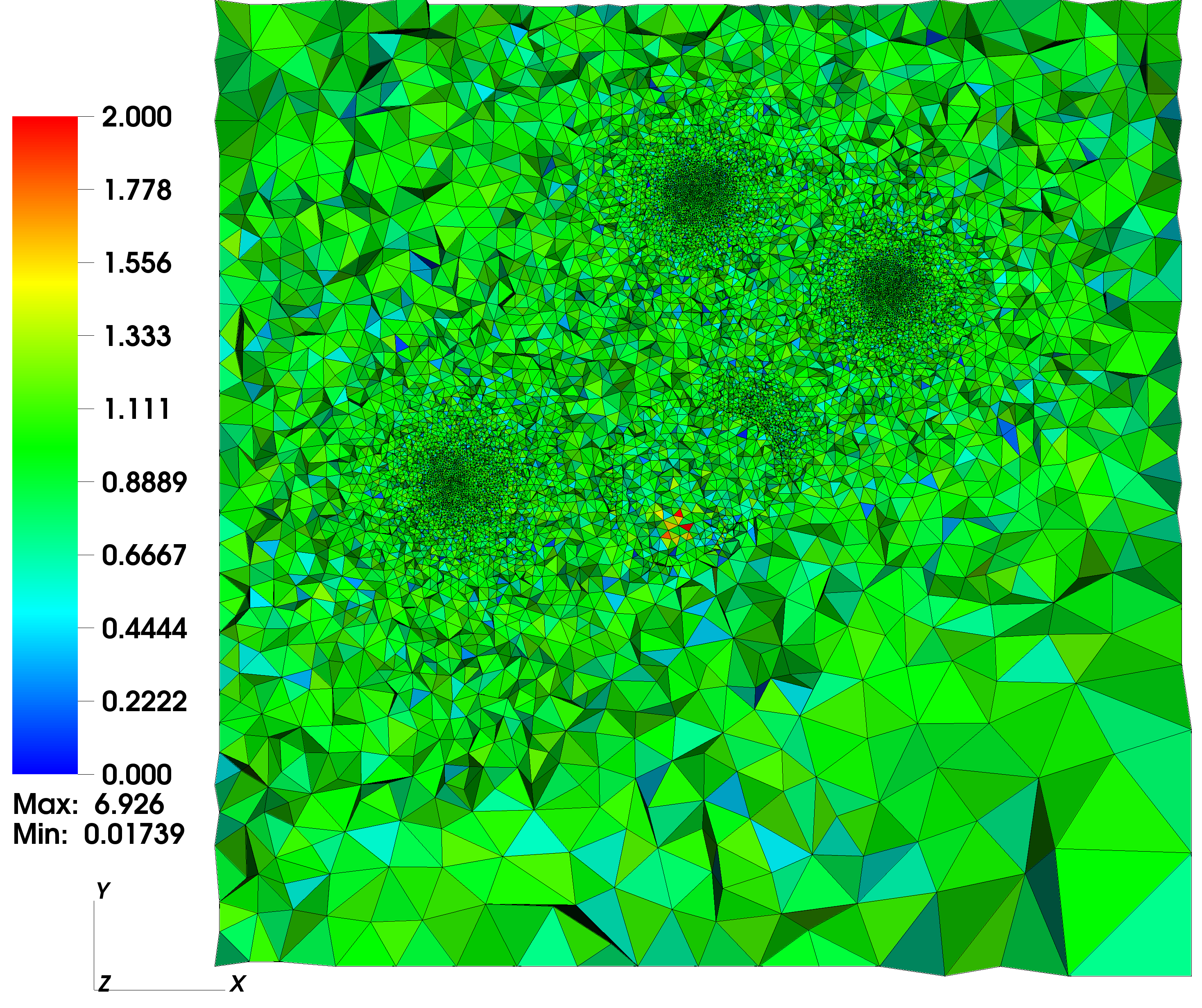}
        \caption{{\scriptsize Ratio of error indicator $\vertiii{\epsilon\nabla v-y^*}_*$
        and combined energy norm of the error, elementwise. Mesh on the $4$th level of AMR ($1.1736e+06$ elements)
        in Example 3 using the error indicator $\|\sqrt{2}\eta\|_{L^2(O_i)}$ with flux equilibration for~$y^*$.}}
        \label{run17_3D_Th_Adapted_With_Functional_Error_Indicator_Depicted_Ratio_EpsGradvMinusyStar_Error_Indicator_To_CEN_Error_Level}
    \end{minipage}
\end{figure}
 
 Since $l=0$ in $\Omega_1$ is a constant function, the patchwise reconstruction from \cite{Braess_Schoberl_2006}
 produces a flux $y^*$ with zero divergence in $\Omega_1$ and therefore the reliability of our majorant is guaranteed.
 In this example we achieve a very tight guaranteed bound on the error in combined energy norm, as well as in energy norm.
 The efficiency index $I_{\text{Eff}}^{\text{CEN,Up}}$ settles at around $1.05$ and the efficiency index $I_{\text{Eff}}^{\text{E,Up}}$
 decreases to $1.30$ (see Table~\ref{Example3_Table2}). This is in a good agreement with the fact that in this example the ratio
 $D_F(v,-\Lambda^*y^*)/M_\oplus^2(v,y^*)$ is well controlled and decreases to around $10\%$
 (see column 2 in Table~\ref{Example3_Table3}). We also note that in this example we obtained very similar results with the
 error indicator $\vertiii{\epsilon\nabla v-y^*}_{*(O_i)}$. For the efficiency index $I_{\text{Eff}}^{\text{CEN,Low}}$ of the lower
 bound on the combined energy norm of the error we obtain values converging to approximately $0.7071$ which is the
 approximate value of $\frac{\sqrt{2}}{2}$ (see column $3$ in Table~\ref{Example3_Table2}). This means that the combined
 energy norm of the error $\sqrt{\vertiii{\nabla(v-u)}^2+\vertiii{y^*-p^*}_*^2}$ is practically equal to $\vertiii{\epsilon\nabla v-y^*}_*$.
 
 Another consequence of this fact is the good accuracy of the practical estimation $P_{\text{rel}}^{\text{CEN}}$ of the relative error
 in combined energy norm (see columns $6$ and $7$ in Table~\ref{Example3_Table2}). The tight bounds on the error also enable
 us to compute tight and guaranteed upper bounds on the relative error in energy norm and combined energy norm as follows:
 \begin{subequations}\label{Sure_Upper_Bounds_For_Relative_Errors}
\begin{eqnarray}
\frac{\vertiii{\nabla(v-u)}}{\vertiii{\nabla u}}&\hspace{-3ex}&\leq\frac{\sqrt{2M_\oplus^2(v,y^*)}}{\vertiii{\nabla v}-\sqrt{2M_\oplus^2(v,y^*)}}
=:{\text{RE}}^{\text{Up}} \label{Sure_Upper_Bounds_For_Relative_Errors_1}\\
\frac{\sqrt{\vertiii{\nabla(v-u)}^2+\vertiii{y^*-p^*}_*^2}}{\sqrt{\vertiii{\nabla v}^2+\vertiii{y^*}_*^2}}&\hspace{-3ex}&\leq
\frac{\sqrt{2M_\oplus^2(v,y^*)}}{\sqrt{\vertiii{\nabla v}^2+\vertiii{y^*}_*^2}-\sqrt{2M_\oplus^2(v,y^*)}}=:
{\text{RCEN}}^{\text{Up}} \qquad \label{Sure_Upper_Bounds_For_Relative_Errors_2}
\end{eqnarray}
\end{subequations}
where \eqref{Sure_Upper_Bounds_For_Relative_Errors_1} is valid if $\vertiii{\nabla v}-\sqrt{2M_\oplus^2(v,y^*)}>0$
and \eqref{Sure_Upper_Bounds_For_Relative_Errors_2} is valid if $\sqrt{\vertiii{\nabla v}^2+\vertiii{y^*}_*^2}-\sqrt{2M_\oplus^2(v,y^*)}>0$.

 Similarly, we compute a tight and guaranteed lower bound for the relative error in combined energy norm by 
 \begin{align}\label{Sure_Lower_Bound_Combined_Energy_Norm_Error}
 \begin{aligned}
 {\text{RCEN}}^{\text{Low}}:=\frac{\frac{1}{\sqrt{2}}\vertiii{\epsilon\nabla v-y^*}_* }{\sqrt{\vertiii{\nabla v}^2+\vertiii{y^*}_*^2}+\sqrt{2M_\oplus^2(v,y^*)}}\leq \frac{\sqrt{\vertiii{\nabla(v-u)}^2+\vertiii{y^*-p^*}_*^2}}{\sqrt{\vertiii{\nabla v}^2+\vertiii{y^*}_*^2}}.
 \end{aligned}
 \end{align}
 \vspace{1ex}
 
{\small
\begin{center}
\captionof{table}{Example 3 (3D)} \label{Example3_Table1} 
\vspace{1ex}
\begin{tabular}{ |c|c|c|c|c|c|c| }
\hline
\multicolumn{7}{|c|}{Example 3 (3D): $k_1=0,\,k_2=0.84,\,\epsilon_1=2,\,\epsilon_2=80$} \\
\hline
\#elts &$\frac{\|v-u\|_0}{\|u\|_0}[\%]$ &$\frac{\vertiii{\nabla(v-u)}}{\vertiii{\nabla u}}[\%]$ & $\frac{\vertiii{y^*-p^*}_*}{\vertiii{p^*}_*}[\%]$ &
$2M_\oplus^2(v,y^*)$& $2M_\oplus^2(v,p^*)$ &$2M_\oplus^2(u,y^*)$\\
\hline
60222 		     		 &76.8320 	 	 &108.015 	&167.589			&425569		&117373 	     &308196\\
103236 		      &11.9257       &46.3306 	     &55.1210			&47104.5		&17845.0	&29259.5\\
222118 		      &1.09233       &17.7353 		&14.9578			&4484.44		&2224,69	&2259.75 \\
552936     	      &0.49820     &8.67222 		&7.09062			&965.067		&513.706		&451.361\\
1.1736e+06         &0.25609      &6.58075 	&5.33661 		     &539.734		&295.254		&244.481\\
2.05668e+06       &0.17094 	      &5.37625 		&4.18207			&350.648		&197.016 	&153.631\\
2.97315e+06       &0.12317     &4.73466  	&3.53852			&265.167		&152.783		&112.385\\
3.90692e+06       &0.10071     &4.32886 		&3.12966			&216.336		&127.703		&88.6336\\
\hline
\end{tabular}
\end{center}
}

{\small
\begin{center}
\captionof{table}{Example 3 (3D)} \label{Example3_Table2} 
\vspace{1ex}
\begin{tabular}{ |c|c|c|c|c| }
\hline
\multicolumn{5}{|c|}{Example 3 (3D): $k_1=0,\,k_2=0.84,\,\epsilon_1=2,\,\epsilon_2=80$} \\
\hline
\#elts & $\vertiii{\nabla(v-u)}^2$ &$\vertiii{y^*-p^*}_*^2$ & $2D_F(v,-\Lambda^*p^*)$ &$2D_F(u,-\Lambda^*y^*)$\\
\hline
60222 		   		&79487.0 	    &191346 			&37886.0		&116850		\\
103236 		     &14623.9      &20699.7 	     		&3221.12		&8559.78		\\
222118 		     &2142.92      &1524.28			&81.7757		&735.474		\\
552936     	     &512.376	    &342.528			&1.32980		&108.833		\\
1.1736e+06        &295.039      &194.026 	 		&0.21458 		&50.455		\\
2.05668e+06      &196.919 	    &119.155 			&0.09743		&34.4762		\\
2.97315e+06      &152.724	    &85.3044 			&0.05857		&27.0805		\\
3.90692e+06      &127.666 	    &66.7303			&0.03663  	&21.9033		\\
\hline
\end{tabular}
\end{center}
}

In Table~\ref{Example3_Table4} we show the computed
by~\eqref{Sure_Upper_Bounds_For_Relative_Errors} and \eqref{Sure_Lower_Bound_Combined_Energy_Norm_Error}
guaranteed bounds on the relative errors.

{\small
\begin{center}
\captionof{table}{Example 3 (3D)} \label{Example3_Table3} 
\vspace{1ex}
\begin{tabular}{ |c|c|c|c|c|c|c| }
\hline
\multicolumn{7}{|c|}{Example 3 (3D): $k_1=0,\,k_2=0.84,\,\epsilon_1=2,\,\epsilon_2=80$} \\
\hline
\#elts & $\frac{D_F(v,-\Lambda^*y^*)}{M_\oplus^2(v,y^*)}[\%]$ & $I_{\text{Eff}}^{\text{CEN,Low}}$ &
$I_{\text{Eff}}^{\text{CEN,Up}}$ & $I_{\text{Eff}}^{\text{E,Up}}$ & $P^{\text{CEN}}_{\text{rel}}~[\%]$ &
\parbox{2.4cm}{True rel. error in $\text{CEN}~[\%]$}\\
\hline
60222 		     		 &40.0541     &0.68627 		&1.25353			&2.31386		&92.8434	&140.985	\\
103236 		      &20.4500     &0.72828 	     &1.15478			&1.79473		&47.6870	&50.9159	\\
222118 		      &16.1172     &0.71615		&1.10583			&1.44661		&16.4040	&16.4054	\\
552936     	      &11.2249	    &0.70786 		&1.06248			&1.37241		&7.90966	&7.92099	\\
1.1736e+06         &9.33477     &0.70731		&1.05053		     &1.35254		&5.98505	&5.99106	\\
2.05668e+06       &9.82289	    &0.70725 		&1.05327			&1.33442		&4.81343	&4.81632	\\
2.97315e+06       &10.2057	    &0.70722 		&1.05547			&1.31767		&4.17784	&4.17960	\\
3.90692e+06       &10.1194 	    &0.70719		&1.05492			&1.30175		&3.77592	&3.77716	\\
\hline
\end{tabular}
\end{center}
}

{\small
\begin{center}
\captionof{table}{Example 3 (3D)} \label{Example3_Table4} 
\vspace{1ex}
\begin{tabular}{ |c|c|c|c|}
\hline
\multicolumn{4}{|c|}{Example 3 (3D): $k_1=0,\,k_2=0.84,\,\epsilon_1=2,\,\epsilon_2=80$} \\
\hline
\#elts & ${\text{RCEN}}^{\text{Low}}[\%]$ &${\text{RCEN}}^{\text{Up}}[\%]$ & ${\text{RE}}^{\text{Up}}[\%]$\\
\hline
60222 		     		 &26.8329     &2480.32 		&-					\\
103236 		      &20.9158     &98.4934 	     &310.049			\\
222118 		      &9.76945     &21.6027		&33.9219			\\
552936     	      &5.14869	    &9.14078 		&13.4714			\\
1.1736e+06         &3.97619     &6.69650 		&9.75647		     \\
2.05668e+06       &3.23651	    &5.33417 		&7.72193			\\
2.97315e+06       &2.82755	    &4.60886 		&6.64970			\\
3.90692e+06       &2.56630 	    &4.14555		&5.96873			\\
\hline
\end{tabular}
\end{center}
}
\vspace{2ex}

As a remark, we note that the efficiency indexes with respect to the energy and combined energy norms
of the error can be improved if we use a flux reconstruction in a bigger space, say $RT_1$, which has better
approximation properties. In this way the error in $\div y^*$ will decrease and as a result,
the term $D_F(v,-\Lambda^*y^*)$ and consequently the dual part of the error
$2M_\oplus^2(u,y^*)=\vertiii{y^*-p^*}_*^2+D_F(u,-\Lambda^*y^*)$ will constitute a smaller part of the whole
majorant and the error, respectively. Even better, we can minimize the majorant with respect to $y^*$
in a subspace of $H(\div;\Omega)$ like $RT_0$, possibly on another finner mesh. Note that in the
limit case we have
$\inf\limits_{y^*\in H(\div;\Omega)}M_\oplus^2(v,y^*)=M_\oplus^2(v,p^*)=\frac{1}{2}\vertiii{\nabla(v-u)}^2+D_F(v,-\Lambda^*p^*)$
and the dual error completely vanishes. In this case, 
\begin{align*}
I_{\text{Eff}}^{\text{CEN,Up}}=I_{\text{Eff}}^{\text{E}}=\frac{\sqrt{2M_\oplus^2(v,p^*)}}{\vertiii{\nabla(v-u)}}=\frac{\sqrt{\vertiii{\nabla(v-u)}^2+2D_F(v,-\Lambda^*p^*)}}{\vertiii{\nabla(v-u)}}
\end{align*}
where the last ratio tends to $1$ because by
\eqref{Upper_Quadratic_Bound_for_DFvMinusLambdaStarpStar} the term
$D_F(v,-\Lambda^*p^*)\sim \|v-u\|_{L^2(\Omega)}^2$
and has a higher order of convergence than $\vertiii{\nabla(v-u)}^2$.
In practice, we can minimize the majorant with respect to $y^*$ only once on
a sufficiently big subspace of $H(\div;\Omega)$ to find some good approximation
$\overline{y}^*$ of $p^*$ and then reuse this $\overline{y}^*$ and obtain guaranteed
and very tight bounds on the error in energy and combined energy norm.
To illustrate these ideas, for the first example we recomputed the value of the
majorant $M_\oplus^2(v,y^*)$ on all mesh levels (sequence of meshes is the
same one from Tables \ref{Example1_Table1}, \ref{Example1_Table2}, \ref{Example1_Table3})
using the flux $\overline y^*$ that we obtained through the patchwise  reconstruction with equilibration
at the last level, $11$, where the mesh consists of $386\,185$ elements.
This $\overline y^*$ gives a very good approximation to the exact flux $p^*$ and thus the error
in $\div \overline y^*$ at all adaptation levels before level $11$ is much smaller relative to the error
measured in energy or combined energy norm.
As a consequence, the ratio $D_F(v,-\Lambda^*\overline y^*)/M_\oplus^2(v,\overline y^*)$ is small
and increases from around $4\%$ to its final value of $73 \%$ at level $11$. The respective efficiency
indexes with respect to the energy and combined energy norms are given in Table \ref{Example1_Table7}.
This time, the majorant $M_\oplus^2(v,\overline y^*)$ gives a much tighter bound on the error in energy
and combined energy norm and the efficiency indexes increase from around $1$ to their final values at
level $11$ of $3.3889$ and $1.9206$, respectively. 
{\small
\begin{center}
\captionof{table}{Example 1 (2D)} \label{Example1_Table7} 
\vspace{1ex}
\begin{tabular}{ |c|c|c|c|c|c|c| }
\hline
\multicolumn{7}{|c|}{Example 1 (2D): $k_1=0.15,\,k_2=0.4,\,\epsilon_1=1,\,\epsilon_2=100$} \\
\hline
\# elements & $\frac{D_F(v,-\Lambda^*y^*)}{M_\oplus^2(v,y^*)}[\%]$ & $I_{\text{Eff}}^{\text{CEN,Low}}$ &
$I_{\text{Eff}}^{\text{CEN,Up}}$ & $I_{\text{Eff}}^{\text{E,Up}}$ & $P^{\text{CEN}}_{\text{rel}}~[\%]$ &
\parbox{2.4cm}{True rel. error in $\text{CEN}~[\%]$}\\
\hline
196 		   &15.8135  	 &0.70520 	&1.08700			&1.08740		&38.5074	&36.4137	\\
347 		   &3.61970     &0.70640 	&1.01760			&1.01870		&22.1410	&21.8386	\\
630 		   &3.24520     &0.70650		&1.01570			&1.01800		&15.5098	&15.4285	\\
1315     	   &2.99700 	 &0.70980 	&1.01930			&1.02350		&11.3338	&11.2565	\\
2865            &5.11630     &0.71080 	&1.03190		     &1.03970		&8.41663	&8.36086	\\
5938            &9.91240	 &0.71310 	&1.06250			&1.08000		&5.75210	&5.69982	\\
12006          &19.9535	 &0.70580 	&1.11560			&1.15160		&4.11607	&4.12246	\\
24571          &35.1659 	 &0.69030		&1.21230			&1.29130		&2.89890	&2.96931	\\
48483          &45.0879     &0.70940 	&1.35380			&1.52340		&2.23724	&2.23000	\\
97423          &59.5529 	 &0.69360		&1.54240			&1.91030		&1.69993	&1.73298	\\
192905        &68.6293	 &0.69130 	&1.74560			&2.51110		&1.39059	&1.42237	\\
386185        &73.0132  	 &0.70550 	&1.92060			&3.38890		&1.23821	&1.24105	\\
\hline
\end{tabular}
\end{center}
}    
\vspace{2ex}
        
\section{Conclusions}

We proved the existence and uniqueness of a solution $u$ of the nonlinear elliptic
Problem~\eqref{PBE_special_form_regular_nonlinear_part}, which appears in context
of solving the nonlinear PBE numerically by two- or three-term regularization.
We further proved an $L^\infty(\Omega)$ a priori bound on the (regular component of the)
solution $u$ (of the PBE), established an analogue of Cea's lemma,
cf.~\eqref{quasi_optimal_a_priori_error_estimate_1}, and used it to prove unqualified
convergence of the $P_1$ Lagrange FEM under uniform mesh refinement.

As a main result we derived the identity~\eqref{Explicit_form_of_Functional_error_equality} by finding the
explicit form of the terms in the abstract relations~\eqref{Func_a_posteriori_estimate_with_forcing_functionals}
and~\eqref{Functional_error_equality}. It defines a natural error measure for the considered class of problems
and is the basis for fully computable guaranteed tight bounds on the global errors (see Table~\ref{Example3_Table4}).

A big advantage of our approach is that it can be used for any conformal approximation ($P_1$, $P_2$, IGA,...)
and that there are no local or global constants present in the estimates for the error in energy and combined enery
norm (CEN). Demonstrated by our theoretical findings as well as by the presented numerical tests, good efficiency
indexes/tight bounds on the errors, require a flux reconstruction with equilibration.
The key factor that determines the efficiency index is the ratio $\frac{D_F(v,-\Lambda^*y^*)}{M_\oplus^2(v,y^*)}$.

Assuming that
$$D_F(v,-\Lambda^*y^*)\approx D_F(v,-\Lambda^* p^*)+D_F( u,-\Lambda^*y^*),$$
which means that the last term in~\eqref{relation_DFs_and_integral_term_in_Prager_Synge_extended_equality_2}
is close to zero, we obtain from~\eqref{Explicit_form_of_Functional_error_equality} the estimate
$$
I_{\text{Eff}}^{\text{CEN,Up}} \approx \frac{1}{\sqrt{1-\frac{D_F(v,-\Lambda^*y^*)}{M_\oplus^2(v,y^*)}}}.
$$
From what we observed, the efficiency index $I_{\text{Eff}}^{\text{E,Up}}$ with respect to the energy norm usually is
no more than twice bigger than $I_{\text{Eff}}^{\text{CEN,Up}}$ (assuming we have a good approximation $y^*$ to $p^*$). 
Therefore, if during the computations we detect that this ratio is 
increasing
we can apply the so-called estimation with one step delay, i.e compute the value of the majorant $M_\oplus^2(v,y^*)$
for the current mesh level with the reconstructed $y^*$ from the next level.
Another strategy is to find somehow a good approximation $\bar y^*$ of $p^*$ and reuse it on several AMR levels
(for example by means of solving the dual problem $P^*$--maximizing $I^*$ on possibly another mesh).
We also conclude that gradient averaging is not appropriate for obtaining good efficiency indexes and that it tends
to overrefine the mesh around the interface.

\bibliographystyle{plain}
\bibliography{references}

\begin{thebibliography}{10}

\bibitem{WaterMoleculeSurfaceMesh}
A collection of molecular surface meshes.
\newblock
  \url{https://www.rocq.inria.fr/gamma/gamma/download/affichage.php?dir=MOLECULE&name=water_mol&last_page=6}.
\newblock Accessed: 2017-08-18.

\bibitem{Best_constants_in_some_exponential_Sobolev_inequalities}
{B. Kawohl, M. Lucia}.
\newblock Best constants in some exponential {S}obolev inequalities.
\newblock {\em Indiana University Mathematics Journal}, 57(4):1907--1928, 2008.

\bibitem{CC_Praetorius_2014}
M.~Page D.~Praetorius C.~Carstensen, M.~Feischl.
\newblock Axioms of adaptivity.
\newblock {\em Comput. Math. Appl.}, 67(6):1195--1253, 2014.

\bibitem{VisIt}
Hank Childs, Eric Brugger, Brad Whitlock, Jeremy Meredith, Sean Ahern, David
  Pugmire, Kathleen Biagas, Mark Miller, Cyrus Harrison, Gunther~H. Weber, Hari
  Krishnan, Thomas Fogal, Allen Sanderson, Christoph Garth, E.~Wes Bethel,
  David Camp, Oliver R\"{u}bel, Marc Durant, Jean~M. Favre, and Paul
  Navr\'{a}til.
\newblock {VisIt: An End-User Tool For Visualizing and Analyzing Very Large
  Data}.
\newblock In {\em {High Performance Visualization--Enabling Extreme-Scale
  Scientific Insight}}, pages 357--372. Oct 2012.

\bibitem{Braess_Schoberl_2006}
{D. Braess, J. Sch{\"o}berl}.
\newblock Equilibrated residual error estimator for {M}axwell's equations.
\newblock {\em RICAM report}, 2006.

\bibitem{Kinderlehrer_Stampacchia}
{D. Kinderlehrer, G. Stampacchia}.
\newblock {\em An Introduction to Variational Inequalities and Their
  Applications}.
\newblock SIAM, 2000.

\bibitem{mmg3d_dobrzynski}
Cecile Dobrzynski.
\newblock {MMG3D: User Guide}.
\newblock Technical Report RT-0422, {INRIA}, March 2012.

\bibitem{Fogolari_Brigo_Molinari_2002}
{F. Fogolari, A. Brigo, H. Molinari}.
\newblock The {P}oisson-{B}oltzmann equation for biomolecular electrostatics: a
  tool for structural biology.
\newblock {\em J. Mol. Recognit.}, 15:377--392, 2002.

\bibitem{Fogolari_Zuccato_Esposito_Viglino_1999}
{F. Fogolari, P. Zuccato, G. Esposito, P. Viglino}.
\newblock Biomolecular electrostatics with the linearized {P}oisson-{B}oltzmann
  equation.
\newblock {\em Biophysical Journal}, 76:1--16, 1999.

\bibitem{Stampacchia_1965}
{G. Stampacchia}.
\newblock Le probl{\`e}me de {D}irichlet pour les {\'e}quations elliptiques du
  second ordre {\`a} coefficients discontinus.
\newblock {\em Annales de l'institut Fourier}, 15(1):189--257, 1965.

\bibitem{Brezis_FA}
{H. Br{\'e}zis}.
\newblock {\em Functional Analysis, {S}obolev Spaces and Partial Differential
  Equations}.
\newblock Springer, 2011.

\bibitem{Brezis_Browder_1978_One_Property_of_Sobolev_Spaces}
{H. Br{\'e}zis, F. Browder}.
\newblock Sur une propri{\'e}t{\'e} des espaces de {S}obolev.
\newblock {\em C. R. Acad. Sc. Paris}, 287:113--115, 1978.

\bibitem{Oberoi_Allewell_1993}
{H. Oberoi, N. M. Allewell}.
\newblock Multigrid solution of the nonlinear {P}oisson-{B}oltzmann equation
  and calculation of titration curves.
\newblock {\em Biophysical Journal}, 65:48--55, 1993.

\bibitem{FreeFem}
F.~Hecht.
\newblock New development in {F}ree{F}em++.
\newblock {\em J. Numer. Math.}, 20(3-4):251--265, 2012.

\bibitem{Ekeland_Temam}
{I. Ekeland, R. Temam}.
\newblock {\em Convex Analysis and Variational Problems}.
\newblock North-Holland Publishing Company, 1976.

\bibitem{Knapp_Schoberl_2014}
{I. Sakalli, J. Sch{\"o}berl, E. W. Knapp}.
\newblock {mFES: A Robust Molecular Finite Element Solver for Electrostatic
  Energy Computations}.
\newblock {\em J. Chem. Theory Comput.}, 10:5095--5112, 2014.

\bibitem{Sharp_Honig_1990}
{K. A. Sharp, B. Honig}.
\newblock Calculating total electrostatic energies with the nonlinear
  {P}oisson-{B}oltzmann equation.
\newblock {\em J. Phys. Chem}, 94:7684--7692, 1990.

\bibitem{Chen2006b}
{Long Chen, Michael J. Holst, Jinchao Xu}.
\newblock Adaptive finite element modeling techniques for the
  {P}oisson-{B}oltzmann equation.
\newblock {\em Siam J. Numer. Anal.}, 45(6):2298--2320, 2007.

\bibitem{Praetorius_Zee_2016}
K.~G. van der~Zee M.~Feischl, D.~Praetorius.
\newblock An abstract analysis of optimal goal-oriented adaptivity.
\newblock {\em SIAM J. Numer. Anal.}, 54(3):1423--1448, 2016.

\bibitem{Holst2012}
{M. Holst, J.A. McCammon, Z. Yu, Y. C. Zhou, Y. Zhu}.
\newblock Adaptive finite element modeling techniques for the
  {P}oisson-{B}oltzmann equation.
\newblock {\em Commun. Comput. Phys.}, 11:179--214, 2012.

\bibitem{Repin}
{P. Neittaanmaki, S. Repin}.
\newblock {\em Reliable Methods for Computer Simulation: Error Control and
  Posteriori Estimates}.
\newblock Elsevier, 2004.

\bibitem{Showalter}
{R. E. Showalter}.
\newblock {\em {H}ilbert Space Methods for Partial Differential Equations}.
\newblock Courier Corporation, 2010.

\bibitem{Repin_noninear_var_problems_2012}
S.~Repin.
\newblock On measures of errors for nonlinear variational problems.
\newblock {\em Russian J. Numer. Anal. Math. Modelling}, 27(6):577--584, 2012.

\bibitem{Kesavan}
{S. Kesavan}.
\newblock {\em Topics in Functional Analysis and Applications}.
\newblock New Age International (P) Limited, 1989.

\bibitem{Repin_2000}
{S. Repin}.
\newblock A posteriori error estimation for variational problems with uniformly
  convex functionals.
\newblock {\em Math. Comp}, 69:481--500, 2000.

\bibitem{Repin_Valdman_2017}
J.~Valdman S.~Repin.
\newblock Error identities for variational problems with obstacles.
\newblock {\em Z. Angew. Math. Mech.}, pages 1--24, 2017.

\bibitem{TetGen}
H.~Si.
\newblock {T}et{G}en, a {D}elaunay-based quality tetrahedral mesh generator.
\newblock {\em ACM Transactions on Mathematical Software (TOMS)}, 41(11), 2015.

\end{thebibliography}
\end{document}